\documentclass[10pt]{amsart}
\usepackage[english]{babel}
\usepackage{amssymb}
\usepackage{hyperref}
\usepackage[initials]{amsrefs}
\usepackage{tikz-cd}
\usepackage{graphicx}
\usepackage{amssymb, mathrsfs, enumerate}
\usepackage{xcolor}
\usepackage{comment}

\newcommand{\bbA}{{\mathbb A}}
\newcommand{\bbN}{{\mathbb N}}
\newcommand{\bbQ}{{\mathbb Q}}
\newcommand{\bbR}{{\mathbb R}}

\newcommand{\bbZ}{{\mathbb Z}}
\newcommand{\bbC}{{\mathbb C}}

\newcommand{\bfG}{{\mathbf G}}
\newcommand{\bfH}{{\mathbf H}}

\newcommand{\calS}{\mathcal{S}}

\newcommand{\calO}{\mathcal{O}}

\newcommand{\bs}{\backslash}

\newcommand{\colim}{\operatorname*{colim}}
\newcommand{\defq}{\mathrel{\mathop:}=}
\newcommand{\id}{\operatorname{id}}
\newcommand{\im}{\operatorname{im}}

\newcommand{\res}{\operatorname{res}}

\newcommand{\cohom}{\operatorname{cohom}}

\newcommand{\pr}{\operatorname{pr}}

\newcommand{\vol}{\operatorname{vol}}

\newcommand{\filling}{\operatorname{FV}}
\newcommand{\combfilling}{\operatorname{cFV}}
\newcommand{\combweightedfilling}{\operatorname{cwFV}}

\newcommand{\Ker}{\operatorname{Ker}}

\newcommand{\Hom}{\operatorname{hom}}

\newcommand{\SL}{\operatorname{SL}}

\newcommand{\SO}{\operatorname{SO}}
\newcommand{\SU}{\operatorname{SU}}
\newcommand{\Sp}{\operatorname{Sp}}

\newcommand{\rank}{\operatorname{rank}}
\newcommand{\frank}{\operatorname{rank}_\pol}

\newcommand{\aut}{\operatorname{aut}}
\newcommand{\Ind}{\operatorname{Ind}}

\newcommand{\St}{\operatorname{St}}
\newcommand{\FP}{\operatorname{FP}}

\newcommand{\mass}{\operatorname{mass}}
\newcommand{\loc}{\operatorname{loc}}

\newcommand{\pol}{\ensuremath{\mathrm{pol}}}
\newcommand{\cell}{\ensuremath{\mathrm{cell}}}
\newcommand{\lip}{\ensuremath{\mathrm{lip}}}

\newcommand{\norm}[1]{{\left\lVert #1\right\rVert}}

\newtheorem{mthm}{Theorem}

\newtheorem{theorem}{Theorem}[section]
\newtheorem{lemma}[theorem]{Lemma}

\newtheorem{cor}[theorem]{Corollary}

\newtheorem{prop}[theorem]{Proposition}

\theoremstyle{definition}
\newtheorem{definition}[theorem]{Definition}
\newtheorem{defn}[theorem]{Definition}

\newtheorem{remark}[theorem]{Remark}
\newtheorem{conjecture}[theorem]{Conjecture}

\begin{document}

\title[Unitary cohomology of arithmetic groups]{Higher Kazhdan property and unitary cohomology of arithmetic groups}

\author{Uri Bader}
\address{Weizmann Institute, Israel}
\email{uri.bader@gmail.com}

\author{Roman Sauer}
\address{Karlsruhe Institute of Technology, Germany}
\email{roman.sauer@kit.edu}


\begin{abstract}
Notions of a higher Kazhdan property can be defined in terms of the vanishing of unitary group cohomology in higher degrees. Garland's theorem for simple groups over non-archimedean fields provides the first examples of a higher Kazhdan property. We prove a version of Garland's theorem for simple Lie groups and their lattices. 
A novelty of our approach is the use of methods from geometric group theory and -- in the rank~$1$ case -- from (global) representation theory pertaining to the spectral gap property.
\end{abstract}

\maketitle


\section{Introduction}

In 1967 Kazhdan introduced a property of groups, known as Kazhdan's property T, which is defined in terms of unitary representations. This property found intriguing and surprising applications in various branches of mathematics (ergodic theory, geometric group theory, algebraic groups, computer science). 

A theorem of Delorme-Guichardet characterizes property T in terms of the first group cohomology with unitary coefficients. Equivalent criteria for property~T are the vanishing of the first cohomology for all unitary representations, the vanishing of the first cohomology for all unitary representations not containing the trivial one and the first cohomology always being Hausdorff. 

The theorem of Delorme-Guichardet suggests notions of higher property~T by considering higher-degree group cohomology. The analogs of the equivalent criteria in degree~$1$ turn out to be not equivalent in higher degrees. 
A definition of higher property~T, which differs from the one below and generalizes the Hausdorff criterion, was given by Bader-Nowak~\cites{BN2}. 
The definition by Chiffre-Glebsky-Lubotzky-Thom\cite{stability}*{Definition 4.1} corresponds to property~$[T_n]$ defined next. 

\begin{defn} \label{def:T_n}
  A locally compact second countable group $G$ has property $[T_n]$ if its continuous cohomology $H^j_c(G, V)$ vanishes for every unitary $G$-representation $V$ and every $1\le j\le n$. It has property $(T_n)$ if the same is true provided that $V$ has no nontrivial $G$-invariant vectors.
\end{defn}
Property $(T_1)$ is equivalent to property $[T_1]$, and both properties are equivalent to Kazhdan's property T~(Lemma~\ref{lem:T_1}).
On the other hand, for every $n\geq 2$, there are groups satisfying property $(T_n)$ but not property $[T_n]$ (e.g. the group $\Sp_{2n+2}(\mathbb{Z})$, as follows from Theorem~\ref{thm:semisimpleT} below).
Higher property~T recently gained interest due to its connection to higher-dimensional expanders and group stability. For example, a finitely presented group with property $[T_2]$ is Frobenius stable~\cite{stability}*{Theorem~1.2}. More applications of $(T_2)$ instead of $[T_2]$ to Frobenius stability can be found in~\cite{stability_four}. 

The first result about higher property~T is by Garland~\cite{Garland}. We state the case of $\SL_n$. 

\begin{theorem}[Garland]
The group $\SL_n(\bbQ_p)$ and any of its lattices have property~$[T_{n-2}]$.  
\end{theorem}

The primary objective of this paper is to generalize Garland's result to the case of real simple Lie groups. 

\begin{mthm}\label{thm: higher property T main result} The group $\SL_n(\bbR)$ and any of its lattices have property~$(T_{n-2})$. 
More generally, if $G$ is a connected almost simple Lie group of real rank~$d$ with finite center, then $G$ and any of its lattices have property $(T_{d-1})$.
\end{mthm}

For an \emph{irreducible} unitary $\SL_n(\bbR)$-representation $V$  the continuous cohomology of $\SL_n(\bbR)$ is well understood and satisfies a vanishing-below-the-rank phenomenon, most notably by the work of Zuckerman~\cite{Zuckerman}, Borel-Wallach~\cite{Borel-Wallach}, Vogan-Zuckerman~\cite{Vogan-Zuckerman} and Casselman~\cite{Casselman}, which we survey in~\S\ref{subsec:cohomreps}. In \S\ref{sec:finiteness} we develop tools to pass from irreducible unitary to arbitrary unitary representations. This allows us to prove Theorem~\ref{thm: higher property T main result} for the Lie group~$G$. It is possible to re-prove Garland's theorem along these lines. However, the passage to lattices in Theorem~\ref{thm: higher property T main result} is much harder than in Garland's result due to the existence of non-uniform lattices. A major novelty of the present paper is to use tools from geometric group theory to deal with the passage to non-uniform lattices. 

Theorem~\ref{thm: higher property T main result} is a special case of a more general result about semisimple Lie groups: 

\begin{mthm} \label{thm:semisimpleT}
Let $G$ be a connected semisimple Lie group~$G$ with a finite center.
Then $G$ has property $(T_{r_0(G)-1})$, where $r_0(G)$ is the invariant given in Appendix~\ref{sec:rG}.
Let $\Gamma$ be an irreducible lattice in $G$.
Then $\Gamma$ has property $(T_n)$, where $n+1=\min\{r_0(G),\rank(G)\}$.
In particular, both $G$ and $\Gamma$ have property $(T_m)$, where $m+1$ equals the minimal rank of a simple factor of $G$.
\end{mthm}
\begin{remark}
  The assumption of having a finite center can be dropped in the previous statement. First one shows the vanishing for irreducible representations using a spectral sequence argument, and then extends it to all unitary representations using the methods in~\S\ref{sec:finiteness} (cf.~the proof of Theorem~2.6 in~\cite{stability}).     
\end{remark}

Unlike $\SL_n(\bbQ_p)$, the group $\SL_n(\bbR)$ has non-trivial continuous cohomology for the trivial representation below the rank, so Theorem~\ref{thm: higher property T main result} does not hold for property $[T_{n-2}]$. It is well known that the continuous cohomology of the connected semisimple Lie group $G$ is isomorphic to the singular cohomology of the compact dual of the symmetric space $G/K$ where $K<G$ is a maximal compact subgroup~\cite{Borel-Wallach}*{Corollary IX.5.6}. The compact dual for $\SL_n(\bbR)$ is $\SU(n)/\SO(n)$, and we obtain that (see~\cite{mimura+toda}*{Theorem III.6.7 on p.~149})
\begin{equation}\label{eq: example cohomology}
H^\ast_c\bigr(\SL_n(\bbR),\bbC\bigr)\cong \begin{cases}\Lambda(e_5, e_9, \dots, e_{4m+1}) & \text{if $n=2m+1$;}\\
\Lambda(e_5, e_9, \dots, e_{4m-3})\otimes \bbC[e_{2m}]/(e_{2m}^2) & \text{if $n=2m$.}
\end{cases}
\end{equation}

The question of the invariance of the cohomology of $\SL_n(\bbZ)$, that is, the question to what extent it comes from the continuous cohomology of $\SL_n(\bbR)$ was studied by Borel in order to compute the rational algebraic K-theory of number fields. He proved the following result~\cite{Borel-stable}, which we state only for $\SL_n$. 

\begin{theorem}[Borel]
The restriction map 
\[H^i_c\bigl(\SL_n(\bbR),\bbC\bigr)\to H^i\bigl(\SL_n(\bbZ),\bbC\bigr)\]
is an isomorphism for $0\le i<(n-1)/4$. 
\end{theorem}

 Our next result improves Borel's stability range and allows for arbitrary unitary coefficients. 

\begin{mthm}\label{thm: bijectivity Lie to lattice}
Let $\Gamma$ be an irreducible lattice in a connected semisimple Lie group~$G$ with a finite center and without compact factors. 
Let $V$ be a unitary $G$-representation.
Then the restriction map 
\[ H_c^i(G,V)\to H^i(\Gamma, V) \]
is an isomorphism for every $0\le i<\rank G$. 
\end{mthm}

For the trivial representation and congruence lattices, Li-Sun provide a similar improvement of Borel's stability range~\cite{Li-Sun}*{Theorem~1.8}. 
We were informed that the result of Li-Sun also follows from the methods in Grobner's paper~\cite{grobner}. Both Li-Sun's and Grobner's papers rely heavily on the work of Franke~\cite{Franke}. Our approach instead relies on geometric group theory and, in the case that $G$ is a product of groups of rank~$1$, on Clozel's results on property~$\tau$. 

The next theorem makes a statement about the invariance of the cohomology of a lattice with coefficients that are a priori not unitary representations of the ambient Lie group.  

\begin{mthm} \label{thm:nonuniformextension}
    Let $\Gamma$ be an irreducible lattice in a connected semisimple Lie group~$G$ with a finite center and without compact factors. Assume $G$ has property T.
    Let $V$ be a unitary $\Gamma$-representation. Then there is 
    unitary subrepresentation $V_0\subset V$ on which the $\Gamma$-representation extends to a unitary $G$-representation such that for every $j<\rank(G)$, $H^j(\Gamma,V_0^\perp)=0$ and the maps 
    \[ H^j_c(G,V_0) \to H^j(\Gamma,V_0) \to H^j(\Gamma,V) \]
    are isomorphisms. Here the first map comes from the restriction from $G$ to $\Gamma$, and the second map comes from the inclusion of $V_0$ in $V$.
\end{mthm}

The subspace $V_0$ consists precisely of the $G$-continuous vectors as in Definition~\ref{def: G-continuous vectors}. 

\begin{conjecture} \label{conj1}
  The property T assumption in Theorem~\ref{thm:nonuniformextension} could be replaced with the assumption $\rank(G)\geq 2$.  
\end{conjecture}

Our next theorems are motivated by the following result of Borel-Yang~\cite{Borel-Yang} that was the key to their solution of the rank conjecture in algebraic K-theory. In the setting of \emph{bounded} cohomology, similar results were proved by~Monod~\cite{monod}. 

\begin{theorem}[Borel-Yang]
The restriction map 
\[ H^\ast_c\bigl(\SL_n(\bbR),\bbC\bigr)\to H^\ast\bigl(\SL_n(\bbQ),\bbC\bigr)\]
is an isomorphism in all degrees. More generally, if $k$ is a number field and $\mathbf{G}$ is a connected, simply connected, almost simple $k$-algebraic group, then the restriction map 
\[ H_c^*\bigl(\mathbf{G}(k\otimes\bbR),\bbC\bigr)\to H^*\bigl(\mathbf{G}(k),\bbC\bigr) \]
is an isomorphism in all degrees. 
\end{theorem}

Borel-Yang conjecture that the assumption that $\bfG$ is simply connected as an algebraic group is not necessary~\cite[Remark 3.4]{Borel-Yang}. We verify this conjecture in Theorem~\ref{thm:adelicC}. 
The paper by Borel-Yang is heavily based on the work of Blasius-Franke-Grunewald~\cite{BFG}, which itself relies on Franke's work~\cite{Franke}. 

The crucial input in our generalization of Borel-Yang's theorem (Theorem~\ref{thm:adelic} and~Theorem~\ref{thm:arith} below) comes again from geometric group theory. 

\begin{mthm} \label{thm:adelic}
Let $k$ be a number field and $\mathbb{A}(k)$ the ring of adeles of~$k$.
Let $\mathbf{G}$ be a connected, simply connected, almost simple $k$-algebraic group
and let $V$ be a unitary representation of the adelic group $\mathbf{G}(\mathbb{A}(k))$.
Then the restriction map
\[ \res:H_c^*(\mathbf{G}(\mathbb{A}(k),V)\to H^*(\mathbf{G}(k),V) \]
is an isomorphism in all degrees.
\end{mthm}

\begin{mthm} \label{thm:arith}
Let $k$ be a number field and $\mathbb{A}(k)$ the ring of adeles of~$k$.
Let $\mathbf{G}$ be a connected and simply connected almost simple $k$-algebraic group and $\Gamma<\mathbf{G}(k)$ an associated arithmetic subgroup.
Let $V$ be a unitary representation of the Lie group $\mathbf{G}(k\otimes \bbR)$ and consider the natural maps
\[ H_c^*(\mathbf{G}(k\otimes \bbR),V)\to H_c^*(\mathbf{G}(\mathbb{A}(k)),V)\to H^*(\mathbf{G}(k),V) \to H^*(\Gamma,V). \]
Then the first two maps are isomorphisms in all degrees and the last map, namely $H^*(\mathbf{G}(k),V) \to H^*(\Gamma,V)$, 
is an isomorphism in degrees lower than
$\rank \mathbf{G}(k\otimes \bbR)$.
\end{mthm}

In fact, the entire unitary cohomology theory of $\mathbf{G}(k)$ is determined by that of $\mathbf{G}(\mathbb{A}(k))$, at least when the $k$-rank of $\bfG$ is at least 2, e.g for $\bfG=\SL_n$ for $n\geq 3$. Although $\bfG(k)$ is not of type I and a complete understanding of its unitary dual is out of reach, Theorems~\ref{thm:adelic} and \ref{thm:Gkhigherrank} could be used to determine the cohomological unitary dual of $\bfG(k)$ completely. See Corollary~\ref{Gkcohomdual}.  

\begin{mthm} \label{thm:Gkhigherrank}
Let $k$ be a number field and let $\bfG$ be a 
connected and simply connected almost simple $k$-algebraic group, satisfying $\rank_k\bfG\geq 2$.
    Let $V$ be a unitary $\bfG(k)$-representation.
    Then there exists a $\bfG(k)$-subrepresentation $V_0$ on which the representation extends to $\bfG(\bbA(k))$ such that $H^\ast(\bfG(k),V_0^\perp)=0$.
    In particular, the maps 
    \[ H^\ast_c(\bfG(\bbA(k)),V_0) \to H^\ast(\bfG(k),V_0) \to H^\ast(\bfG(k),V) \]
    are isomorphisms,
    where the first map comes from the restriction from $\bfG(\bbA(k))$ to $\bfG(k)$ and the second map comes from the inclusion of $V_0$ in $V$.
\end{mthm}

The assumption on the rank of $\bfG$ in Theorem~\ref{thm:Gkhigherrank} could be relaxed, see Theorem~\ref{thm:GknonGA} below. This provides some evidence for the following.

\begin{conjecture} \label{conj2}
  The assumption on the rank of $\bfG$ in Theorem~\ref{thm:Gkhigherrank} can be removed.  
\end{conjecture}

Both Conjecture~\ref{conj1} and Conjecture~\ref{conj2} follow from Conjecture~106, which is concerned with the spectral gap property of higher rank lattices, in our survery paper~\cite{survey}. We refer the reader to~\cite{survey}*{\S~3.4.1} for a detailed discussion of that conjecture and its ramifications.

\subsection{Further applications}

Since the first version of this paper was circulated, several applications of our techniques and results have appeared in the literature. We first list applications of our techniques, before turning to applications of our results. 

The ultraproduct techniques of~\S\ref{sec:finiteness} were used in~\cite{waist} to prove certain coboundary expansion properties in the cellular cohomology, which lead to waist inequalities in codimension~$2$ for Riemannian manifolds with Kazhdan fundamental groups. Techniques about Shapiro-like induction and polynomial cohomology as in~\S\ref{sec:PCH}, specifically, a quantitative version of Proposition~\ref{prop: comparison map and polynomial filling rank}, were used by L{\'o}pez Neumann and Paucar~\cite{neumann+paucar} to obtain results on the invariance of Betti numbers under $\ell^p$-measure equivalence. Further, they develop techniques to deal with unitary Shapiro-like induction for non-uniform lattices in Lie groups of rank~$1$, extending our results on higher rank lattices. 

In~\cite{FournierFacio-Sauer} Fournier-Facio and the second author use group-theoretic Dehn fillings and Theorem~\ref{thm:semisimpleT}, specifcally, property $(T_2)$ for cocompact lattices in $F_4^{-20}$, to construct a family of simple Kazhdan groups with an uncountable range of second $\ell^2$-Betti numbers. See also Fournier-Facio's preprint~\cite{FournierFacio} for a related construction that shows that poperty $(T_2)$ does not pass to quotients in general. 
Cohomological vanishing in  degree 2 has applications in the theory of stability and finitary approximations of groups, which was an important insight of the work of de Chiffre-Glebsky-Lubotzky-Thom~\cites{stability, Thom-ICM}. Theorem~\ref{thm: higher property T main result} yields a positive answer to Question~3.16 in Thom's ICM survey~\cite{Thom-ICM} on finitary approximations of groups. 
Via Theorem~\ref{thm: higher property T main result}, it is proved in~\cite{stability_four} that the non-trivial central extensions of symplectic lattices are not Frobenius-approximable. In~\cite{survey} we announced a version of Theorem~\ref{thm: higher property T main result} for coefficients in $L^p(X)$, $p\ge 1$, which will appear in future joint work with Shaked Bader and Saar Bader. The proofs draw essentially on the techniques and results developed in the present paper in two ways: First, the ultraproduct techniques of~\S\ref{sec:finiteness} and the Shapiro-type results for polynomial cohomology in~\S\ref{sec:PCH} are agnostic about the value of~$p$. Second, the Mazur map~\cite{bader+furman+gelander+monod}*{Theorem~2.17} allows to transfer 
certain spectral gap results from the case $p=2$, i.e.~from Hilbert spaces, to the case of general~$p$.

Finally, for a detailed exploration of the implications of higher property~T we refer to our survey paper~\cite{survey}.

\subsection{Structure of the paper}
We start with \S\ref{sec: setup} where we give our setting and notation. 
In \S\ref{sec:properties} 
we establish some general results regarding the properties $(T_n)$ and $[T_n]$ given in Definition~\ref{def:T_n}.
The most important results of this section are Theorem~\ref{thm:gensh} and Theorem~\ref{thm:r0}.
The former 
could be seen as a generalization of Shalom's \cite[Theorem 0.2]{Shalom}, which proves a conjecture by Karpushev--Vershik~\cite{Karpushev-Vershik}, while the latter is an application for semisimple groups.
Theorem~\ref{thm:gensh} allows us to extend vanishing results for the continuous cohomology of irreducible unitary representations of semisimple Lie groups in~\S\ref{sec:CSS} to arbitrary unitary representations. The methods use finiteness properties of lattices and ultrapowers of unitary representations. 

In \S\ref{sec:SG} we discuss spectral gap properties of semisimple groups.
The most important result of this section is Theorem~\ref{Clozel++}, which is a generalization of Clozel's Theorem \cite[Theorem 3.1]{Clozel} regarding property~$\tau$. This section is only needed when we deal with semisimple groups whose simple factors do not have Kazhdan's property~T. The spectral gap property ensures, when applying the results of~\S\ref{sec:finiteness}, that the ultrapower of a representation without invariant vectors has no invariant vectors. If property T is not available, then we need Theorem~\ref{Clozel++}. 

In \S\ref{sec:CSS} we discuss the theory of cohomological representation of semisimple groups, following Zuckerman \cite{Zuckerman}, Borel-Wallach \cite{Borel-Wallach}, Vogan-Zuckerman \cite{Vogan-Zuckerman} and Casselman\cite{Casselman}, among others.

In \S\ref{sec:PCH} we discuss the polynomial cohomology of semisimple groups and their lattices. We use in an essential way results of Leuzinger-Young~\cite{leuzinger+young} and Bestvina-Eskin-Wortman~\cite{bestvina+eskin+wortman} regarding filling functions of arithmetic lattices to show that polynomial cohomology and ordinary cohomology coincide in low degrees (see~\S\ref{subsec:Shapiro++}). 

There we address a major problem briefly mentioned in the introduction. If we want to deduce a higher property~T result from a semisimple Lie group~$G$ to a non-uniform lattice of $G$, then the usual Shapiro isomorphism fails us  since the induction module is a space of locally square-integrable functions and not unitarizable. There is no general induction scheme for non-uniform lattices and unitary coefficients available. 
We are not quite able to prove a strict analog of the Shapiro isomorphism for the polynomial cohomology of (non-uniform) lattices and unitary coefficents but Theorems~\ref{thm: Shapiro surjectivity} and~\ref{thm: injectivity G to lattice} are sufficiently strong to prove everything in our specific situation that we would get from a full Shapiro isomorphism. 

Finally, all the theorems stated in the introduction are proven in \S\ref{sec:pfmain}. We also discuss some further results, among those a complete determination of the cohomological dual of the rational points of a semisimple algebraic group with property~T.

\subsection{Acknowledgment}
We are grateful to Rami Aizenbud, Michael Cowling, Alex Furman, Dima Gourevitch, Erez Lapid, Piotr Nowak, Andrei Rapinchuk, and Marco Tadic for many conversations, advice, and reference pointers. 

Special thanks go to Alex Lubotzky and Yehuda Shalom for lots of discussions about the paper, encouragement and sharing their wisdom.

We are especially grateful to Marc Burger who suggested to try to reprove the result of Borel-Yang~\cite{Borel-Yang} using our methods after hearing us giving a talk on Theorem~\ref{thm: bijectivity Lie to lattice} above. His suggestion led us ultimately to proving Theorem~\ref{thm:adelic}. 

U.B acknowledges support by the ISF Moked grant 2919/19.
R.S acknowledges support by the projects  441426599 and 441848266 and 281869850 funded by the DFG (Deutsche Forschungsgemeinschaft). R.S. thanks the 
the Weizmann Institute for the hospitality when part of this work was carried out. 

\section{Notation and setup}\label{sec: setup}

All fields considered in this paper will be assumed to be of characteristic zero.
We will denote by $k$ a number field and by $\calO$ its ring of integers.
By a \emph{local field} we mean a non-discrete locally compact field.
A local field which is not a priori an extension of $k$ will be typically denoted by $F$.
By a \emph{place of $k$}, typically denoted by $s$, we mean a compatible uniform structure on $k$ such that the completion, denoted $k_s$, is a local field.
The place $s$ is called archimedean (or infinite) if $k_s$ is archimedean.
In this case $\calO$ is either discrete or dense in $k_s$. Otherwise $s$ is said to be non-archimedean (or finite), and we denote by 
$\calO_s$ the closure of $\calO$ in $k_s$, which is a compact subring. 
By the letter $S$ we will typically denote a set of places, not necessarily finite.
We will denote by $\mathbb{A}_S(k)$ the $k$-algebra of \emph{$S$-adeles}, that is the restricted product of all local fields $k_s$, with respect to the a.e defined compact subrings $\calO_s$, running over all non-archimedean $s\in S$.
We will endow this $k$-algebra with the restricted product topology, which is locally compact.
When $S$ is the full set of places of $k$, we get the algebra of \emph{$k$-adeles}, $\mathbb{A}(k)$, and when $S$ is the subset of non-archimedean places we get the algebra of \emph{finite $k$-adeles}, $\mathbb{A}_f(k)$.
When $S$ is the subset of archimedean places, we get the algebra of \emph{infinite $k$-adeles}, that is the product of the archimedean completions of $k$, which we identify with $k\otimes \bbR$.
In particular, we obtain the identification $\mathbb{A}(k)\cong (k\otimes \bbR) \times \mathbb{A}_f(k)$.
Assuming $S$ contains all archimedean places of $k$, we denote by $\calO_S$ the subring of $S$-integers of $k$, that is the preimage of $\prod_{s\in S^c} \calO_s$ under the natural map $k\to \mathbb{A}_{S^c}(k)$, where $S^c$ is the completion of $S$.
The image of the ring homomorphism $\calO_S \to \mathbb{A}_{S}(k)$ is cocompact.
In particular, the images of $\calO\to k\otimes \bbR$ and $k\to \mathbb{A}(k)$ are cocompact.

All algebraic groups considered here are affine, typically semisimple.
They will be considered as schemes, thus could be defined over arbitrary rings.  
We will typically denote an algebraic group by $\bfG$. If it is defined over a local field $F$, then $\bfG(F)$ is often considered as a topological group with respect to its $F$-topology.  
Similarly, if $\bfG$ is defined over $k$, we regard the adelic group $\bfG(\mathbb{A}(k))$ as a topological group. For a detailed account of the corresponding group topologies, we direct the reader to~\cite{Conrad}.
In particular, by \cite[Theorem 3.4(1)]{Conrad}, $\bfG$ is defined over $\calO_{S_0}$, for a certain 
finite set $S_0$ of places of $k$ containing all archimedean ones,
thus $\bfG(O_s)$ is well defined for every $s \notin S_0$,
and the topology of $\bfG(\mathbb{A}(k))$ coincides with the restricted product of the groups $\bfG(k_s)$ with respect to the a.e defined compact subgroups $\bfG(\calO_s)$.
A similar discussion applies to the groups $\bfG(\mathbb{A}_S(k))$ when $S$ is an arbitrary set of places of $k$.
If $\bfG$ is reductive then it is quasi-split, hence isotropic, over $k_s$ for almost every place $s$ of $k$~\cite{conrad-reductive}*{Ex.~5.5.3}.

Assume $\bfG$ is a $k$-algebraic group.
The locally compact group $\bfG(\mathbb{A}_f(k))$ is totally disconnected, hence it has a basis of identity neighborhoods consisting of compact open subgroups. The intersection of such a compact open subgroup with $\bfG(k)$ is called \emph{a congruence subgroup} of $\bfG(k)$. A subgroup of $\bfG(k)$ that is commensurable with a congruence subgroup is said to be \emph{an arithmetic subgroup} associated with the $k$-group $\bfG$. 
More generally, given a set $S$ consisting of places of $k$ and including all archimedean places, the intersection of $\bfG(k)$ with a compact open subgroup of $\bfG(\mathbb{A}_{S^c}(k))$ is called \emph{an $S$-congruence subgroup} of $\bfG(k)$ and a subgroup of $\bfG(k)$ that is commensurable with an $S$-congruence subgroup is said to be \emph{an $S$-arithmetic subgroup}.
Assume $\bfG$ is semisimple.
Then the image of an $S$-arithmetic subgroup in $\bfG(\mathbb{A}_{S}(k))$ is a lattice by a celebrated theorem of Borel and Harish-Chandra.
These lattices are called \emph{$S$-arithmetic lattices}.
They are cocompact if and only if $\bfG$ is $k$-anisotropic.
In the special case where $S$ consists of all places we get that the image of $\bfG(k)$ is a lattice in $\bfG(\mathbb{A}(k))$.
When $S$ consists of only the archimedean places, we get that the images in $\bfG(k\otimes \bbR)$ of arithmetic subgroups  of $\bfG(k)$ are lattices, called \emph{arithmetic lattices}.

The symbol $\Gamma$ always stands for an arbitrary group, typically countable, possibly satisfying some finiteness properties and many times considered as a lattice in a locally compact topological group. The symbol $G$ always stands for a locally compact second countable topological group.
However, many times $G$ will have further structure and properties.
Below we indicate the two main instances of such structures we impose on $G$, namely the structure of a \emph{semisimple Lie group} and the structure of a general \emph{semisimple group}. We will be always explicit regarding the structure we impose on $G$.

Often times $G$ will be assumed to be a semisimple Lie group with finite center and without compact factors.
In this case, a lattice $\Gamma<G$ is said to be \emph{irreducible} if the image of $\Gamma$ is dense modulo each non-compact almost simple factor of $G$. 
Typically, we denote the Lie algebra of a Lie group $G$ by $\mathfrak{g}$. 
A notable exception is~\S\ref{subsec:cohomreps}, where $\mathfrak{g}$ stands for the complexification of the Lie algebra of $G$. This will be clearly stated.
The \emph{rank} of $G$, $\rank(G)$, of a semisimple Lie group $G$ is defined to be the rank of its Lie algebra, namely the dimension of any of its Cartan subalgebras.

When we say that $G$ is a \emph{semisimple group} (without mentioning the word ``Lie"), we mean that as topological groups, $G \cong \prod_{i=1}^l \bfG_i(F_i)$ where each $F_i$ is a local field and each $\bfG_i$ is a connected almost simple $F_i$-group.\footnote{We could allow here the groups $\bfG_i$ to be $F_i$-semisimple, but we choose not to do so, for simplicity.} 
In this case, $G$ is totally disconnected if and only if  each $F_i$ is non-archimedean and $G$ is a semisimple Lie group if and only if  each $F_i$ is archimedean.
We will typically denote $G_i=\mathbf{G}_i(F_i)$ and regard these groups as the simple factors of $G$.
A lattice in a semisimple group, $\Gamma<G$, is \emph{irreducible} if it is dense modulo each non-compact simple factor of $G$.
The \emph{rank} of $G$, $\rank(G)$, of a semisimple group $G \cong \prod_{i=1}^l \bfG_i(F_i)$ is defined to be the sum of the ranks of its factors, $\rank(G)=\sum \rank_{F_i}(\bfG_i)$.

Every semisimple Lie group $G$ is closely related to a real algebraic group $\bfG$, namely the group of Lie algebra automorphisms of its Lie algebra $\mathfrak{g}$.
We note that the two possible notions of rank coincide, so there is no chance of confusion.
A semisimple group $G$, Lie or not, is said to be of higher rank if $\rank(G)\geq 2$.

Topological vector spaces are considered over the complex numbers. They will typically be assumed to be Fr\'{e}chet; in particular, Hilbert spaces are generally assumed to be separable. 
An exception is when we consider an ultrapower of another (separable) Hilbert space. In such a case we will explicitly address the non-separability. Another exception is the space of continuous cohomology of a Fr\'{e}chet module, which carries the structure of a (often non-Hausdorff) topological vector space. See below. 

We will use the term $\hat{\otimes}$ to denote the \emph{projective tensor product} of Fr\'{e}chet spaces and we use the term $\bar{\otimes}$ exclusively to denote the \emph{Hilbertian tensor product} of Hilbert spaces.

For an arbitrary group $\Gamma$ and a $\Gamma$-module $V$, we denote by $H^j(\Gamma,V)$ the corresponding group cohomology and we set, as usual,  $H^*(\Gamma,V)=\oplus_{j=0}^\infty H^j(\Gamma,V)$.

For a locally compact second countable group $G$, a Fr\'{e}chet module $V$
is a Fr\'{e}chet space that is endowed with an action of $G$ by continuous automorphisms such that the action map $G\times V\to V$ is continuous.
In such a case we denote by $H^j_c(G,V)$ the corresponding continuous $j$-cohomology and we set $H^*_c(G,V)=\oplus_{j=0}^\infty H^j_c(G,V)$.
Each of these spaces have a natural vector space topology which is typically non-Hausdorff. Their Hausdorffification, that is the spaces obtained by moding out $\overline{\{0\}}$, are Fr\'{e}chet spaces which are denoted 
$\bar{H}^j_c(G,V)$ and $\bar{H}^*_c(G,V)$ correspondingly.
These spaces are said to be the \emph{reduced continuous cohomologies} of $G$.
It is common to say that $H^j_c(G,V)$ is reduced if it is Hausdorff. We will avoid this practice to prevent confusion.
For an important background material on continuous cohomology of Fr\'{e}chet spaces we direct the reader to the manuscripts \cite{Blanc} and \cite{Guichardet}.

For a Hilbert space $V$ we denote by $\mathcal{U}(V)$ the group of unitary transformations of $V$. A homomorphism $\pi:G \to \mathcal{U}(V)$ is said to be a \emph{unitary representation} of $G$ if the associated action of $G$ on $V$ makes it a Fr\'{e}chet module. Equivalently, this means that the representation $\pi$ is strong operator topology continuous. Our convention is to regard a unitary representation as a Fr\'{e}chet module, thus we often  refer to the representation $\pi$ by the underlying Hilbert space $V$. 
The trivial representation is denoted either $\bbC$ or ${\bf 1}$.

\section{Properties $[T_n]$ and $(T_n)$} \label{sec:properties}

In this section we establish some general results regarding the properties $(T_n)$ and $[T_n]$ given in Definition~\ref{def:T_n}.
The most important results of this section are Theorem~\ref{thm:gensh} and Theorem~\ref{thm:r0}.
The former 
could be seen as a generalization of Shalom's characterization of property T, \cite[Theorem 0.2]{Shalom}, which proved a conjecture by Karpushev--Vershik~\cite{Karpushev-Vershik} and generalized work of Mok~\cite{mok} in the Riemannian context, while the latter is an application for semisimple groups. Both theorems rely on taking ultralimits of unitary representations, a technique that was used in the context of property T in, for example,~\cites{ozawa, fisher+margulis}. The novelty here lies in using the ultrapower technique in higher degree cohomology. As a first application, we obtain a higher property~T result for semisimple groups (see Theorem~\ref{thm:r0}).

But first, a word of caution. 

\begin{remark}
    Definition~\ref{def:T_n} of higher property~T slightly differs from previous definitions given in~~\cite{BN1}*{Definition 30} and \cite{stability}*{Definition 4.1}.
\end{remark}

\subsection{Generalities regarding higher property T}
For a locally compact second countable group $G$ and a Fr\'{e}chet $G$-representation $U$,
the continuous group cohomology $H^n_c(G,U)$ has a natural topology which is not necessarily Hausdorff. The \emph{reduced cohomology} $\bar{H}^n_c(G,U)$ is its maximal Hausdorff quotient that is a Fr\'{e}chet space. 
Our background references for continuous group cohomology are~\cites{Blanc, Guichardet}.

Unitary representations are always assumed to be continuous in the strong operator topology. 

\begin{lemma} \label{lem:T_1}
    For a locally compact second countable group $G$ the following properties of $G$ are equivalent: Kazhdan's property T, $(T_1)$, $[T_1]$, and $H^1_c(G,V)$ being Hausdorff for every unitary $G$-representation. 
\end{lemma}

\begin{proof}
By the Delorme-Guichardet Theorem, property $[T_1]$ is equivalent to Kazhdan's property T.
    Clearly, $[T_1]$ implies $(T_1)$.
    Assume $G$ has $(T_1)$ and let $U$ be a unitary representation of $G$.
    Let $V=U^G$. We deduce from $(T_1)$ that
    \[ H^1_c(G,U)=H^1_c(G,V\oplus V^\perp)\cong H^1_c(G,V)\oplus H^1_c(G,V^\perp)\cong \Hom(G,V)\oplus 0. \]
    We show that $G$ has a compact  abelianization, hence $\Hom(G,V)=0$.
    Assume that $G$ has a non-compact abelian quotient $A$. We have  $L^2(A)^G=0$ but there exists in $L^2(A)$ an almost invariant sequence of unit vectors. Then $H^1_c(G,L^2(A))\neq 0$ by~\cite{Tbook}*{Proposition~2.12.2 on p.~128} -- contradicting the assumption that $G$ has $(T_1)$. 
    Finally, $H^1_c(G, \_)$ being Hausdorff is equivalent to other properties (see~\emph{loc. cit.}). 
\end{proof}

Compact groups satisfy property $[T_n]$ for every $n$.
In fact, their higher degree cohomology groups are trivial for arbitrary Fr\'{e}chet coefficients, as the following basic result shows.

\begin{lemma}[{\cite[III, Corollary 2.1]{Guichardet}}] \label{lem:compact}
    Assume $G$ is compact. Then for every Fr\'{e}chet representation $V$ of $G$, $H^\ast_c(G,V)=H^0_c(G,V)=V^G$. 
    In particular, $G$ satisfies property $[T_n]$ for every $n$.
\end{lemma}

\begin{lemma}\label{lem: product with compact group}
Let $G$ be a locally compact second countable group and $K \lhd G$ a compact normal subgroup. Let $V$ be a unitary representation of $G$. Then the inclusion of the fixed points $V^K\hookrightarrow V$ and the projection $G\to G/K$ induce an isomorphism 
\[ H_c^\ast(G/K,V^K)\xrightarrow{\cong} H_c^\ast(G, V).\]
In particular, $G$ satisfies property $(T_n)$ if and only if  $G/K$ satisfies property $(T_n)$ and $G$ satisfies property $[T_n]$ if and only if  $G/K$ satisfies property $[T_n]$.
\end{lemma}

\begin{proof}
    by Lemma~\ref{lem:compact}, $H^\ast_c(K,V)=H^0_c(K,V)=V^K$. In particular, all the cohomologies are reduced, thus we may apply the Hochschild-Serre spectral sequence associated with the normal subgroup $K\lhd G$,
    which computes $H_c^\ast(G, V)$ and in which we have $E_2^{p,q}=H_c^p(G/K,H_c^q(K,V))$, see \cite[Theorem 9.1]{Blanc}.
    The result is now straightforward, as all terms with $q>0$ vanish.
\end{proof}

\begin{lemma} \label{lem:directint}
    Let $G$ be a locally compact second countable group and $j\in\bbN$. 
    Then $\bar{H}_c^j(G,U)=0$ for every irreducible unitary representation $U$ if and only if $\bar{H}_c^j(G,U)=0$ for every unitary representation $U$. 
        Furthermore, $\bar{H}_c^j(G,U)=0$ for every non-trivial irreducible unitary representation $U$ if and only if $\bar{H}_c^j(G,U)=0$ for every unitary representation $U$ with  $U^G=0$.  
\end{lemma}

\begin{proof}
Let $V$ be a unitary $G$-representation. 
Assume that for every irreducible unitary representation $U$, $\bar{H}^j_c(G,U)=0$.
    By \cite[Theorem F.5.3]{Tbook} $V$ is equivalent to a direct integral of irreducible unitary representations over a standard Borel space.
    Thus it follows by \cite[Theorem 7.2]{Blanc} that $\bar{H}^j_c(G,V)=0$.
    This proves the first part of the lemma. 
    The second part of the lemma follows from the fact that if $V^G=0$ then almost every irreducible representation in the direct integral decomposition is non-trivial.
\end{proof}

\subsection{Groups with finiteness properties} \label{sec:finiteness}

We now consider a countable group $\Gamma$ with property $\FP_n(\bbQ)$,
that is, the trivial $\bbQ\Gamma$-module $\bbQ$ has a projective resolution
\[ \cdots \to P_2 \to  P_1 \to P_0 \to \bbQ\to 0, \]
where the $\bbQ\Gamma$-modules $P_0,\ldots,P_n$ are finitely generated. 
Without loss of generality, we may and will assume that each $P_i$, $i\in\{0,\dots,n\}$, is finitely generated free. Let $m_k$ be the $\bbQ\Gamma$-rank of $P_k$. We pick $\bbQ\Gamma$-bases in $P_0, \dots, P_n$. The differential $P_k\to P_{k-1}$ for $k\le n$ corresponds to a matrix $D_k\in M_{m_{k}\times m_{k-1}}(\bbQ\Gamma)$. 

Let $V$ be a unitary $\Gamma$-representation. Then $H^*(\Gamma,V)$ is the cohomology of the cochain complex $\hom_{\bbQ\Gamma}(P_\ast, V)$. Up to degree~$n$, this is a cochain complex of Hilbert spaces with bounded differentials $d^k$ via the isomorphisms 
\begin{equation}\label{eq: cochain bases}
\Hom_\Gamma(P_k,V)\cong V^{m_k},~~d^k\colon V^{m_{k-1}}\to V^{m_k}
\end{equation}

induced by the $\bbQ\Gamma$-bases. We may view $V$ as a right $\bbQ\Gamma$-module via the involution on $\bbQ\Gamma$ induced by taking inverses on $\Gamma$. 
Then, similarly, the chain complex $V\otimes_{\bbQ\Gamma} P_\ast$ 
is a chain complex of Hilbert spaces up to degree~$n$ via the  isomorphism $V\otimes_{\bbQ\Gamma} P_k\cong V^{m_k}$. For $k\le n$ the differential 
\begin{equation}\label{eq: chain bases}
d_k\colon V^{m_k}\cong V\otimes_{\bbQ\Gamma} P_k\to  V\otimes_{\bbQ\Gamma} P_{k-1}\cong V^{m_{k-1}} 
\end{equation}
and the differential $\hom_{\bbQ\Gamma}(P_{k-1}, V)\to \hom_{\bbQ\Gamma}(P_k, V)$ are adjoints of each other. 
\begin{equation}\label{eq: adjoint differentials}
(d_k)^\ast=d^k
\end{equation}

The following theorems rely on taking ultrapowers of unitary representations which violate the standing assumption that all Hilbert spaces are assumed to be separable. We will frequently rely on Lemma~\ref{lem: non-separable reps} below to reduce the situation to separable representations. 
\begin{lemma}\label{lem: non-separable reps}
 Let $G$ be a locally compact second countable group and $j\in\bbN$. Let $V$ be a possibly non-separable unitary $G$-representation. 
 For every $x\in H_c^j(G,V)$ there is a separable unitary $G$-subrepresentation $U$ of $V$ such that $x$ is in the image of $ H_c^j(G,U)\to H_c^j(G,V)$. A similar statement holds true for reduced cohomology.  
 \end{lemma}

\begin{proof}
The element $x$ is represented by a continuous cocycle $\phi:G^{j+1}\to V$. Choose a countable dense subset in $G^{j+1}$ and let $V'$ be the closure of the span of its image under $\phi$. Then $V'$ is a separable subspace of $V$ that contains the image of $\phi$. Let $U$ be a closed separable $G$-invariant subspace of $V$ that contains $V'$. For instance, take $U$ to be the closure of the image of $L^1(G)\otimes V'$. Then $x$ is clearly in the image of $H_c^j(G,U)\to H_c^j(G,V)$. 
\end{proof}
\begin{theorem} \label{thm:gensh}
    Let $n\geq 1$. Let $\Gamma$ be a countable group with property $\FP_n(\bbQ)$. Then we have the following equivalences. 
\begin{enumerate}[a)]
    \item The group $\Gamma$ has $[T_n]$ if and only if for every irreducible unitary representation $U$ and for every $0<k\leq n$ we have $\bar{H}^k(\Gamma,U)=0$. 
    \item The group $\Gamma$ has $(T_n)$ if and only if $\Gamma$ has a finite abelianization and for every non-trivial irreducible unitary representation $U$ and for every $0<k\leq n$ we have  $\bar{H}^k(\Gamma,U)=0$.
\end{enumerate}
\end{theorem}

\begin{proof}
    In both statements the ``only if" part is obvious. We will prove the ``if" parts.

 For the first statement we assume that $\bar{H}^k(\Gamma,U)=0$ for every irreducible unitary representation $U$ and for every $0<k\leq n$. Lemma~\ref{lem:directint} implies that 
\begin{equation} \label{eq:red}    
    \bar{H}^k(\Gamma,V)=0~\text{for every $0<k\leq n$ and every unitary $\Gamma$-representation $V$.}
\end{equation}
By contradiction, let us assume that there is $0<k\leq n$ and a unitary $\Gamma$-representation $V$ such that ${H}^k(\Gamma,V)\ne 0$. We adhere to the notation~\eqref{eq: cochain bases}~\eqref{eq: chain bases}, and freely use~\eqref{eq: adjoint differentials}. 

By assumption, the image of $d^k\colon V^{m_{k-1}}\to V^{m_{k}}$ is not closed. Hence the image of $d_{k}=(d^k)^\ast\colon V^{m_{k}}\to V^{m_{k-1}}$ is not closed.
By the Open Mapping Theorem, there exists a sequence of unit cocycles $\bar{v}^i=(v_1^i,\ldots,v_{m_{k}}^i)_{i=1}^\infty$
   in $V^{m_{k}}$ satisfying $\lim_i d_k(\bar{v}^i)=0$ in $V^{m_{k-1}}$.
We fix an ultrafilter $\omega$ on $\bbN$ and consider the ultrapower $V_\omega$
and the corresponding element $(\bar{v}^i_\omega) \in V_\omega^{m_{k-1}}$ (we make here the obvious identification between $(V_\omega)^{m_{k}}$ and $(V^{m_{k}})_\omega$).
The $k$-cochain $(\bar{v}^i_\omega)\in V^{m_{k}}_\omega$ in the chain complex 
\[ \cdots \leftarrow \Hom_\Gamma(P_{n+1},V_\omega) \leftarrow V^{m_{n}}_\omega \xleftarrow{d^n}  \cdots \xleftarrow{d^{k+1}} V^{m_{k}}_\omega \xleftarrow{d^k}  \cdots \xleftarrow{d^1}  V^{m_0}_\omega \leftarrow 0, \]
is a unit cocycle satisfying $d_k(\bar{v}^i_\omega)=(d^k)^\ast(\bar{v}^i_\omega)=0\in V_\omega^{m_{k-1}}$.
It follows that $(\bar{v}^i_\omega)$ is a cocycle in $V^{m_{k}}_\omega$ that is orthogonal to all coboundaries.
Thus $\bar{H}^k(\Gamma,V_\omega)\neq 0$. 
The ultrapower $V_\omega$ violates the standing assumption about separability but with Lemma~\ref{lem: non-separable reps} we get a contradiction to~\eqref{eq:red}.
This concludes the proof of the first statement.

Next, we consider the second statement.
Since $\Gamma$ is finitely generated and it has a finite abelianization, we have ${H}^1(\Gamma,\bbC)=0$,
thus we get that $\Gamma$ has $[T_1]$ by the first statement.
Assume that $V^\Gamma=0$. The proof of the first statement applies verbatim to show that $\Gamma$ has $(T_n)$
provided we have $(V_\omega)^\Gamma=0$. But since there are no almost invariant vectors in $V$ by property T, the ultrapower $V_\omega$ has no invariant vectors. Another application of Lemma~\ref{lem: non-separable reps} finishes the argument. 
\end{proof}

\begin{remark}
    Fixing a choice of the ultrafilter $\omega$ in the above proof, we obtain a \emph{ultrapower functor} $V\mapsto V_\omega$. This functor commutes with direct sums and it has the property that if $T:V\to U$ fails to have a closed image then $T_\omega:V_\omega \to U_\omega$ fails to be injective.
\end{remark}

\begin{remark} \label{rem:Lap}
    In the above proof, knowing that the image of $d_k^*$ is not closed, we could have deduced that the image of the Laplacian operator $d_{k+1}d_{k+1}^*+d_k^*d_k$ is not closed and apply the functoriality mentioned above to deduce that the Laplacian is not injective on $V_\omega^{m_k}$, thus the corresponding reduced cohomology is non-trivial.
    However, this method fails in degree $n$ without assuming that $\Gamma$ has property $\FP_{n+1}(\bbQ)$.
\end{remark}

\begin{remark}
    If $\Gamma$ satisfies $\FP_n(\bbQ)$ then for every $k\leq n$, $H^k(\Gamma,\bbC)$ is finite dimensional. If, moreover, $\Gamma$ has $(T_n)$ then for every unitary $\Gamma$-representation $U$,
$H^k(\Gamma,U)\cong H^k(\Gamma,U^G)\cong H^k(\Gamma,\bbC)\otimes U^G$ and this cohomology is reduced.
Note, however, that for $n>1$, having that for every $k\leq n$ and every unitary $\Gamma$-representation $U$,
$H^k(\Gamma,U)$ is reduced, does not guarantee that $\Gamma$ has $(T_n)$, as $\SL_{n+1}(\bbQ_p)$ and its lattices have reduced cohomology groups in \emph{all} degrees.
\end{remark}

\begin{theorem} \label{thm:red}
    Fix $n\geq 2$. 
    Let $\Gamma$ be a countable group satisfying property $\FP_n(\bbQ)$ and property $(T_{n-1})$.
    Then for every unitary $\Gamma$-representation $U$, $H^{n}(\Gamma,U)$ is Hausdorff.

    In particular, if $\Gamma$ is a finitely presented group with property T, then 
    for every unitary $\Gamma$-representation $U$, $H^{2}(\Gamma,U)$ is Hausdorff.
\end{theorem}

We note that the case $\Gamma=\SL_3(\bbZ)$ in the above theorem follows from \cite{Kaluba-Mizerka-Nowak}, who proved it with a computer aid, using the technique suggested in \cite{BN2}.
In fact, Theorem~\ref{thm:red} shows that the assumption that $H^{n+1}$ is reduced could be dropped in the main theorem of \cite{BN2}.

\begin{proof}
    Let $V$ be a unitary $\Gamma$-representation. 
    Then $H^{n}(\Gamma,V^G)\cong H^{n}(\Gamma,\bbC) \otimes V^G$ 
    and $H^{n}(\Gamma,\bbC)$ is finite dimensional by property $\FP_n(\bbQ)$. Hence $H^{n}(\Gamma,V^G)$ is Hausdorff. 
    So we may assume that $V^G=0$.
    Assume that $H^{n}(\Gamma,V)$ is not Hausdorff, that is, $d^n$ does not have a closed image.
    Applying now the method indicated in Remark~\ref{rem:Lap} basically finishes the proof, but we will follow closely the proof and the notation of Theorem~\ref{thm:gensh}.
    
    By the Open Mapping Theorem, there exists a sequence of unit vectors $\bar{v}^i=(v_1^i,\ldots,v_{m_{n-1}}^i)_{i=1}^\infty$
   in $d^{n-1}(V^{m_{n-2}})^\perp$ satisfying $\lim_i d^{n}(\bar{v}^i)=0$ in $V^{m_{n}}$.
We fix an ultrafilter $\omega$ on $\bbN$ and consider the ultrapower $V_\omega$
and the corresponding element $(\bar{v}^i_\omega) \in V_\omega^{m_{n-1}}$.
Since $\Gamma$ has property T and $V^\Gamma=0$, we obtain that $(V_\omega)^\Gamma=0$.
The vector $(\bar{v}^i_\omega)$ is a unit vector in $d^{n-1}(V^{m_{n-2}}_\omega)^\perp<V^{m_{n-1}}_\omega$ satisfying $d^n(\bar{v}^i_\omega)=0\in V_\omega^{m_{n}}$.
It follows that $\bar{v}^i_\omega$ is a cocycle in $V^{m_{n-1}}_\omega$ that is orthogonal to all coboundaries.
Thus $\bar{H}^{n-1}(\Gamma,V_\omega)\neq 0$. 
An application of Lemma~\ref{lem: non-separable reps} implies a similar non-vanishing for a separable subrepresentation of the ultrapower, 
contradicting the assumption that $\Gamma$ has $(T_{n-1})$.
\end{proof}

\subsection{The Shapiro Lemma for cocompact lattices} \label{subsec:SL}

The Shapiro Lemma relates the cohomology of a group and its closed subgroup.
In this subsection we review it in the setting where the subgroup is a cocompact lattice and discuss consequences of the results of the previous subsection. 
We will discuss the Shapiro Lemma again, in a more general setting, in \S\ref{subsec:Shapiro++}.

For a unimodular locally compact second countable group $G$, a discrete subgroup $\Gamma<G$ and unitary representation $U$ of $\Gamma$, the \emph{local unitary induction} $I^2_{\loc}(U)$ is defined
as the subspace of $L^2_{\loc}(G,U)$ that consists of functions $\phi$ such that for a.e~$g\in G$ and for every $\gamma\in \Gamma$, $\phi(g\gamma^{-1})=\gamma\phi(g)$.
On $I^2_{\loc}(U)$ the group $G$ acts by $(g\phi)(\_)=\phi(g\_)$. 
For $\phi\in I^2_{\loc}(U)$, the function $\|\phi(\_)\|_U$ yields a well defined function on $G/\Gamma$ and hence a norm $\|\phi\|_{I^2(U)}=\int_{G/\Gamma} \|\phi(\cdot)\|_U$. 
We define the \emph{unitary induction} 
as 
\[ I^2(U)=\bigl\{\phi\in I^2_{\loc}(U)\mid \|\phi\|_{I^2(U)}<\infty\bigr\}.\]
The unitary induction $I^2(U)$ is a Hilbert space. Its $G$-action is a continuous unitary representation of $G$. If $\Gamma<G$ is a cocompact lattice, then $I^2(U)=I^2_{\loc}(U)$~\cite{Blanc}*{Lemme~8.8}. 

Hereafter, the symbol $\bar{\otimes}$ denotes the Hilbertian tensor product.

\begin{lemma}[Shapiro Lemma] \label{lem:Shapiro}
Let $G$ be a locally compact second countable and let $\Gamma<G$ be a cocompact lattice. 
    Given a unitary representation $U$ of $\Gamma$, we have
\[ H^*_c(\Gamma,U)\cong H^*_c(G,I^2(U)), \quad \bar{H}^*_c(\Gamma,U)\cong \bar{H}^*_c(G,I^2(U)).\]
If $U$ is a $G$-representation, then the induction of its restriction to $\Gamma$ satisfies $I^2(U) \cong U \bar{\otimes} L^2(G/\Gamma)$. In particular,  
\[ H^*_c(\Gamma,U)\cong H^*_c(G,U) \oplus H^*_c(G,U\bar{\otimes} L^2_0(G/\Gamma)), \]
\[ \bar{H}^*_c(\Gamma,U)\cong \bar{H}^*_c(G,U) \oplus \bar{H}^*_c(G,U\bar{\otimes} L^2_0(G/\Gamma)).\]
The corresponding inclusion  map $H^*_c(G,U) \to H^*_c(\Gamma,U)$ is induced by restriction, which is thus injective. Therefore, for every $j$, the restriction map $H^j_c(G,U) \to H^j_c(\Gamma,U)$ is an isomorphism if and only if ${H}^j_c(G,U\bar{\otimes} L^2_0(G/\Gamma))=0$.
\end{lemma}

\begin{proof}
    The first statement is~\cite{Blanc}*{Théorème 8.8}. The isomorphism $U\bar\otimes L^2(G/\Gamma)\to I^2(U)$, where $U$ is the $\Gamma$-restriction of a unitary $G$-representation, is induced by $f\otimes u\mapsto \phi$ with $\phi(g)=f(g)g^{-1}u$. The compatibility with the restriction homomorphism is a consequence of the proof of~\cite{Blanc}*{Théorème 8.7}. In~\eqref{eq: Shapiro} we review the Shapiro homomorphism in order to generalize it to certain non-uniform lattices. The compatibility with the restriction homomorphism is evident from the explicit description in~\eqref{eq: Shapiro} as well. 
\end{proof}

\begin{lemma} \label{lem:Tninduction}
     Let $G$ be a locally compact second countable group and let $\Gamma<G$ be a cocompact lattice.     
    Then $G$ has $[T_n]$ if and only if $\Gamma$ has $[T_n]$.
    If $G$ has $(T_n)$ then $\Gamma$ has $(T_n)$.
    If $\Gamma$ has $(T_n)$ and for every $G$-representation $U$ we have $U^\Gamma=U^G$
    then $G$ has $(T_n)$.
\end{lemma}

\begin{proof} \label{lem:Tnfi}
    If $G$ has $[T_n]$, then $\Gamma$ has $[T_n]$ by the Shapiro isomorphism~\ref{lem:Shapiro}. 
    If $\Gamma$ has $[T_n]$, then $G$ has $[T_n]$ by the injectivity of the restriction map in cohomology~\ref{lem:Shapiro}.
    Similarly, if $G$ has $(T_n)$, then $\Gamma$ has $(T_n)$ since for every $U$ with $U^\Gamma=0$ we have $I^2(U)^G=0$. 
    The last statement follows by the injectivity of the restriction map.
\end{proof}

\begin{lemma} \label{lem:GammaG criterion}
    Let $G$ be a locally compact second countable and let $\Gamma<G$ be a cocompact lattice.
    Assume that $\Gamma$ has property $\FP_n(\bbQ)$. Then we have:
    \begin{enumerate}[a)]
    \item If 
    $\bar{H}^k_c(G,U)=0$ for every $0<k\leq n$ and every irreducible $G$-representation~$U$, then $G$ and $\Gamma$ have property $[T_n]$. 
    \item If $G$ has a compact abelianization     
     and $\bar{H}^k_c(G,U)=0$ for every $0<k\leq n$ and every non-trivial irreducible $G$-representation $U$,
     then $\Gamma$ has $(T_n)$. 
    Furthermore, if for every $G$-representation $U$ we have $U^\Gamma=U^G$,
    then also $G$ has $(T_n)$.
    \end{enumerate}
\end{lemma}

\begin{proof}
Assume that $\bar{H}^k_c(G,U)=0$ for every $0<k\leq n$ and for every irreducible $G$-representation~$U$.     
    By Lemma~\ref{lem:directint}, $\bar{H}^k_c(G,U)=0$ for every $G$-representation $U$ of $G$. 
    Then by the Shapiro Lemma~\ref{lem:Shapiro},  $\bar{H}^k(\Gamma,V)=0$ 
    for every $\Gamma$-representation~$V$.
    
    We conclude by the first part of Theorem~\ref{thm:gensh} that $\Gamma$ has $[T_n]$
    and deduce form Lemma~\ref{lem:Tninduction} that also $G$ has $[T_n]$.

Assume that $G$ has a compact abelianization     
     and that $\bar{H}^k_c(G,U)=0$ for every $0<k\leq n$ and for every non-trivial irreducible $G$-representation~$U$. 
    
    The group $\Gamma$ is finitely generated as it has $\FP_1(\bbQ)$. Hence $G$ is compactly generated.
    Since $G$ has a compact abelianization, we have $H^1_c(G,\bbC)=0$. By Lemma~\ref{lem:directint} we get that for every representation $U$ of $G$,
    $\bar{H}^1_c(G,U)=0$.
    It follows by \cite[Theorem 6.1]{Shalom} that $G$ has T, thus also $\Gamma$ has T.
    In particular, $\Gamma$ has a finite abelianization.
    Using Lemma~\ref{lem:directint} again, we get that for every representation $U$ of $G$ with $U^G=0$,
    $\bar{H}^1_c(G,U)=0$ and it follows from the Shapiro Lemma~\ref{lem:Shapiro} that 
    for every representation $V$ of $\Gamma$ with $V^\Gamma=0$,
    $\bar{H}^1_c(\Gamma,V)=0$.
    We conclude by the second part of Theorem~\ref{thm:gensh} that $\Gamma$ has $(T_n)$
    and that if for every $G$-representation $U$, $U^\Gamma=U^G$, then $G$ has $(T_n)$.
\end{proof}

\subsection{The invariant $r_0$ and semisimple groups}

\begin{definition} \label{def:r}
    Given a locally compact second countable group $G$, we say that 
    an irreducible unitary representation $V$ of $G$ is \emph{cohomological} if $H^\ast_c(G,V)\neq 0$.
    The \emph{cohomological dual} of $G$, denoted $\widehat{G}_{\cohom}$, is the subset of the unitary dual of $G$, $\widehat{G}$, consisting of (classes of) cohomological representations. 
    Note that the trivial representation ${\bf 1} \in \widehat{G}_{\cohom}$, as $H^0_c(G,{\bf 1})\cong \bbC$.
    We denote by $r_0(G)$ the minimal degree in which some non-trivial irreducible cohomological representation has a non-vanishing cohomology, that is
    \[ r_0(G)=\min \{r \mid \exists ~{\bf 1} \neq V \in \widehat{G}_{\cohom},~H^r_c(G,V)\neq 0  \}.\]
    In particular, $r_0(G)=\infty$ if and only if  $\widehat{G}_{\cohom}=\{{\bf 1}\}$.
    We denote by $r(G)$ the minimal degree in which some irreducible cohomological representation which has a compact kernel has a non-vanishing cohomology.
\end{definition}

For a compact group $G$, we have by Lemma~\ref{lem:compact}, $\widehat{G}_{\cohom}=\{{\bf 1}\}$, $r(G)=0$ and $r_0(G)=\infty$.
However, if $G$ is not compact, we clearly have $r(G)\geq r_0(G)$.
This will typically be the case in our considerations.
In \S\ref{subsec:cohomreps} and Appendix~\ref{sec:rG} we will explain how to compute explicitly the invariants $r_0(G)$ and $r(G)$ for semisimple groups. 

\begin{theorem} \label{thm:r0}
    Let $G \cong \prod_{i=1}^l \bfG_i(F_i)$ where each $F_i$ is a local field and each $\bfG_i$ is a connected, simply connected almost simple $F_i$-group.
    Then $G$ and all cocompact lattices in $G$ satisfy property $(T_{r_0(G)-1})$.
    Further, if $r_0(G)\geq 2$ then their degree $j$ cohomologies with unitary coefficients are always Hausdorff, for every $j\leq r_0(G)$.
\end{theorem}

\begin{proof}
By \cite{Borel-Harder} there exists a cocompact lattice in $G$.
Fix such a compact lattice $\Gamma$ and note that it has property $\FP$.
    By the Howe-Moore theorem, we have for every $G$-representation $U$, $U^\Gamma=U^G$.
    Also, $G$ is locally compact second countable and has a compact abelianization.
    Thus $G$ and $\Gamma$ satisfy property $(T_{r_0(G)-1})$ by Lemma~\ref{lem:GammaG criterion}.

    Assume now $r_0(G)\geq 2$.
    For every cocompact lattice $\Gamma<G$ and for every unitary $\Gamma$-representation $U$, $H^{r_0(G)}(\Gamma,U)$ is Hausdorff by Theorem~\ref{thm:red} and for $j<r_0(G)$,
    \[ H^{j}(\Gamma,U) \cong H^{j}(\Gamma,U^\Gamma) \cong U^\Gamma \otimes H^{j}(\Gamma,\bbC),\]
    and $H^{j}(\Gamma,\bbC)$ is finite dimensional by property $\FP$, so $H^{j}(\Gamma,U)$ is Hausdorff.
    Finally, for every unitary $G$-representation $V$ and for $j\leq r_0(G)$,
    we have by Shapiro Lemma~\ref{lem:Shapiro} a continuous injection $H_c^{j}(G,V) \hookrightarrow H^{j(G)}(\Gamma,V)$, thus $H_c^{j}(G,V)$ is Hausdorff.
\end{proof}

\begin{remark} \label{rem:p-adicT}
    In case all the fields $F_i$ in the setting of Theorem~\ref{thm:r0} are non-archimedean, that is $G$ is totally disconnected, the assumption that the groups $\bfG_i$ are simply connected could be removed.
    Indeed, this assumption was used only via the use of the Howe-Moore theorem, giving rise to the equation $U^\Gamma=U^G$, which is not needed in the totally disconnected case. The result in this case follows also from the work of Dymara-Januskiewicz~\cite{dymara+janus} which deals with the broader context of automorphism groups of buildings. 
\end{remark}

\begin{theorem} \label{thm:Lier0}
    Let $G$ be a connected semisimple Lie group with a finite center.
    Then $G$ and all cocompact lattices in $G$ satisfy property $(T_{r_0(G)-1})$.
    Furthermore, if $r_0(G)\geq 2$ and $j\leq r_0(G)$, then $j$-th cohomology of~$G$  with unitary coefficients is always Hausdorff.
\end{theorem}

\begin{proof}
By Lemma~\ref{lem: product with compact group} we assume as we may that $G$ is center free and with no compact factors.
We identify $G$ with the identity component of the Lie group $\bfG(\bbR)$, where $\bfG$ is the identity component of the group of Lie algebra automorphism of the Lie algebra of $G$, thus a connected semisimple $\bbR$-algebraic group.
We consider the simply connected cover $\tilde{\bfG}\to \bfG$, set $\tilde{G}=\tilde{\bfG}(\bbR)$ and note that $G$ is the image of $\tilde{G}$ under the natural map $\tilde{\bfG}(\bbR)\to {\bfG}(\bbR)$.
By Theorem~\ref{thm:r0}, we have that $\tilde{G}$ satisfy property $(T_{r_0(G)-1})$ and by using again Lemma~\ref{lem: product with compact group} we conclude that ${G}$ satisfy property $(T_{r_0(G)-1})$.
It follows by Lemma~\ref{lem:Tninduction} that cocompact lattices in $G$ also satisfy property $(T_{r_0(G)-1})$.

The last part of the proof is verbatim the last paragraph of the proof of Theorem~\ref{thm:r0}.
\end{proof}

\section{Spectral gap properties of simple groups} \label{sec:SG}

The main purpose of this section is to discuss 
spectral gap properties of unitary $G$-representations of the form $L_0^2(G/\Gamma)\bar\otimes V$, where $V$ is a unitary $G$-representation and $\Gamma<G$ a lattice, with respect to every simple factor of~$G$. In higher rank situations, where every simple factor has property T, this is automatic, and the deep background results addressed in this section can be avoided. However, they are essential if the algebraic group $\bfG$ has $k$-rank $1$ or the semisimple Lie group~$G$ is a product of rank~$1$ factors in the main theorems in the introduction.  

In \S\ref{subsec: unitary dual} we recall some basics and set the terminology about the unitary dual. In~\S\ref{subsec: integrability} 
we discuss the relationship between $L^p$-integrability of matrix coefficients of unitary representations of algebraic or arithmetic groups and the spectral gap property. In~\S\ref{subsec: automorphic reps} we discuss and prove the spectral gap property in the case where $G$ is an adelic group. To this end, we prove some functoriality properties of the automorphic dual with respect to group homomorphisms. Finally, in~\S\ref{subsec: consequences lie} we draw some conclusions about the spectral gap property for semisimple groups from the previous results, which were obtained in the algebraic setting, by means of the Margulis' arithmeticity theorem.

\subsection{The unitary dual}\label{subsec: unitary dual}
 We direct the reader to \cite[\S7]{Folland} for an introduction to the unitary dual. We review some basic facts in order to fix the terminology. 
  
Let $G$ be a second countable locally compact topological group $G$. We denote by $\widehat{G}$ the \emph{unitary dual} of $G$, that is, the set of isomorphism classes of irreducible unitary $G$-representations. In an appropriate context, we may regard elements of $\widehat{G}$ as actual representations. We endow $\widehat{G}$ with the Fell topology.
The group $G$ is of \emph{type I} if the topology on $\widehat{G}$ separates the points, i.e. it is $T_0$. 
It is a fundamental fact that for type~I groups the Borel-Mackey structure on $\widehat{G}$ coincides with the Borel $\sigma$-algebra associated with the Fell topology, and this $\sigma$-algebra is standard.

Every unitary representation $\rho$ of $G$ decomposes as a direct integral over $\widehat{G}$ with respect to some measure class on $\widehat G$. 
The support of this measure class is the closure of irreducible representations that are weakly contained in~$V$. For unitary $G$-representations 
$\rho$ and $\pi$, weak containment $\rho\prec\pi$ is equivalent to the operator norm inequality $\norm{\rho(f)}\le \norm{\pi(f)}$ for every $f\in L^1(G)$. 
We say $\rho$ is \emph{weakly contained} in a subset $S$ of (not necessarily irreducible) unitary $G$-representations  (write $\rho\prec S$) if $\norm{\rho(f)}\le \sup_{\pi\in S}\norm{\pi(f)}$ for every $f\in L^1(G)$. The closure of a subset $S\subset\hat G$ in the Fell topology consists of all irreducible representations weakly contained in~$S$. 
A subset $S\subset \widehat G$ is an \emph{ideal} if 
$\pi\otimes\sigma$ is weakly contained in $S$ 
for every $\pi\in S$ and every unitary $G$-representation $\sigma$. 

\begin{definition}
For $X\subset \widehat{G}$, we denote by  $\overline{X}^\otimes$ the smallest closed ideal that contains~$X$. 
\end{definition}
%
Explicitly, we have 
\begin{equation}\label{eq: closed ideal explicit} \overline{X}^\otimes=\bigl\{\pi\in \widehat G\mid \pi\prec\{\rho\otimes\sigma\mid \rho\in X,~\sigma\text{ any unitary $G$-representation}\} \bigr\}.\end{equation}
The unitary dual is not (contravariantly) functorial with regard to a continuous group homomorphism $f\colon G\to H$ since the pullback $f^\ast\rho$ of an irreducible representation $\rho$ might not be irreducible. We define the \emph{pullback} of a subset $Y\in \widehat G$ under~$f$ as 
\[ f^\ast (Y)=\bigl\{ \pi\in \widehat G\mid \pi\prec \{f^\ast\rho\mid \rho\in Y\}\bigr\}.\]
The pullback $f^\ast Y$ is automatically closed. Moreover, the next inclusion follows directly from the definition. 

\begin{lemma}\label{lem: continuity of pullback}
We have $f^\ast(\overline{Y}^\otimes)\subset \overline{f^\ast(Y)}^\otimes$. 
\end{lemma}

\subsection{Integrability of matrix coefficients}\label{subsec: integrability}

In this subsection, let $F$ be a local field, and let $\bfG$ be an almost simple, connected and simply connected $F$-algebraic group. 
We set $G=\bfG(F)$ and assume it is not compact.
The group $G$ is type~I by work of 
Harish-Chandra~\cite{harishchandra} in the archimedean and by work of Bernstein~\cite{Bernstein} in the non-archimedean case. 
The group $G$ has the \emph{Kunze-Stein property}, that is,  for every $1 \leq p<2$,
\[L^p(G)* L^2(G) \subseteq L^2(G).\] 
This fundamental fact was proved in the archimedean case by Cowling~\cite{Cowling} and in the non-archimedean case 
by Veca~\cite{Veca}.

\begin{definition}
Let $2\leq p \leq \infty$. A unitary $G$-representation $U$ is a \emph{$p$-type representation} if for every $u,v\in U$ and for every $q>p$ the matrix coefficient function $\langle gu,v \rangle$ is in $L^q(G)$.  
We denote by $\widehat{G}_p\subset\widehat G$ the subset of $p$-type representations.
\end{definition}
Note that $\widehat{G}_\infty=\widehat{G}$ and $\widehat{G}_2$ consists of the \emph{tempered representations}, that is, the ones that are weakly contained in the regular representation $L^2(G)$.

The following theorem is taken from \cite{Tim+Timo}. Its proof is based on a result by  
\cite{Samei-Wiersma}.

\begin{theorem} \label{thm:Samei-Wiersma}
    Let $F$ be a local field. Let $\bfG$ be an almost simple, connected and simply connected $F$-algebraic group and denote $G=\bfG(F)$.
    Then $\widehat{G}_p$ is a closed ideal of $\widehat{G}$ for every $2\leq p \leq \infty$.
\end{theorem}

\begin{proof}
    This is \cite[Theorem 1.1]{Tim+Timo}; we only need to clarify that the space $\widehat{G}_{L^{p+}}$ considered in \emph{loc. cit.} coincides with  $\widehat{G}_p$ by~\cite[Proposition 4.2]{Tim+Timo}.
\end{proof}

\begin{theorem} \label{Cowling}
Let $F$ be a local field of characteristic 0 and let $\bfG$ be an almost simple, connected and simply connected $F$-algebraic group. 
Assume that $G=\bfG(F)$ is not compact.
In case $F$ is non-archimedean and $\rank_F(\bfG)=1$, we additionally assume that $\bfG$ is $F$-isomorphic to the $F$-group $\SL_2$ or to the $F$-group $\SU_3$ associated with a quadratic extension of $F$. 

Then for every closed subset $X\subset \widehat{G}$, ${\bf 1}\notin X$ if and only if there exists $2\leq p < \infty$ such that $X\subset \widehat{G}_p$.
In particular, ${\bf 1}\notin X$ if and only if ${\bf 1} \notin \overline{X}^\otimes$.
\end{theorem}

If $X=\{\pi\}\subset\widehat G$ then, in particular, 
$1\not\prec \pi$ implies that $1\not\prec \pi\otimes \sigma$ for every unitary $G$-representation. 

\begin{conjecture}
    The conditions regarding non-archimedean rank 1 groups in Theorem~\ref{Cowling} could be dropped.
\end{conjecture}

\begin{proof}
Since $G$ is not compact, ${\bf 1} \notin \widehat{G}_p$ for every $2\leq p < \infty$. 
Thus, for a closed subset $Y\subset \widehat{G}$, if $Y\subset \widehat{G}_p$ then ${\bf 1}\notin Y$.
The converse is less trivial.

If $\rank_F(\bfG)\geq 2$ and $F$ is archimedean, this is~\cite[Theorem 2.4.2]{Cowling79} and the general higher rank case was proved in~\cite{Oh}.
If $\rank_F(\bfG)=1$ and $F$ is archimedean, this is proved in~\cite[Theorem 2.5.2]{Cowling79}. See also the discussion in~\cite[\S5]{Tim+Timo}.
Hence it suffices to treat from now on the case that $\rank_F(\bfG)=1$ and $F$ is non-archimedean, in which case we additionally assumed that 
$\bfG$ is $\SL_2$ or $\SU(3)$ associated with a quadratic extension of~$F$. 

The Langlands classification associates with every non-trivial irreducible unitary representation an \emph{irreducible admissible representation}.
These are classified by means of \emph{Langlands parameters}, 
where to every irreducible admissible representation one associates a standard parabolic subgroup and a tempered representation of its Levi subgroup, see \cite[XI \S2]{Borel-Wallach}. Given this classification, a main difficulty in the description of the unitary dual is in determining which admissible representations are unitarizable.

By assumption, $G=\SL_2$ or $\SU(3)$.
In these cases, the unitarizability problem was solved completely in~\cite{Keys} 
from which we adopt the terminology used below. 
We may and will disregard tempered representations, as ${\bf 1} \notin \widehat{G}_2$.
Thus all representations considered below are parabolically induced.
Since $G$ is rank 1, we have a unique conjugacy class of proper parabolic subgroup.
We fix a proper parabolic $P=MN<G$ with unipotent radical $N$ and Levi subgroup~$M$.
Note that $G$ is quasi split of rank~1, thus $M$ is abelian and it is compact modulo a one dimensional torus. 
We fix a unitary character $\lambda$ of $M$ and for every $s\in \bbC$ we consider the character $\lambda_s$, as defined in \cite[\S1]{Keys}.
The corresponding induced representation is denoted $\Ind_P^G(\lambda_s)$.
It is unitarizable only if $s$ is purely imaginary or real, in which cases we say that the associated induced representation is in the \emph{principal series} or the \emph{complementary series} correspondingly.
We care only about the complementary series, as the principal series representations are all tempered. 
Thus we assume below that the parameter $s$ is real.
By \cite[XI, Proposition 3.6]{Borel-Wallach}, for all non-trivial representations the matrix coefficients are in $L^p$, and it follows from the consideration there that $p$ is uniformly bounded by means of $s$, independently of $\lambda$.

The general structure of the complementary series for quasi-split groups and corank 1 parabolics described by Shaidi in \cite[Theorem 8.1]{Shahidi} is applicable in our situation.
It follows that the complementary series associated with a unitary character $\lambda$ is given by a symmetric interval $[-\nu,\nu]$, where $\nu$ equals either $1/2$ or $1$,
and these values are determined by the zeros of certain functions~\cite[\S6]{Keys}.
These functions are computed explicitly in~\cite[\S5, Theorem]{Keys}.
They have a zero at $1/2$ unless the character $\lambda$ is unramified, in which case it has a zero at $1$. That is, the matrix coefficients of $\Ind_P^G(\lambda_s)$ are in $L^p$ for some fixed $p$, unless the representations are spherical, that is, they have a non-trivial vector fixed by the maximal compact subgroup $K<G$.

In the latter case, it follows as in~\cite[XI, Proposition 3.6]{Borel-Wallach} that the matrix coefficients are in $L^p$ if and only if the parameter $s$ is bounded away form 1, which exactly means that the representations do not have almost invariant vectors. 
Alternatively, this can be seen by the explicit computation of the \emph{spherical functions},
i.e the matrix coefficients for $K$-invariant vectors, given in \cite[Chapter II]{FTN}.
We deduce that, indeed, if ${\bf 1}\notin X$ then there exists $2\leq p < \infty$ such that $X\subset \widehat{G}_p$.

Finally, the ``in particular" part follows by Theorem~\ref{thm:Samei-Wiersma}.
\end{proof}

\subsection{On automorphic representations}
\label{subsec: automorphic reps}
Let $k$ be a number field and let $\bbA(k)$ be the ring of adeles of~$k$.
Let $\bfG$ be a connected almost simple  $k$-algebraic group.
Let $s$ be a place of $k$, set $G=\bfG(k_s)$ and consider $L^2\bigl(\bfG(\bbA(k)/\bfG(k)\bigr)$ as a unitary $G$-representation.
In this setting, the \emph{automorphic dual} of $G$ is defined as 
\[\widehat{G}_{\aut}=\Bigl\{\pi\in\widehat G\mid \pi\prec L^2\bigl(\bfG(\bbA(k)/\bfG(k)\bigr)\Bigr\}.\] 
Note that it depends on $k,\bfG$ and $s$, which are implicit in our notation.

If $G$ has property T, which is the case when $\bfG$ has no $k_s$-rank~1 factors, then $\{{\bf 1}\}$ is open in $\widehat{G}$, hence also in $\widehat{G}_{\aut}$. 
We are interested in generalizations of this property that are valid unconditionally.

The following theorem was proved for $k=\bbQ$ and $\bfG=\SL_2$ by Selberg in case $k_s=\bbR$ and by Gelbart-Jaquet~\cite{Gelbart-Jaquet} for the non-archimedean places.
Finally, using the a reduction procedure due to Burger-Sarnak~\ref{Burger-Sarnak}, Clozel proved the following.

\begin{theorem}[{\cite[Theorem 3.1]{Clozel}}] \label{Clozel}
Let $k$ be a number field and let $\bfG$ be connected and simply connected almost simple $k$-algebraic group.
Let $s$ be a place of $k$ and consider $G=\bfG(k_s)$.
Then $\{{\bf 1}\}$ is open in $\widehat{G}_{\aut}$. 
    \end{theorem}

Theorem~\ref{Clozel} is equivalent to ${\bf 1} \notin \overline{\widehat{G}_{\aut}\setminus \{{\bf 1}\}}$
and it is trivial in case $G$ is compact.
The main result of this section is the following generalization.

\begin{theorem} \label{Clozel++}
Let $k$ be a number field and let $\bfG$ be connected and simply connected almost simple $k$-algebraic group.
Let $s$ be a place of $k$ and assume $G=\bfG(k_s)$ is non-compact.
Then we have ${\bf 1} \notin \overline{\widehat{G}_{\aut}\setminus\{{\bf 1}\}}^\otimes$.
\end{theorem}

In most cases, Theorem~\ref{Clozel++} follows immediately from Theorem~\ref{Clozel} using Theorem~\ref{Cowling}.
The cases not handled by Theorem~\ref{Cowling} could be treated by a combination of the following two results. 
The first one is an elaboration on a result which was proven by Burger-Sarnak~\cite{Burger-Sarnak} in the archimedean case and extended by Clozel-Ullmo~\cite{Clozel-Ullmo} to the non-archimedean case.

\begin{theorem} \label{Burger-Sarnak}
Let $k$ be a number field and let $\bfG$ and $\bfH$ be connected and simply connected almost simple $k$-algebraic groups.
Assume $\bfG$ is absolutely simple.
Let $s$ be a place of $k$ and consider $G=\bfG(k_s)$ and $H=\bfH(k_s)$.
If $\phi\colon\bfH \to \bfG$ is a $k$-morphism with central kernel, then the pullback under $\phi_{k_s}\colon H\to G$ of $\widehat{G}_{\aut}\bs\{{\bf 1}\}$ is contained in $\widehat{H}_{\aut}\setminus\{{\bf 1}\}$.
\end{theorem}

\begin{proof}
We first claim that the pull back under $\phi_{k_s}\colon H\to G$ of $\widehat{G}_{\aut}$ is contained in $\widehat{H}_{\aut}$.
We denote the image of $\phi$ in $\bfG$ by $\bfH'$.
This is an almost simple, connected $k$-algebraic, which need not be simply connected.
We set $H'=\bfH'(k_s)$.
By \cite[Theorem 3.3]{Clozel}, the pull back of $\widehat{G}_{\aut}$ under the inclusion $H' \to G$ is contained in $\widehat{H'}_{\aut}$.
We regard $\phi$ as a map $\bfH \to \bfH'$, thus we are left to show that the pull back under $\phi_{k_s}\colon H\to G$ of $\widehat{H'}_{\aut}$ is contained in $\widehat{H}_{\aut}$. 

Since $\phi:\bfH \to \bfH'$ is a $k$-isogeny, we have that $\phi_{\bbA(k)}:\bfH(\bbA(k))\to \bfH'(\bbA(k))$ has a compact abelian kernel, which we denote by $C$, and an open image.
Note also that the commutator map $\bfH' \times \bfH' \to \bfH'$ factors via $\bfH$ and conclude that $\phi_{\bbA(k)}(\bfH(\bbA(k)))$ is an open subgroup of $\bfH'(\bbA(k))$ which contains the commutator, hence it is normal and $\phi_{\bbA(k)}$ has an abelian discrete cokernel. 
Since $\bfH'(\bbA(k))/\bfH'(k)$ has a finite measure, its quotient 
\[ D=\bfH'(\bbA(k))/(\bfH'(k)\cdot \phi_{\bbA(k)}(\bfH(\bbA(k))))\]
must be a finite abelian group.
    We thus have a Borel $\bfH(\bbA(k))$-equivariant isomorphism
    \[ \bfH'(\bbA(k))/\bfH'(k) \cong D \times \bfH(\bbA(k))/(\bfH(k)\cdot C), \]
    where the action on $D$ is trivial, 
    and we get a corresponding unitary $\bfH(\bbA(k))$-isomorphism 
    \begin{equation} \label{eq:adelicSC}
        L^2(\bfH'(\bbA(k))/\bfH'(k)) \cong L^2(D) \otimes L^2(\bfH(\bbA(k))/\bfH(k))^C, 
    \end{equation}
    where the action on $L^2(D)$ is trivial.
    This proves the claim, as the representation on the right hand side is clearly supported on $\widehat{H}_{\aut}$.

    We are left to show that the pullback in $\widehat{H}_{\aut}$ of $\widehat{G}_{\aut}\setminus\{{\bf 1}\}$ does not contain~$\bf 1$. 
    Let $\pi$ be a unitary $G$-representation 
    in $\widehat{G}_{\aut}\setminus\{{\bf 1}\}$. By the strong approximation theorem~\cite{Margulis}*{II, Theorem 6.8}, $\pi$ is ergodic. By the Howe-Moore theorem, $\pi$ is mixing. 
    Since the image of $H$ is non-compact in $G$, we conclude that the inflation $\phi_{k_s}^\ast\pi$ to~$H$ is ergodic. 
    Hence $\{\phi_{k_s}^\ast\pi\mid \pi\in \widehat{G}_{\aut}\setminus\{{\bf 1}\}\}\subset \widehat{H}_{\aut}\setminus\{{\bf 1}\}$. Since 
    $\widehat{H}_{\aut}\setminus\{{\bf 1}\}$ is closed by Theorem~\ref{Clozel} and the $\phi_{k_s}$-pullback of $\widehat{G}_{\aut}\setminus\{{\bf 1}\}$ is the closure 
    of $\{\phi_{k_s}^\ast\pi\mid \pi\in \widehat{G}_{\aut}\setminus\{{\bf 1}\}\}$, the proof is finished. 
\end{proof}

The second result was proven by Clozel in the course of the proof of Theorem~\ref{Clozel}.

\begin{theorem} \label{Clozel-reduction}
Let $k$ be a number field and let $\bfG$ be a connected and simply connected absolutely almost simple $k$-algebraic groups.
Let $s$ be a non-archimedean place of $k$ and assume $\rank_{k_s}(\bfG)=1$.
    Then there exists an almost simple, connected and simply connected $k$-algebraic group $\bfH$ 
    and a morphism of $k$-algebraic groups $\phi:\bfH\to \bfG$ which has a central kernel,
    such that the extension of scalars of $\bfH$ to $k_s$ is isomorphic to either $\SL_2$ or the the quasi split group $\SU_3$ associated with the unique unramified quadratic extension of~$k_s$.
\end{theorem}

\begin{proof}
If $\bfG$ is $k$-isotropic then the theorem holds with $\bfH=\SL_2$ by Jacobson-Morozov,
we thus assume that $\bfG$ is anisotropic.
In this case the group $\bfH$ could be constructed as a $k$-subgroup of $\bfG$.
This is done by a case by case reduction in the fourth to sixth paragraphs of \cite[\S3.2]{Clozel}, based on \cite[Theorem 1.1]{Clozel} and its proof.
\end{proof}

\begin{remark}
    According to \cite[Remark 1.2]{Clozel}, it is expected that $\bfH$ in Theorem~\ref{Clozel-reduction} could be constructed as a $k$-subgroup of $\bfG$ also in case $\bfG$ is isotropic, but the proof is not known.
\end{remark}

We can now prove Theorem~\ref{Clozel++}, based on Theorem~\ref{Cowling}.

\begin{proof}[Proof of Theorem~\ref{Clozel++}]
As in the first paragraph of \cite[\S3.2]{Clozel}, we may assume that $\bfG$ is $k$-absolutely almost simple, by passing to a finite extension and forming a restriction of scalars~\cite[\S3.1.2]{Tits}.
The restriction of scalars is compatible with the structure of the adelic group as a topological group by~\cite[Example 4.2]{Conrad}.
    We note that $\widehat{G}_{\aut}\setminus\{{\bf 1}\}$ is closed by Theorem~\ref{Clozel}.
    By Theorem~\ref{Cowling} it suffices to consider the case that $s$ is non-archimedean and $\bfG$ is of $k_s$-rank 1. 

    Using Theorem~\ref{Clozel-reduction} we 
    fix an almost simple, connected $k$-algebraic group $\bfH$ 
    and a morphism of $k$-algebraic groups $\phi:\bfH\to \bfG$ which has a central kernel,
    such that the extension of scalars of $\bfH$ to $k_s$ is isomorphic to either $\SL_2$ or the the quasi split group $\SU_3$ with respect to the unique unramified quadratic extension of $k_s$.
By Theorem~\ref{Burger-Sarnak}, the pullback under $\phi_{k_s}:H\to G$ of $\widehat{G}_{\aut}\setminus\{{\bf 1}\}$ is contained in $\widehat{H}_{\aut}\setminus\{{\bf 1}\}$, which is closed by Theorem~
\ref{Clozel}. 
By Theorem~\ref{Cowling}, ${\bf 1} \notin \overline{\widehat{H}_{\aut}\setminus\{{\bf 1}\}}^{\otimes}$.
By Lemma~\ref{lem: continuity of pullback} we deduce that \[\phi_{k_s}^\ast\Bigl(\overline{\widehat{G}_{\aut}\setminus\{{\bf 1}\}}^{\otimes}\Bigr)\subset \overline{\widehat{H}_{\aut}\setminus\{{\bf 1}\}}^{\otimes}.\] 
This implies that 
${\bf 1} \notin \overline{\widehat{G}_{\aut}\setminus\{{\bf 1}\}}^{\otimes}$.
\end{proof}

The following is an immediate corollary of Theorem~\ref{Clozel++}.

\begin{cor} \label{cor:Clozel++}
Let $k$ be a number field and let $\bfG$ be a connected and simply connected almost simple $k$-algebraic groups.
Let $s$ be a place of $k$ such that $G=\bfG(k_s)$ is non-compact.
For every unitary $G$-representation $U$, the unitary $G$-representation  
    $L^2_0(\bfG(\bbA(k))/\bfG(k))\bar{\otimes} U$ does not have almost invariant vectors. 
\end{cor}

\subsection{Consequences regarding spectral gap}\label{subsec: consequences lie}

Recall the definition of $S$-arithmetic lattices given in \S\ref{sec: setup}.
The following theorem is a version of Corollary~\ref{cor:Clozel++} which is given in this setting. 

\begin{theorem} \label{thm:Clozel++arith}
Let $k$ be a number field and let $\bfG$ be an almost simple, connected and simply connected $k$-algebraic group.
Let $S$ be a set of places of $k$ that includes all the archimedean ones, and let $\Gamma<\bfG(k)$ be an $S$-arithmetic subgroup.
Let $G$ be the restricted product of the groups $\bfG(k_s)$, $s\in S$,
and consider $\Gamma$ as a lattice in $G$.
Fix $t\in S$ such that $G_t=\bfG(k_t)$ is non-compact
and let $U$ be a unitary $G_t$-representation $U$.
Then the unitary $G_t$-representation  
    $L^2_0(G/\Gamma)\bar{\otimes} U$ does not have almost invariant vectors. 
\end{theorem} 

\begin{proof}
As in the first paragraph of \cite[\S3.2]{Clozel}, we may assume that $\bfG$ is $k$-absolutely almost simple, by passing to a finite extension and forming a restriction of scalars~\cite[\S3.1.2]{Tits}.
The restriction of scalars is compatible with the structure of the adelic group as a topological group by~\cite[Example 4.2]{Conrad}.

We let $G'$ be the restricted product of the groups $\bfG(k_s)$, $s\notin S$,
and identify $\bfG(\bbA(k))\cong G\times G'$.
We let $K$ be a compact open subgroup of $G'$ and let $\Lambda$ be the $S$-congruence subgroup of $\bfG(k)$ obtained by intersecting it with $K$.
We regard
$G \times K$ as an open subgroup of $\bfG(\mathbb{A}(k))$.
We view $\Lambda$ as a lattice in it and obtain the identification
\[ (G \times K)/\Lambda \hookrightarrow \bfG(\mathbb{A}(k))/\bfG(k) \]
and accordingly 
\[ L^2\bigl(G/\Lambda\bigr) \cong L^2\bigl((G \times K)/\Lambda\bigr)^K \hookrightarrow L^2\bigl(\bfG(\mathbb{A}(k))/\bfG(k)\bigr), \]
where in the left hand side we regard $\Lambda$ as a lattice in $G$.
We conclude that the $G_t$-representation $L^2(G/\Lambda)$ is supported on $\widehat{(G_t)}_{\aut}$.
By the Strong Approximation Theorem, \cite{Margulis}*{II, Theorem 6.8}, $G_t$ acts ergodically on $G/\Lambda$, thus $L^2_0(G/\Lambda)$ is supported on $\widehat{(G_t)}_{\aut}\setminus\{{\bf 1}\}$.
We conclude by Theorem~\ref{Clozel} that $L^2_0(G/\Lambda)$
does not have $G_t$-almost invariant unit vectors.
Using \cite[Lemma~3.1]{KM99}, we get that also $L^2_0(G/\Gamma)$
does not have $G_t$-almost invariant unit vectors, as $\Gamma$ and $\Lambda$ are commensurable.
If either $t$ is archimedean or $\bfG$ is of higher $k_t$-rank, the theorem now follows by Theorem~\ref{Cowling}.

We now assume that $t$ is non-archimedean and $\rank_{k_t}(\bfG)=1$
and we alter the above proof as in the proof of Theorem~\ref{Clozel++}.
    Using Theorem~\ref{Clozel-reduction} we 
    fix an almost simple, connected $k$-algebraic group $\bfH$ 
    and a morphism of $k$-algebraic groups $\phi:\bfH\to \bfG$ which has a central kernel,
    such that the extension of scalars of $\bfH$ to $k_t$ is isomorphic to either $\SL_2$ or the the quasi-split group $\SU_3$ with respect to the unique unramified quadratic extension of $k_t$.
    We set $H=H(k_{t})$ and conclude by Theorem~\ref{Burger-Sarnak} and Theorem~\ref{Clozel} that $L^2_0(\bfG(\bbA(k))/\bfG(k))$ has no $H$-almost invariant vectors.
    We deduce that its subrepresentation $L^2_0(G/\Lambda)$ has no $H$-almost invariant vectors,
    and by \cite[Lemma~3.1]{KM99} we conclude that $L^2_0(G/\Gamma)$ has no $H$-almost invariant vectors.
    Finally, by Theorem~\ref{Cowling}, we conclude that for every unitary representation $U$ of $G_i$,
The $G_i$-representation $L^2_0(G/\Gamma)\bar{\otimes} U$ does not have $H$-almost invariant vectors.
In particular, it does not have $G_t$-almost invariant vectors. 
\end{proof}

The following theorem is a variant of Theorem~\ref{thm:Clozel++arith} which applies for general lattices, not assuming arithmeticity a priori.
Of course, it relies on arithmeticity a posteriori, via Margulis' Arithmeticity Theorem.

\begin{theorem} \label{thm:specgap}
Let $G$ be a semisimple group, that is $G \cong \prod_{i=1}^l \bfG_i(F_i)$, where each $F_i$ is a local field of characteristic zero and each $\bfG_i$ is a connected almost simple $F_i$-group. We assume that each $\bfG_i$ is simply connected and that the groups $G_i=\bfG_i(F_i)$ are non-compact. 
Let $\Gamma<G$ be an irreducible lattice.
If $G$ has rank 1, we also assume that $G$ is a Lie group.
Then for every $i$ and a unitary representation $U$ of $G_i$,
The unitary $G_i$-representation
$L^2_0(G/\Gamma)\bar{\otimes} U$ does not have almost invariant unit vectors.
\end{theorem}

\begin{proof}
    Assume first that $G$ is a Lie group of rank~1.
    Then $L^2_0(G/\Gamma)$ has a spectral gap by \cite[Lemma~3]{Bekka}
    (cf.~\cite[III, Corollary~1.10 and Remark~1.12]{Margulis}).
    We thus get by Theorem~\ref{Cowling} that for every unitary representation $U$ of $G$,
$L^2_0(G/\Gamma)\bar{\otimes} U$ does not have almost invariant vectors for $G$.
See also \cite[Lemma 4]{Bekka}.

We now assume that $G$ is of higher rank.
We may assume that each $\bfG_i$ is $F_i$-absolutely almost simple, by passing to a finite extension and forming a restriction of scalars, see \cite[\S3.1.2]{Tits}.
By Margulis' Arithmeticity Theorem \cite[IX, Theorem (1.11)]{Margulis}, $\Gamma$ is an arithmetic lattice as defined in \cite[IX, Definition (1.4)]{Margulis}.
Since each $\bfG_i$ is absolutely almost simple and simply connected, by \cite[IX, Remark (1.3)(i)]{Margulis}, this means (by a standard abuse of notation) that 
there exist a number field $k$ over which $\bfG$ is defined and a finite set of places $S$, containing all archimedean places,
such that the fields $F_i$ are the completions $k_s$ for $s\in S$ over which $\bfG$ is isotropic, and $\Gamma$ is $S$-arithmetic.
Denoting by $K$ the product of the groups $\bfG(k_s)$ running over all $s\in S$ such that $\bfG$ is $S$-anisotropic, we identify $\prod_{s\in S}\bfG(k_s) \cong G\times K$ and view $\Gamma$ as a lattice in $G\times K$. 

Let $U$ be a unitary representation of~$G_i$.
By Theorem~\ref{thm:Clozel++arith}, the unitary $G_i$-representation $L^2_0((G\times K)/\Gamma) \bar{\otimes} U$ does not have almost invariant unit vectors.
Identifying 
\[ L^2_0(G/\Gamma) \bar{\otimes} U\cong (L^2_0((G\times K)/\Gamma) \bar{\otimes} U)^K < L^2_0((G\times K)/\Gamma) \bar{\otimes} U, \]
we conclude that the unitary $G_i$-representation $L^2_0(G/\Gamma) \bar{\otimes} U$ 
 does not have almost invariant unit vectors.
\end{proof}

\begin{cor} \label{cor:Liespecgap}
Let $G$ be a connected semisimple Lie group with finite center and with no compact factors.
Let $\Gamma<G$ be an irreducible lattice.
Then for every unitary representation $U$ of $G$,
$L^2_0(G/\Gamma)\bar{\otimes} U$ does not have almost invariant vectors for every factor of $G$.
\end{cor}

\begin{proof}
Using \cite[Lemma~3.1]{KM99} we may assume that $G$ is center free.
We identify $G$ with the identity component of the Lie group $\bfG(\bbR)$, where $\bfG$ is the identity component of the group of Lie algebra automorphism of the Lie algebra of $G$, thus a connected semisimple $\bbR$-algebraic group (cf.~\cite{zimmer}*{Proposition~3.1.6 on p.~35}).
Considering the simply connected cover $\tilde{\bfG}\to \bfG$, setting $\tilde{G}=\tilde{\bfG}(\bbR)$ and letting $\tilde{\Gamma}$ be the preimage of $\Gamma$ in $\tilde{G}$ under the natural map $\tilde{G}\to G$,
we identify $G/\Gamma \cong \tilde{G}/\tilde{\Gamma}$ and the corollary follows  
from Theorem~\ref{thm:specgap}, applied to $\tilde{G}$.
\end{proof}

\section{Cohomology of semisimple groups and their cocompact lattices} \label{sec:CSS}

In this section we study the cohomology with unitary coefficients of semisimple groups and their cocompact lattices.
In \S\ref{subsec:cohomreps} we survey the cohomology of irreducible representations of semisimple groups. This subsection is mostly expository. 
In \S\ref{sec:lowdeg} we prove versions of some of the main theorems in the introduction for uniform  lattices of semisimple groups. These version are not superseded by the main theorems, which hold for non-uniform lattices as well, as they have better bounds on the degrees in some cases.

\subsection{Irreducible cohomological representations} \label{subsec:cohomreps}

The purpose of this subsection, which is primarily expository, is to establish the cohomological dual and the associated invariants $r$ and $r_0$, defined in Definition~\ref{def:r}, for semisimple groups.

The case of algebraic groups over non-archimedean local fields is well understood due to the following result of Casselman (see also the historical survey by the end of this section).

\begin{theorem}[{\cite[Theorem~2]{Casselman}, \cite[XI, Theorem 3.9]{Borel-Wallach}}] \label{thm:Casselman}
    Let $F$ be a non-archimedean local field, let $\bfG$ a connected almost simple $F$-algebraic group
    and denote $G=\bfG(F)$. Assume $G$ is non-compact.
    Then we have $\widehat{G}_{\cohom}=\{{\bf 1}, \St\}$, where $\St$ denotes the Steinberg representation of $G$.
    Further we have
    \[ H^\ast_c(G,{\bf 1})=H^0_c(G,{\bf 1})=\bbC, \quad H^\ast_c(G,\St)=H^{\rank_F(\bfG)}_c(G,\St)=\bbC. \]
    In particular, we have $r_0(G)=r(G)=\rank_F(\bfG)$.
\end{theorem}

The case of connected semisimple Lie groups with finite center was treated by Vogan and Zuckerman \cite{Vogan-Zuckerman}, following Kumaresan \cite{Kumaresan}, see also \cite[\S4]{Vogan-survey}.
We now survey their theory.

Let $G$ be a non-compact, connected semisimple Lie group with finite center.
Denote by $\mathfrak{g}_0$ the Lie algebra of $G$ and by $\mathfrak{g}$ its complexification.
Let $K$ be a choice of a maximal compact subgroup in $G$
and let $\theta:\mathfrak{g}\to \mathfrak{g}$ be the associated Cartan involution.
Denote by $\mathfrak{g}=\mathfrak{k}\oplus \mathfrak{p}$ the associated eigenspace decomposition, thus $\mathfrak{k}$ is the complexification of the Lie algebra of $K$ and $\mathfrak{p}$ is its orthogonal complement with respect to the Killing form.

A complex parabolic subalgebra $\mathfrak{q}$ of $\mathfrak{g}$ is said to be \emph{$\theta$-stable} if $\theta(\mathfrak{q})=\mathfrak{q}$ and $\mathfrak{l}=\mathfrak{q} \cap \bar{\mathfrak{q}}$ is a complex Levi subalgebra of $\mathfrak{q}$, see \cite[Definition 4.1]{Vogan-survey}.
For each such $\theta$-stable subalgebra $\mathfrak{q}$ an irreducible $(\mathfrak{g},K)$-module $A_{\mathfrak{q}}$ is constructed in \cite[Theorem 2.5]{Vogan-Zuckerman}.
It is proven in \cite[Theorem 1.3]{Vogan} that the modules $A_{\mathfrak{q}}$ are unitarizable, that is for every $\theta$-stable subalgebras $\mathfrak{q}$ there exists an irreducible unitary $G$-representation $U_{\mathfrak{q}}$ which associated $(\mathfrak{g},K)$-module is isomorphic to $A_{\mathfrak{q}}$.
This unitarizability result was unknown at the time \cite{Vogan-Zuckerman} was written,
so our exposition differs from \cite{Vogan-Zuckerman} by taking this into account.
By \cite[Proposition 4.4]{Vogan-survey} there are only finitely many $K$-conjugacy classes of $\theta$-stable parabolic subalgebras in $\mathfrak{g}$ and $U_{\mathfrak{q}}$ depends only on the $K$-conjugacy class of $\mathfrak{g}$ by \cite[Theorem 4.6]{Vogan-survey}.
In \cite[Theorem 4.1]{Vogan-Zuckerman} it is shown that if $U$ is a cohomological irreducible unitary representation of $G$ then $U\cong U_{\mathfrak{q}}$ for some $\theta$-stable subalgebras $\mathfrak{q}$.
The cohomology of $U_{\mathfrak{q}}$ is given in \cite[Theorem 3.3]{Vogan-Zuckerman}, in view of \cite[Proposition 3.2(a)]{Vogan-Zuckerman} and the discussion above, by the formula
\begin{equation} \label{eq:U_q}
    H^\ast_c(G,U_{\mathfrak{q}}) \cong \Hom_{\mathfrak{l}\cap\mathfrak{k}}(\wedge^{\ast-r}(\mathfrak{l} \cap \mathfrak{p}),\bbC),
\end{equation}
where $r$ is the complex dimension of the intersection of $\mathfrak{p}$ with the nilpotent radical of $\mathfrak{q}$. The algebra $\mathfrak{l}$ and the number $r$ in \eqref{eq:U_q} depend implicitly on the choice of the $\theta$-stable subalgebra $\mathfrak{q}$.
In case $\mathfrak{q}=\mathfrak{g}$, we have $U_{\mathfrak{g}}\cong \bbC$ is the trivial representation and $r=0$, otherwise $r>0$.

It follows that $r_0(G)$ is the minimal value of $r$, running over all proper $\theta$-stable subalgebras of $\mathfrak{g}$. 
Note that this value depends on $G$ only up to a local isomorphism.
For almost simple Lie groups it was computed in \cite{Enright}*{Theorem 7.2} and \cite{Kumaresan}*{Theorem 2} (for complex Lie groups) and \cite{Vogan-Zuckerman}*{Theorem 8.1} (for real Lie groups).
For the convenience of the reader, the full table is given in Appendix~\ref{sec:rG}.
Note in particular that we always have $r_0(G)=r(G)\geq \rank(G)$.
The following is a corollary of the above discussion.

\begin{theorem}[Vogan-Zuckerman] \label{thm:Vogan-Zuckerman}
    Let $G$ be a non-compact, connected almost simple Lie group with finite center.
    Then $\widehat{G}_{\cohom}$ consists exactly of the representations $U_{\mathfrak{q}}$ associated with $\theta$-stable subalgebras of $\mathfrak{g}$ as described above. In particular, $\widehat{G}_{\cohom}$ is finite.
    For every representation $U_{\mathfrak{q}}$ the cohomology is given by the formula \eqref{eq:U_q}. In particular, it is finite dimensional. 
    The value $r_0(G)=r(G)$ is given explicitly by the table in Appendix~\ref{sec:rG}.
    In particular, $r_0(G)=r(G)\geq \rank(G)$.
\end{theorem}

We let $G_1,\ldots, G_m$ be the almost simple factors of $G$
and 
decompose accordingly $\mathfrak{g}=\mathfrak{g}_1\oplus \cdots \oplus \mathfrak{g}_m$.
We note that the $\theta$-stable parabolic subalgebras of $\mathfrak{g}$ are decomposed accordingly into direct sums of $\theta$-stable parabolic subalgebras of the $\mathfrak{g}_i$'s,
thus, with an obvious notation, \eqref{eq:U_q} gives
\begin{equation} \label{eq:U_qss}
    H^*_c(G,U_{\mathfrak{q}}) \cong \bigotimes_{i=1}^m \Hom_{\mathfrak{l}_i\cap\mathfrak{k}_i}(\wedge^{*-r_i}(\mathfrak{l}_i \cap \mathfrak{p}_i),\bbC) \cong \bigotimes_{i=1}^m H^*_c(G_i,U_{\mathfrak{q}_i}).
\end{equation}

We now turn to consider a general semisimple group, as we set up in \S\ref{sec: setup}.
The following theorem describes its cohomological dual by means of its simple factors.

\begin{theorem}[Product Formula] \label{thm:productformula}
Let $G \cong \prod_{i=1}^l G_i$ where for each $i$, $F_i$ is a local field, $\bfG_i$ is a connected almost simple $F_i$-group and $G_i=\bfG_i(F_i)$.
Under the identification $\widehat{G} \cong \prod_{i=1}^l \widehat{G_i}$ given by the exterior Hilbertian product, we have $\widehat{G}_{\cohom} \cong \prod_{i=1}^l \widehat{(G_i)}_{\cohom}$.
In particular, $\widehat{G}_{\cohom}$ is finite.
For every representation $V\cong \bar{\otimes} V_i\in \widehat{G}_{\cohom}$ the cohomology is given by 
\begin{equation} \label{eq:prodfor}
    H^*_c(G,V) \cong \bigotimes_{i=1}^l H^*_c(G_i,V_i).
\end{equation}
In particular, it is finite dimensional. 
We have $r_0(G)=\min_i r_0(G_i)$ and $r(G)=\sum_i r(G_i)$. In particular, $r_0(G)\geq \min_i\rank_{F_i}(\bfG_i)$ and $r(G)\geq \rank(G)$.
\end{theorem}

\begin{proof}
We first prove equation~\eqref{eq:prodfor}.
If all $F_i$ are archimedean, that is, $G$ is a Lie group, this is equation\eqref{eq:U_qss}. 
By~\cite[XII, Corollary 3.2]{Borel-Wallach} we have
\[ H^*_c(G' \times G'',V'\bar{\otimes} V'' ) \cong H^*_c(G',V') \otimes H^*_c(G'',V''),\]
where $G'$ is a Lie group and $G''$ a totally disconnected group.
Therefore we may assume that all $F_i$ are non-archimedean, and the equation follows by induction using \cite[X, Corollary 6.2]{Borel-Wallach}.
This completes the proof of equation~\eqref{eq:prodfor}.

  The bijection $\widehat{G}_{\cohom} \cong \prod_{i=1}^l \widehat{(G_i)}_{\cohom}$ and the formulas for $r_0(G)$ and $r(G)$ are formal consequences of equation~\eqref{eq:prodfor} 
  and the finiteness of $\widehat{G}_{\cohom}$ and finite dimensionality of the cohomologies follow from Lemma~\ref{lem:compact}, Theorem~\ref{thm:Vogan-Zuckerman} and Theorem~\ref{thm:Casselman}.
\end{proof}

We finish this subsection with an adelic version of the product formula.

\begin{theorem} \label{thm:adeliccohom}
    Let $k$ be a number field and let $\bfG$ be connected almost simple $k$-algebraic group.
    Then the cohomological dual of the corresponding adelic group  $\widehat{\bfG(\bbA(k))}_{\cohom}$  is countable and discrete.
Each representation in it is the pullback of a cohomological representation of $\prod_S \bfG(k_s)$ for some finite set $S$ consisting of places of $k$.
In particular,
each representation in it has a finite dimensional cohomology. 
More precisely, for every $V\in \widehat{\bfG(\bbA(k))}_{\cohom}$ there exist a finite set $S$ of places of $k$,
and for each $s\in S$ a cohomological representations $V_s$ of $\bfG(k_s)$ such that $V\cong \bar{\otimes}_S V_s$ and 
\[ H^*(\bfG(\bbA(k)),V) \cong \otimes_S H^*_c(\bfG(k_s),V_s). \]
\end{theorem}

\begin{proof}
Set $G=\bfG(\bbA(k))$. 
By \cite[Theorem 3.4(1)]{Conrad}, $\bfG$ is defined over $\calO_{S_0}$, where $S_0$ is a finite set of places of $k$ containing all archimedean ones.
We recall the description of $\widehat{G}$ as given in \cite[\S6]{Moore} or \cite{Tadic} (see also the useful introduction of \cite{Bekka-Cowling}).
Every irreducible representation $V$ of $G$ is of the form 
$\bar{\otimes} V_s$, where for every place $s$ of $k$, $V_s$ is an irreducible representation of $G_s=\bfG(k_s)$ which is spherical for almost every $s$.
Here spherical means that $V^{\bfG(\calO_s)}\neq 0$, which is a well defined expression for $s\notin S_0$.

For every finite set of places $S$ we write $G_S=\prod_S G_s$ and let $G'_S$ be the restricted product of the places $s\notin S$. Similarly, we write $V_S=\bar{\otimes}_{s\in S} V_s$ and $V'_S=\bar{\otimes}_{s\notin S} V_s$.
We claim that 
\[ H^\ast_c(G,V) \cong H^\ast_c(G_S,V_S)\otimes H^\ast_c(G'_S,V'_S). \]
For $S_\infty$ being the archimedean places, this is \cite[XII, Corollary 3.2]{Borel-Wallach}
and for every finite set $S$ of non-archimedean places, by induction using \cite[X, Corollary 6.2]{Borel-Wallach}, we have 
\[ H^\ast_c(G'_{S_\infty},V'_{S_\infty}) \cong H^\ast_c(G_S,V_S) \otimes H^\ast_c(G'_{S_\infty\cup S},V'_{S_\infty\cup S}). \]
The claim now follows by Theorem~\ref{thm:productformula}.

It follows that for every $V\in \widehat{G}_{\cohom}$, for every $s$, $V_s\in \widehat{(G_s)}_{\cohom}$.
In particular, for  every non-archimedean $s$, $V_s$ is either trivial or the Steinberg representation of $G_s$.
Since the Steinberg representation, by construction, is not spherical, we get that  $V_s$ is the trivial representation of $G_s$ for almost every place~$s$,.
The only part of the theorem which does not follow immediately is that $\widehat{G}_{\cohom}$ is discrete.

To this end, we fix a cohomological representation $V$ and show that it is closed in $\widehat{G}_{\cohom}$,
that is, it is not weakly contained in $U$, where $U$ be the direct sum of all other cohomological representations.
Assume the contrary.
Let $S$ be a finite set of places that includes $S_0$ and also all places with $V_s$ non-trivial.
Decompose $U=U_S\oplus U'_S$, where $U_S$ is the sum of all cohomological representations (other than $V$) which are $G_s$-trivial for every $s\notin S$ and $U'_S$ is its complement.

We claim that $V$ is weakly contained in $U_S$.
Let $K_S<G'_S$ be the product of the compact subgroups $\bfG(\mathcal{O}_s)<G_s$, $s\notin S$.
Assume $V$ is weakly contained in $U'_S$.
Then it also weakly contained there as a $K_S$-representation.
But $V_S$ is $K_S$-trivial and $K_S$ is compact, hence has property T.
It follows that $U'_S$ contains a non-trivial $K_S$-invariant vector.
But, by the construction of $U'_S$, each of its $G$-irreducible subrepresentations has a place $s\notin S$ in which the local representation is Steinberg, thus has no $K_s$-invariant vector.
This is a contradiction, and we deduce that $V$ is not weakly contained in $U'_S$.
Thus, $V$ is weakly contained in $U_S$, proving the claim.

We deduce that there is a $G$-irreducible representation $V_0$ in $U_S$ such that $V$ is weakly contained in $V_0$, as $U_S$ is a direct sum of finitely many $G$-irreducible representation.
By \cite[Lemma 4]{Bekka-Cowling}, we get that $V$ is equivalent to $V_0$, contradicting the definition of $U$. This contradicts the assumption that $V$ is weakly contained in $U$
and the proof is now complete.
\end{proof}

The following is an immediate corollary of Theorems \ref{thm:adeliccohom}, \ref{thm:Vogan-Zuckerman} and \ref{thm:Casselman} which pertains to the trivial representation $V=\bbC$.

\begin{cor} \label{cor:adeletriv}
        Let $k$ be a number field and let $\bfG$ be connected almost simple $k$-algebraic group.
        Then the map 
        \[ H_c^*(\mathbf{G}(k\otimes \bbR),\bbC)\to H_c^*(\mathbf{G}(\mathbb{A}(k)),\bbC), \]
        associated with the projection map $\mathbb{A}(k)\to k\otimes \bbR$ stemming from the identification of $k\otimes \bbR$ as the algebra of infinite adeles,
        is a continuous isomorphism.
\end{cor}

\subsection{Low degree cohomology and cocompact lattices} \label{sec:lowdeg}

Recall the definitions of the invariant $r_0$ and $r$ given in Definition~\ref{def:r}.
These invariants are computed explicitly; see \S\ref{subsec:cohomreps} and the  Appendix~\ref{sec:rG}.
In Theorem~\ref{thm:r0} we related the invariant $r_0$ to higher property T.
If $G$ is simple, this is the best we can do. However, in general $r(G)\geq r_0(G)$, and we can do better.
This subsection addresses the case of semisimple groups that are not necessarily simple.
Most theorems in this section come in two flavors: one for a semisimple group as in the setup given in \S\ref{sec: setup} and one for connected semisimple Lie groups.  

\begin{theorem} \label{thm:vanishss-general}
Let $G \cong \prod_{i=1}^l \bfG_i(F_i)$ be a non-compact semisimple group where, for each $i$, $F_i$ is a local field, $\bfG_i$ is a connected almost simple $F_i$-group. Let $G_i=\bfG_i(F_i)$.  
Let $V$ be a unitary $G$-representation such that $V^{G_i}=0$ for each non-compact factor $G_i$.  
    Then $\bar{H}^j_c(G,V)=0$ for every $j<r(G)$. 
    If, for each non-compact factor $G_i$, $V$ has no almost invariant vectors as a $G_i$-representation, then $H^j_c(G,V)=0$ for every $j<r(G)$.
\end{theorem}

\begin{proof}
    Let $G^{is}$ be the product of the non-compact almost simple factors of $G$. Let $G^{an}$ be the product of the compact almost simple factors. By Lemma~\ref{lem: product with compact group} we have \[H^\ast_c(G,V)\cong H_c^\ast\bigl(G^{is}, V^{G^{an}}\bigr),\]
    and $V^{an}$ has no invariant vectors or no almost invariant vectors, respectively, for any almost simple factor of $G^{is}$. 
    Therefore we may and will assume that $G$ has no compact factors. 
    
    By~\cite[Theorem F.5.3]{Tbook}, $V$ decomposes as a direct integral $\int_X V_x$ of irreducible unitary representations over a standard measure space~$X$. For every factor, $V_x^{G_i}=0$ almost everywhere.
       Let $j<r(G)$. By definition of $r(G)$ given in Definition~\ref{def:r}, we have  ${H}^j_c(G,V_x)=0$ almost everywhere. 
       Hence $\bar{H}^j_c(G,V)=0$ by the compatibility of reduced cohomology with direct integral decompositions~\cite[Theorem 7.2]{Blanc}.
This completes the proof of the first part of the theorem.

  Each $G_i$ possesses a cocompact lattice $\Lambda_i$~\cite{Borel-Harder}. Let $\Lambda=\Lambda_1\times \cdots \times \Lambda_l$, which is 
    a cocompact lattice in $G$. 
    We claim that for every unitary $\Lambda$-representation~$U$ with $U^{\Lambda_i}=0$ for each $\Lambda_i$ we have $\bar{H}^j_c(\Lambda,U)=0$ for every $j<r(G)$:  
    Since $(L^2(G_i/\Lambda_i)\bar{\otimes} U)^{G_i}\cong U^{\Lambda_i}=0$, we obtain that 
    \[ I^2(U)^{G_i} \cong \Bigl(L^2(G/\Lambda)\bar{\otimes} U\Bigr)^{G_i} \cong \Bigl(L^2(G_1/\Lambda_1)\bar{\otimes}\cdots \bar{\otimes}L^2(G_l/\Lambda_l)\bar{\otimes} U\Bigr)^{G_i}=0.\]
By the Shapiro Lemma~\ref{lem:Shapiro} and by the first part of the theorem we obtain that 
\[\bar{H}^j(\Lambda,U)=\bar{H}^j_c(G,I^2(U))=0. 
\]
Next we assume that, for any factor $G_i$, $V$ has no almost invariant vectors as a $G_i$-representation.

Let $j<r(G)$. 
    We first show that ${H}^j(\Lambda,V)=0$.
    Otherwise, as in the proof of Theorem~\ref{thm:gensh} we would obtain that  $\bar{H}^j_c(\Lambda,V_\omega)\neq 0$ for the ultrapower representation $V_\omega$
    associated with a fixed ultrafilter $\omega$ on $\bbN$.     
    By \cite[III, Corollary (1.10)]{Margulis}, $V$ has no almost invariant vectors as a $\Lambda_i$-representation. Hence $(V_\omega)^{\Lambda_i}=0$ for $i\in\{1,\dots,l\}$ and, by the above argument,  $\bar H^j(\Lambda, V_\omega)=0$, a contradiction -- if one neglects that the ultrapower $V_\omega$ is not separable. But this can be resolved with Lemma~\ref{lem: non-separable reps}. So we get that ${H}^j(\Lambda,V)=0$ for $j<r(G)$. 

    Finally, by the Shapiro Lemma~\ref{lem:Shapiro}, ${H}^j_c(G,V)$ injects into ${H}^j_c(\Lambda,V)$ and it follows that ${H}^j_c(G,V)=0$.
    This finishes the proof.
\end{proof}

\begin{remark}
The proof of the previous theorem shows that it suffices to assume that $V^{G^{an}}$ has no (almost) invariant factors with respect to any non-compact factor. 
\end{remark}

\begin{theorem} \label{thm:Lie-vanishss-general}
Let $G$ be a non-compact connected semisimple Lie group with finite center. 

Let $V$ be a unitary $G$-representation such that for each almost simple factor $G_i$ of $G$, $V^{G_i}=0$.
    Then $\bar{H}^j_c(G,V)=0$ for every $j<r(G)$.
    
    If for each almost simple factor $G_i$ of $G$, $V$ has no almost invariant vectors as a $G_i$-representation then ${H}^j_c(G,V)=0$ for every $j<r(G)$.
\end{theorem}

\begin{proof}
Lemma~\ref{lem: product with compact group} allows us to take the quotient by the center and reduce the proof to the adjoint case. Then $G$ decomposes as a product $G=G_1\times \dots\times G_n$ still satisfying the assumption. Now we follow verbatim the proof of Theorem~\ref{thm:vanishss-general}. 
\end{proof}

\begin{theorem} \label{thm:isoss}
Let $G \cong \prod_{i=1}^l \bfG_i(F_i)$ be a semisimple group where, for each $i$, $F_i$ is a local field, $\bfG_i$ is a connected, simply connected, almost simple $F_i$-group and $G_i=\bfG_i(F_i)$ is non-compact. 
Let $\Gamma<G$ be an irreducible cocompact lattice.
    Let $V$ be a unitary $G$-representation.
    Then for every $j<r(G)$ the restriction map $H^j_c(G,V)\to H^j(\Gamma,V)$ is a isomorphism.
\end{theorem}

\begin{proof}
By the Shapiro Lemma~\ref{lem:Shapiro}, for every $j$, the restriction map $H^j_c(G,V) \to H^j(\Gamma,V)$ is an isomorphism if and only if  ${H}^j_c(G,V\bar{\otimes} L^2_0(G/\Gamma))=0$.
This follows from Theorem~\ref{thm:vanishss-general}, using Theorem~\ref{thm:specgap}.
\end{proof}

\begin{theorem} \label{thm:Lie-isoss}
Let $G$ be a connected semisimple Lie group with finite center and no compact factors. Let $\Gamma<G$ be an irreducible cocompact lattice.
    Let $V$ be a unitary $G$-representation.
    Then for every $j<r(G)$ the restriction map $H^j_c(G,V)\to H^j(\Gamma,V)$ is a isomorphism.
\end{theorem}

\begin{proof}
    The proof is the same as the proof of Theorem~\ref{thm:isoss}, using Theorem~\ref{thm:Lie-vanishss-general} instead of Theorem~\ref{thm:vanishss-general} and Corollary~\ref{cor:Liespecgap} instead of Theorem~\ref{thm:specgap}.
\end{proof}

\begin{definition}\label{def: G-continuous vectors}
Let $\Gamma$ be a lattice in a semisimple group or semisimple Lie group $G$ with the almost direct product decomposition $G=G_1\cdots G_l$ into almost simple factors. Assume that the projection $\pr_i$ of $\Gamma$ on $G/G_i$ is dense for each $i\in\{1,\dots,l\}$. Let $V$ be a unitary $\Gamma$-representation. 
A vector $v\in V$ is \emph{$G/G_i$-continuous} if 
$\norm{\gamma_pv-v}\to 0$ for every sequence $(\gamma_p)$ in $\Gamma$ with $\pr_i(\gamma_p)\to e$. A vector is \emph{$G$-continuous} if it is a finite sum of $G/G_i$-continuous vectors where $i$ runs over $1,\dots, l$.  
\end{definition}

Some remarks on $G$-continuous vectors are in order. Let $V_i$ be the subspace of $G/G_i$-continuous vectors. It is easy to see that $V_i$ is closed. Since $\Gamma$ projects densely to $G/G_i$, the $\Gamma$-action on $V_i$ extends uniquely to a $G$-action via 
$g\cdot v=\lim_{j\to\infty} \gamma_j v$ for any sequence $(\gamma_j)$ with $\pr_i(\gamma_j)=\pr_i(g)$. The $G$-action on $V_i$ is strongly continuous and unitary.
The closed embedding of Hilbert spaces\[V\to I^2(V),~v\to (g\mapsto g^{-1}v),\]  restricts to an isomorphism $V_i\cong I^2(V)^{G_i}$ of unitary $G$-representations~\cite{BBHP}*{Section~5.1}. One sees inductively that the sum $V_0=V_1+\dots +V_n$ is closed in~$V$ by using the fact that $V_1+\dots+V_{n-1}$ is $G_n$-invariant and the orthogonal projection of $V_n$ onto $V_1+\dots+V_{n-1}$ lies in $V_n$. 
In the extreme case of $l=1$, we have $V_0=V^G$.

\begin{theorem} \label{thm:ssGamma}
    Let $G \cong \prod_{i=1}^l G_i$ where for each $i$, $F_i$ is a local field, $\bfG_i$ is a connected simply connected almost simple $F_i$-algebraic group and $G_i=\bfG_i(F_i)$ is non-compact. 
Let $\Gamma<G$ be an irreducible cocompact lattice.
    Assume $G$ has property T.
    Let $V$ be a unitary $\Gamma$-representation, and let $V_0\subset V$ be the subspace of $G$-continuous vectors. 
    Then for every $j<r(G)$ we have $H^j(\Gamma,V_0^\perp)=0$, and 
    the map 
    \[ H^j_c(G,V_0) \to H^j(\Gamma,V_0) \to H^j(\Gamma,V) \]
    is an isomorphism. Here the first map comes from the restriction from $G$ to $\Gamma$ and the second map comes from the inclusion of $V_0$ in $V$.
\end{theorem}

\begin{proof}
    Since each $\bf G_i$ is simply connected, the lattice $\Gamma$ satisfies the density assumption in Definition~\ref{def: G-continuous vectors} according to~\cite{Margulis}*{Theorem~(6.7) on p.~102}.  
    The theorem follows from Theorem~\ref{thm:isoss}, once we show $H^j(\Gamma,V_0^\perp)=0$.
    We thus assume that $V=V_0^\perp$, that is $V_0=0$, and will show that $H^j(\Gamma,V)=0$.
    Every almost simple factor $G_i$ of $G$ has property T since $G$ does.
    By~\cite[Proposition 5.1]{BBHP} or the discussion above , $I^2(V)^{G_i}\cong V_i =0$. Hence $I^2(V)$ has no almost invariant vectors as a $G_i$-representation. 
    Theorem~\ref{thm:vanishss-general} implies that $H^j_c(G,I^2(V))=0$.
    We conclude by the Shapiro Lemma that indeed, $H^j(\Gamma,V)=0$.
\end{proof}

\begin{theorem} \label{thm:Lie-ssGamma}
Let $G$ be a connected semisimple Lie group with finite center and without compact factors. Let $\Gamma<G$ be an irreducible cocompact lattice.
    Assume $G$ has property T.
    Let $V$ be a unitary $\Gamma$-representation, 
    and let $V_0\subset V$ be the subspace of $G$-continuous vectors.
    Then for every $j<r(G)$ we have $H^j(\Gamma,V_0^\perp)=0$.
    In particular, the maps 
    \[ H^j_c(G,V_0) \to H^j(\Gamma,V_0) \to H^j(\Gamma,V) \]
    are isomorphisms,
    where the first map comes from the restriction from $G$ to $\Gamma$ and the second map comes from the inclusion of $V_0$ in $V$.
\end{theorem}

\begin{proof}
    The proof is the same as the proof of Theorem~\ref{thm:ssGamma}, using Theorem~\ref{thm:Lie-isoss} instead of Theorem~\ref{thm:isoss} and Theorem~\ref{thm:Lie-vanishss-general} instead of Theorem~\ref{thm:vanishss-general}.
\end{proof}

\begin{conjecture}
Theorems~\ref{thm:ssGamma} and \ref{thm:Lie-ssGamma} hold true for every semisimple group of higher rank, with or without property T.
\end{conjecture}

The following generalization of a theorem of~\cite[XII, Proposition 3.5]{Borel-Wallach} is an adelic version of Theorem~\ref{thm:vanishss-general}.

\begin{theorem} \label{thm:adelvanish}
    Let $k$ be a number field and let $\bfG$ be a connected almost simple $k$-group.
    Let $V$ be a unitary representation of $\bfG(\bbA(k))$.
    Assume that for every place $s$ of $k$ such that $\bfG(k_s)$ is non-compact, $V$ does not have almost invariant vectors as a $\bfG(k_s)$-representation.
    Then $H^\ast_c(\bfG(\bbA(k)),V)=0$.
\end{theorem} 

\begin{proof}
The group $\bfG$ is quasi-split, hence isotropic, over $k_s$ for almost every place $s$ of $k$~\cite{conrad-reductive}*{Ex.~5.5.3}.
Denote by $A$ the finitely many places $s$ such that $\bfG$ is anisotropic over $k_s$
and set $G_A=\prod_{s\in A} \bfG(k_s)$.
By \cite[Theorem 3.4(1)]{Conrad}, $\bfG$ is defined over $\calO_{S_0}$, where $S_0$ is a finite set of $k$ containing all archimedean ones.
Let $S$ be a finite set of places of $k$ which includes $S_0\setminus A$ and is disjoint from $A$.
Set $G_S=\prod_{s\in S} \bfG(k_s)$, $K_S= \prod_{s\notin S\cup A} \bfG(\mathcal{O}_s)$, consider $G_A\times G_S \times K_S$ as an open subgroup of the adele group $G=\mathbf{G}(\bbA(k))$
and note that 
\[ G\cong \colim_S G_A\times G_S \times K_S. \]
Therefore the $n$-cochain group $C(G^n,V)\cong C(G^{n+1},V)^G$ is the inverse limit 
\[C(G^n,V) \cong \lim_S C((G_A\times G_S \times K_S)^n,V). \]
The structure maps of the limit are surjective since we extend any continuous map on $(G_A\times G_S \times K_S)^n$ to $(G_A\times G_{S'}\times K_{S'})^n$, where $S\subset S'$, by Dugundji's generalization~\cite{dugundji}*{Theorem~4.1} of Tietze's extension theorem. So the inverse system on the right satisfies the Mittag-Leffler condition.
Milnor's exact sequence for inverse systems~\cite{weibel}*{Theorem~3.5.8 on p.~83} now gives
\[ 0\to {\lim}^1_S H^{n-1}_c\bigl(G_A\times G_S \times K_S,V\bigr) \to H^n_c\bigl(G,V\bigr) \to \lim_S H^n_c\bigl(G_A\times G_S \times K_S,V\bigr) \to 0.
\]
By Theorem~\ref{thm:vanishss-general}, 
$H^n_c\bigl(G_A\times G_S \times K_S,V\bigr)=0 $ 
for every $n<|S|\leq r(G_S)$ from which 
we conclude that $H^\ast_c\bigl(G,V\bigr)=0$.
\end{proof}

\section{Polynomial cohomology of semisimple groups} \label{sec:PCH}

Cohomology involving growth restrictions on cochains can be found in a variety of contexts. Instances include bounded cohomology, $\ell^2$-cohomology, and the de Rham cohomology of moderately growing differential forms. In this section, we study polynomial cohomology (see~\S\ref{subsec: definition polynomial cohomology} and\S\ref{subsec: homological algebra for polynomial}), i.e.~the group cohomology of polynomially bounded cochains, as it applies to arithmetic lattices. Connes and Moscovici first introduced polynomial cohomology in their work on the Novikov conjecture~\cite{connes+moscovici}. 
In~\S\ref{subsec: filling functions} we discuss filling functions and isoperimetric inequalities of lattices in semisimple groups. A crucial ingredient for this paper are two theorems from geometric group theory: Theorem~\ref{thm: leuzinger-young} by Leuzinger-Young and Theorem~\ref{thm: bestvina-eskin-wortman} by Bestvina-Eskin-Wortman presented in~\S\ref{subsec: filling functions}. 

Isoperimetric inequalities in groups and polynomial cohomology will be related first in Proposition~\ref{prop: continuous chain contraction from filling} and then in Proposition~\ref{prop: comparison map and polynomial filling rank} in~\S\ref{subsec: comparison map}. The relation between polynomial filling functions and the fact that the comparison map from polynomial to ordinary cohomology is an isomorphism was first observed by Ogle~\cite{ogle}. Ji-Ramsey~\cite{ji+ramsey} clarified that Ogle's definition coincides (up to polynomial equivalence) with the usual definition (see Lemma~\ref{lem: polynomial equivalence}). 
In the proof of~\cite{ji+ramsey}*{Theorem~2.6} they implicitly use however another definition of polynomial filling function which corresponds to the contractibility of the chain complex~\eqref{eq: bornological iso between completions}. We believe that this causes a problem for the $(1)\Rightarrow (3)$-direction of~\cite{ji+ramsey}*{Theorem~2.6}. The aforementioned Proposition~\ref{prop: comparison map and polynomial filling rank}, which we prove from scratch, generalizes the $(3)\Rightarrow (1)$-direction of~\cite{ji+ramsey}*{Theorem~2.6}.  

A major insight of the present paper is to use tools from geometric group theory and the polynomial cohomology to prove a version of the Shapiro lemma in low degrees for non-uniform lattices in~\S\ref{subsec:Shapiro++}. One may see this as a far-reaching generalization of ideas of Shalom about $1$-cohomology and unitary induction of square-integrable lattices~\cite{Shalom}.

\subsection{Definition of polynomial cohomology}
\label{subsec: definition polynomial cohomology}
Let $G$ be a compactly generated locally compact second countable group. Let $l\colon G\to [0,\infty]$ be the word length function associated with a compact generating set in $G$. Let $V$ be a unitary $G$-representation. 

We consider the (continuous) homogeneous bar complex $C(G^{\ast+1}, V)$ for $G$ with coefficients in $V$.  A cochain $f\in C(G^{n+1}, V)$ is \emph{polynomial} if there are constants $C>0$ and $k\in\bbN$ such that 
\[ \norm{f(g_0,\dots,g_n)}\le C\cdot \bigl(1+l(g_0)+\dots+l(g_n)\bigr)^k\]
for all $(g_0,\dots,g_n)\in G^{n+1}$. 
The polynomial cochains form a $G$-invariant subcomplex of the homogeneous bar resolution $C(G^{\ast+1}, V)$ of continuous cochains which we call the \emph{polynomial chain complex} $C_\pol\bigl(G^{n+1}, V\bigr)$. 
The 
diagonal $G$-action on the homogeneous bar resolution restricts to a $G$-action on the polynomial chain complex. 
The \emph{polynomial cohomology} of $G$ with coefficients in a unitary $G$-representation $V$ is defined as 
  \[ H_\pol^n(G;V)=H^n\bigl(C_\pol^\ast(G, V)^G\bigr),\]
  see~\cite{connes+moscovici}*{p.~384}. 
The inclusion of the polynomial chain complex into the homogeneous bar complex induces a natural homomorphism 
\[ H_\pol^\ast(G; V)\to H^\ast_c(G; V),\]
which we call the \emph{comparison map}. 

Since all word metrics of a compactly generated group are quasi-isometric the definition of polynomial cohomology is independent of the choice of~$l$.

In the case of a semisimple group $G$ in the sense of Setup~\S\ref{sec: setup} or in the case of a semisimple Lie group $G$ with finite center we will pass from the polynomial cohomology of a lattice $\Gamma<G$ to the polynomial cohomology of $G$ and back. Thus it is important that the restriction of a word length function of $G$ is a word length function of $\Gamma$. This is provided by the following remarkable theorem of Lubotzky-Mozes-Raghunathan~\cite{lubotzky+mozes+raghunathan} which we will use implicitly throughout. Note that the interesting case is the one of non-uniform lattices. 

\begin{theorem}[Lubotzky-Mozes-Raghunathan]\label{thm: lubotzky-mozes-raghunathan}
Let $G$ be a semisimple group of higher rank or a semisimple Lie group of higher rank with finite center. Let $\Gamma<G$ be an irreducible lattice. Then the restriction of a word length function of $G$ to $\Gamma$ is a word length function of~$\Gamma$. 
\end{theorem}

The above theorem is proved in the context of semisimple groups as in~\S\ref{sec: setup}. But it follows for semisimple Lie groups with finite center from that fact 
that every metric quasi-isometric to a word metric is a word metric and the following remark. 

\begin{remark}\label{rem: QI remark on G}
If $G$ is a connected semisimple Lie group with finite center $Z$, then $G/Z$ is a finite index subgroup of 
a semisimple group in the sense of Setup~\ref{sec: setup} (See~\cite{zimmer}*{Proposition~3.1.6 on p.~35}. In particular, such $G$ is quasi-isometric to a semisimple group in the sense of Setup~\ref{sec: setup}. 
\end{remark}

\subsection{Homological algebra for polynomial cohomology}\label{subsec: homological algebra for polynomial}

For the case of discrete groups we review a characterization of polynomial cohomology in the sense of relative homological algebra over the Fr\'{e}chet algebra $\calS(\Gamma)$ defined below. This approach to polynomial cohomology was first taken  by Ji~\cite{ji}, and was much developed much further by Meyer in the setting of bornological modules~\cites{meyer, meyer-derived}. 

Relative homological algebra over a Fr\'{e}chet algebra $A$ is quite similar to the relative homological for continuous group cohomology as in~\cite{Blanc}. A (continuous) \emph{$A$-module} is a Fr\'{e}chet space $M$ with a continuous bilinear map $A\times M\to M$ satisfying the usual properties. Morphisms between $A$-modules are continuous $A$-linear maps. The category of $A$-modules is not an abelian category but an exact category in the sense of Quillen where the class of extensions are the ones with a continuous $\bbC$-linear section. Therefore the notions of exact sequences, resolutions and projectivity are understood in the sense of relative homological algebra, that is, assuming the existence of continuous $\bbC$-linear sections or chain contractions. We tacitly assume this when speaking about $A$-resolutions and exactness of a sequence of $A$-modules. 
For instance, the exactness of a chain complex of $A$-modules 
\[ 0\leftarrow M_0\xleftarrow{\partial_0} M_1\xleftarrow{\partial_1}\dots\xleftarrow{\partial_{n-1}} M_n\]
means that there are $\bbC$-linear continuous sections of the maps $\im(\partial_{i-1})\leftarrow M_i$. Equivalently, there is a $\bbC$-linear continuous contracting chain homotopy up to degree~$n$. 
We also tacitly assume a $A$-linear map to be continuous. 

Let $\hat\otimes$ denote the projective tensor product. A \emph{free $A$-module} is a module of the form $A\hat\otimes M$ with the obvious $A$-action. Every free module is projective. 
See~\cites{meyer, meyer-derived} for a background reference. 

For a function $f\colon\Gamma\to\bbC$ we set $p_k(f)= \sum_{\gamma\in\Gamma} |f(\gamma)|(1+l(\gamma))^k$. 
Let $\calS(\Gamma)$ be the convolution algebra 
\[ \calS(\Gamma)=\bigl\{ f\colon \Gamma\to \bbC \mid \forall_{k\in\bbN}~p_k(f)<\infty \bigr\}.\]
This is a Fr\'{e}chet algebra with regard to the family of norms $p_k$. The group algebra $\bbC\Gamma$ densely embeds into $\calS(\Gamma)$. 
The trivial $\calS(\Gamma)$-module $\bbC$ has a canonical free $\calS(\Gamma)$-resolution 
\[ 0\leftarrow \bbC\leftarrow\calS(\Gamma)\leftarrow \calS(\Gamma)^{\otimes^2}\leftarrow \calS(\Gamma)^{\otimes^3}\leftarrow\dots\]
which is the completion of the standard ``bar" $\bbC\Gamma$-resolution of~$\bbC$. 

Let $X$ be a free simplicial $\Gamma$-complex whose skeleta are cocompact. We equip $X$ with the path length metric that on each $d$-simplex is the one of the standard Euclidean $d$-simplex. 
We choose a base vertex $v_0$ in $X$. Each (ordered) simplex $s$ is weighted by the distance of its first vertex to $v_0$ which we denote by $w(s)$. We define the chain complex $\calS C_\ast^\cell(X)$ as the completion of the complex cellular chain complex $C_\ast^\cell(X,\bbC)$ with regard to the family of norms $(p_k^X)_{k\ge 1}$, where for a cellular $d$-chain $c=\sum_\sigma a_\sigma\cdot\sigma$ we set 
\begin{equation}\label{eq: space norm family} p^X_k(c) = \sum_\sigma |a_\sigma|(1+w(\sigma))^k.
\end{equation}
The $d$-skeleton of $X$ only has finitely many orbits of cells, and each cellular chain group is a finitely generated free $\bbC\Gamma$-module. If $\sigma_0$ is a $d$-simplex then the function $\gamma\mapsto w(\gamma \sigma_0)$ is equivalent to the length function~$l$. This implies that the choice of a $d$-simplex in each orbit yields a continuous isomorphism $\calS(\Gamma)^{l_d}\cong \calS C_d^\cell(X)$ where $l_d$ is the number of orbits.
Therefore the inclusion $C_\ast^\cell(X,\bbC)\hookrightarrow\calS C_\ast^\cell(X)$ induces an isomorphism 
\begin{equation}\label{eq: bornological iso between completions}
\calS(\Gamma)\otimes_{\bbC\Gamma} C_\ast^\cell(X,\bbC)\cong\calS(\Gamma)^{l_\ast}\cong \calS C_\ast^\cell(X).
\end{equation}
\begin{lemma}\label{lem: bounded maps to Hilbert space}
Let $V$ be a unitary $\Gamma$-representation. 
The inclusion \[ \bbC[\Gamma^{n+1}]=\bbC[\Gamma]\otimes\dots\otimes\bbC[\Gamma]\subset\calS(\Gamma)\,\hat\otimes\dots\hat\otimes\,\calS(\Gamma) \]
induces an isomorphism 
\[ C_\pol(\Gamma^{n+1},V)^\Gamma\cong 
\hom_{\calS(\Gamma)}\bigl(\calS(\Gamma)\,\hat\otimes\dots\hat\otimes\,\calS(\Gamma), V\bigr). \]
\end{lemma}

\begin{proof}
This is an immediate consequence of the fact that the projective tensor product $\calS(\Gamma)\hat\otimes\dots\hat\otimes\calS(\Gamma)$ is the completion of $\bbC[\Gamma^{n+1}]$ with respect to the family~$(q_k)$ 
\[ q_k\bigl(\sum_{\bar\gamma\in\Gamma^{n+1}} a_{\bar\gamma} (\gamma_0,\dots,\gamma_n)\bigr)=\sum_{\bar\gamma\in\Gamma^{n+1}}|a_{\bar\gamma}|\bigl(1+l(\gamma_0)+\dots+l(\gamma_n)\bigr)^k\]
of norms. 
\end{proof}

\begin{prop}\label{prop: Schwartz chain contraction and comparison map}
Let $X$ be a contractible free simplicial $\Gamma$-complex with finite $(d+1)$-skeleton such that $0\leftarrow\bbC\leftarrow \calS C^\cell_0(X)\leftarrow\dots\leftarrow \calS C_{d+1}^\cell(X)$ is exact. Let $V$ be a unitary $\Gamma$-representation. 
Then the comparison map $H^i_\pol(\Gamma, V)\to H^i(\Gamma, V)$ is an isomorphism for $0\le i\le d$ and a monomorphism for $i=d+1$.  
\end{prop}

\begin{proof}
    The truncation of the cellular chain complex 
    \begin{equation}\label{eq: truncated cellular chain complex} 0\leftarrow \bbC\leftarrow C_0^\cell(X,\bbC)\leftarrow\dots\leftarrow C_{d+1}^\cell(X,\bbC)
    \end{equation}
is exact and consists of finitely generated free $\bbC\Gamma$-modules. 
By assumption and~\eqref{eq: bornological iso between completions}, the complex~\eqref{eq: truncated cellular chain complex} stays exact after tensoring with $\calS(\Gamma)$; we extend it to a free $\calS(\Gamma)$-resolution $S_\ast$ of $\bbC$ so that $S_i=\calS(\Gamma)\otimes_{\bbC\Gamma} C_i^\cell(X,\bbC)$ for $i\le d+1$.  
Using projectivity of $\calS(\Gamma)\otimes_{\bbC\Gamma} C_\ast^\cell(X,\bbC)$ and exactness of $S_\ast$, we obtain a $\calS(\Gamma)$-chain map 
\[f_\ast\colon\calS(\Gamma)\otimes_{\bbC\Gamma} C_\ast^\cell(X,\bbC)\to S_\ast\] that is the identity up to degree $d+1$. 
We choose a $\bbC\Gamma$-linear chain map 
\[ C_\ast^\cell(X,\bbC)\to \bbC[\Gamma^{\ast+1}]=\bbC\Gamma^{\otimes^{\ast+1}} \]
between the two projective $\bbC\Gamma$-resolutions. It extends to a $\calS(\Gamma)$-chain map 
\[ \calS(\Gamma)\otimes_{\bbC\Gamma} C_\ast^\cell(X,\bbC)\to \calS(\Gamma)^{\otimes^{\ast+1}}.\]
By the fundamental theorem of homological algebra, we obtain a $\calS(\Gamma)$-chain homotopy equivalence 
\[ S_\ast\to \calS(\Gamma)^{\otimes^{\ast+1}}\]
so that all three chain maps form a homotopy-commutative triangle. Hence we obtain an induced homotopy-commutative diagram: 
\[
\begin{tikzcd}\hom_{\calS(\Gamma)}\bigl(\calS(\Gamma)^{\otimes^{\ast+1}}, V\bigr)\arrow[r, "\simeq"]\arrow[dd]&\hom_{\calS(\Gamma)}\bigl(S_\ast, V\bigr)\arrow[d, "f^\ast"] \\
 & \hom_{\calS(\Gamma)}\bigl(\calS(\Gamma)\otimes_{\bbC\Gamma} C_\ast^\cell(X,\bbC), V\bigr)\arrow[d, "\cong"]\\
\hom_{\bbC\Gamma}\bigl(\bbC\Gamma^{\otimes^{\ast+1}}, V\bigr)\arrow[r, "\simeq"]& \hom_{\bbC\Gamma}\bigl(C_\ast^\cell(X,\bbC), V\bigr)
\end{tikzcd}
\]
The left vertical restriction map induces the comparison map. Every arrow except the left vertical and $f^\ast$ are chain homotopy equivalences or chain isomorphisms. 
It is clear that $f^\ast$ is an isomorphism up to degree~$d+1$ and hence induces a cohomology isomorphism up to degree~$d$ and a cohomology monomorphism in degree~$d+1$. 
Thus so does the left vertical map too.
\end{proof}

\begin{prop}\label{prop: comparison iso from geometric exactness}
Let $X$ be a contractible free simplicial $\Gamma$-complex with cocompact $(d+1)$-skeleton. Let $V$ be a unitary $\Gamma$-representation. 
Let $h_\ast: C_\ast^\cell(X,\bbC)\to C_{\ast+1}^\cell(X,\bbC)$ be a chain contraction such that $h_i$ is continuous for every $0\le i\le d$ where $C_\ast^\cell(X,\bbC)$ carries the subspace topology from $\calS C_\ast^\cell(X)$. 
Then the comparison map $H^{d+1}_\pol(\Gamma, V)\to H^{d+1}(\Gamma, V)$ is surjective. 
\end{prop}
In the context of relative homological algebra, the exactness 
of the sequence of $\calS(\Gamma)$-modules in Proposition~\ref{prop: Schwartz chain contraction and comparison map} assumes the existence of a continuous $\bbC$-linear homotopy. The assumption of the above Proposition says that the continuous linear homotopy that makes $0\leftarrow\bbC\leftarrow \calS C^\cell_0(X)\leftarrow\dots\leftarrow \calS C_{d+1}^\cell(X)$ exact restricts to the dense subspaces $C_\ast^\cell(X,\bbC)$.  

\begin{proof}
We endow $\bbC\Gamma^{\otimes^{\ast+1}}$ with the subspace topology inherited from its dense embedding into $\calS(\Gamma)^{\hat\otimes^{\ast+1}}$. 

Inductively and using $h_\ast$, we construct a $\bbC\Gamma$-chain map $g_\ast\colon \bbC\Gamma^{\otimes^{\ast+1}}\to C_\ast^\cell(X,\bbC)$ that is continuous up to degree~$d+1$. In degree zero it comes from a choice of a base vertex in~$X$. 
Suppose $g_0, \dots, g_k$ are  already constructed. The $\bbC\Gamma$-module $\bbC\Gamma^{\otimes^{k+1}}$ is free. We define $g_{k+1}$ on the $\bbC\Gamma$-basis $\{1\}\times\Gamma^{k+1}\subset\bbC[\Gamma^{k+2}]\subset \bbC\Gamma^{\otimes^{k+2}}$ as 
$h_k\circ g_k\circ \partial$ and extend it by equivariance to a $\bbC\Gamma$-map. 
Since each $h_i$ is continuous for $i\le d$, 
the maps $g_0, \dots, g_{d+1}$ are continuous.  As a map between projective $\bbC\Gamma$-resolutions, $g_\ast$ is automatically a chain homotopy equivalence. 

Let $z\in \hom_{\bbC\Gamma}\bigl(\bbC\Gamma^{\otimes^{d+2}}, V\bigr)$ be a $(d+1)$-cocycle. Up to adding a coboundary to $z$, we may assume that there is cocycle $z'\in \hom_{\bbC\Gamma}\bigl(C_{d+1}^\cell(X,\bbC), V\bigr)$ such that $z'\circ g_{d+1}=z$. The cocycle $z'$ is automatically continuous. By continuity of $g_{d+1}$, the cocycle $z$ is continuous which means that $z$ is polynomial by Lemma~\ref{lem: bounded maps to Hilbert space}.  
\end{proof}

\subsection{Filling functions of lattices in semisimple groups}\label{subsec: filling functions}

We review the filling functions of a group which are higher-dimensional analogs of the Dehn function. 

Let $X$ be a simplicial complex or a Riemannian manifold. In the former case we equip every simplex with the metric of the Euclidean standard simplex and $X$ with the induced path length metric. 

A \emph{Lipschitz singular chain} in $X$ is a singular chain such that each singular simplex is Lipschitz. 
The chain complex of Lipschitz singular chains $C_\ast^\lip(X)$ is a subcomplex of the singular chain complex $C_\ast(Y)$. 
Let $c=\sum_{i=1}^n a_i\sigma_i$ be a Lipschitz singular $(d-1)$-chain. The \emph{$(d-1)$-volume} $\vol^{d-1}(\sigma_i)$ of the singular $(d-1)$-simplex $\sigma_i\colon \Delta^{d-1}\to X$ is defined as 
\[ \vol^{d-1}(\sigma_i)=\int_{\Delta^{d-1}}|J_{\sigma_i}(x)|dx.
\]
The Jacobian is almost everywhere defined by Rademacher's theorem. The \emph{mass} of~$c$ is defined as 
\[ \mass(c)=\sum_{i=1}^n |a_i|\cdot \vol^{d-1}(\sigma_i).\]
The \emph{filling volume} of a (integral) cycle $c\in C_{d-1}^\lip(X,\bbZ)$
is defined as 
\begin{equation*}
\filling^d_X(c)=\inf_{\substack{b\in C^\lip_d(X,\bbZ)\\ \partial b=c}}\mass(b)
\end{equation*}
Let $X$ be $(d-1)$-connected. 
We define the $d$-dimensional \emph{filling function} $\filling^d_X$ of $X$ as 
\begin{equation}\label{eq: geometric filling volume}
\filling^d_X(v)=\sup_{\substack{c\in C_{d-1}^\lip(X,\bbZ)\\\partial c=0,  \mass(c)\le v}} \filling^d(c). 
\end{equation}

We say that the filling function is \emph{polynomial} if there is some $k\in\bbN$ and $C>0$ such that $\filling^d(v)\le C\cdot v^k$ for every $v\ge 1$. 

Next we review two combinatorial analogs of the filling function for a simplicial complex~$X$ or, slightly more general, for a CW-complex $X$ where each cell is isometric to a convex Euclidean polyhedron and the attaching maps are isometries. On each cellular chain group $C_i^\cell(X)$ of $X$ we consider both the $\ell^1$-norm $\norm{\_}$ with respect to the basis of $i$-cells and the norm $p_1^X$ defined in~\eqref{eq: space norm family}. 
These norms take integer values on $C^\cell_d(X,\bbZ)$.

The \emph{combinatorial filling volume} of an (integral) cycle $c\in C_{d-1}^\cell(X,\bbZ)$ is defined as 
\[ \combfilling^d(c)=\min_{\substack{b\in C^\cell_d(X,\bbZ)\\ \partial b=c}}\norm b.\]
The \emph{weighted combinatorial filling volume} of a cycle $c\in C_{d-1}^\cell(X,\bbZ)$ is defined as 
\[ \combweightedfilling^d(c)=\min_{\substack{b\in C^\cell_d(X,\bbZ)\\ \partial b=c}}p_1^X(b).\]

If $X$ is $(d-1)$-connected, we define the $d$-th \emph{combinatorial filling function} $\combfilling^d_X$ by 
\begin{equation}\label{eq: combinatorial filling volume}
\combfilling^d_X(v)=\sup_{\substack{c\in C_{d-1}^\cell(X,\bbZ)\\\partial c=0,  \norm{c}\le v}} \combfilling^d(c). 
\end{equation}
The weighted version is defined as 
\begin{equation}\label{eq: combinatorial weighted filling volume}
\combweightedfilling^d_X(v)=\sup_{\substack{c\in C_{d-1}^\cell(X,\bbZ)\\\partial c=0,  p_1^X(c)\le v}} \combweightedfilling^d(c). 
\end{equation}

The following well-known fact is a consequence of the Federer-Fleming deformation technique. See~\citelist{\cite{abrams_etal}*{Section~2}\cite{epstein}*{10.3}} for statements and further references. 

\begin{theorem}\label{thm: federer-fleming thm}
Let $X$ be a contractible free simplicial $\Gamma$-complex with cocompact $(d+1)$-skeleton or a Riemannian manifold on which $\Gamma$ acts cocompactly and properly discontinuously by isometries. 
Then 
\[ \filling_X^i\approx \combfilling_X^i\]
for $2\le i\le d$. 
\end{theorem}

The following is an easier result than the previous theorem. The combinatorial weighted filling function is asymptotically bounded by the square of the combinatorial filling function. 

\begin{lemma}[\cite{ji+ramsey}*{Corollary~2.5}]\label{lem: polynomial equivalence}
 Let $X$ be a contractible free simplicial $\Gamma$-complex with cocompact $(d+1)$-skeleton. For $i\le d$, the function  $\combfilling_X^i$ is polynomial if and only if $\combweightedfilling_X^i$ is polynomial. 
 \end{lemma}

The $\Rightarrow$-implication of this lemma is needed in the proof of Proposition~\ref{prop: continuous chain contraction from filling}. To rectify a minor oversight in the proof of~\cite{ji+ramsey}*{Corollary~2.5} we present the argument below. 

\begin{proof}[Proof of $\Rightarrow$-implication]
Suppose that $f=\combfilling^i_X$ is polynomial. Let $c$ be an integral $(i-1)$-cycle, and let $b$ be an integral $i$-chain so that 
$\partial b=c$ and $\norm{b}_1=f(\norm{c}_1)$. We can write $b$ as a sum $\sum_{p=1}^s b_p=\sum_{p=1}^s\sum_\sigma b_{\sigma,p}\sigma$ of chains such that the simplices $\sigma$, $b_{\sigma,p}\ne 0$, appearing in each $b_p$ form a connected subcomplex and no simplex of $b_p$ has a common vertex with a simplex of $b_q$ for any $p\ne q$. Pick a simplex $v_p$ in each $\partial b_p$. 
Since there is no cancellation between simplices coming from $b_p$ and $b_q$ for $p\ne q$ we have 
$\sum_{i=1}^s \bigl(w(v_p)-1\bigr)\le p_1^X(c)$. Hence $\sum_{i=1}^s w(v_p)\le p_1^X(c)+s\le p_1^X(c)+\norm{b}_1$. 
Let $N_p$ be the number of simplices in $b_p$. By integrality we have $N_p\le \norm{b_p}_1$. 
Therefore, each vertex of a simplex in $b_p$ has at most distance $\norm{b_p}_1$ from each vertex of $v_p$. 
We conclude that 
\begin{align*}
    p_1^X(b)=\sum_{p=1}^s\sum_\sigma |b_{\sigma,p}|w(\sigma)
    &\le \sum_{p=1}^s\sum_\sigma |b_{\sigma,p}|\bigl(\norm{b_p}_1 +w(v_p)\bigr)\\
    &\le  \norm{b}_1^2+\sum_p^s\sum_\sigma |b_{\sigma,p}|w(v_p)\\
    &\le  \norm{b}_1^2+\norm{b}_1\cdot\sum_{p=1}^s w(v_p)\\
    &\le  \norm{b}_1^2+\norm{b}_1\cdot \bigl(p_1^X(c)+\norm{b}_1\bigr)\\
    &\le f(\norm{c}_1)^2+f(\norm{c}_1)\cdot \bigl(p_1^X(c)+f(\norm{c}_1\bigr)\bigr)\\
    &\le f\bigl(p_1^X(c)\bigr)^2+f\bigl(p_1^X(c)\bigr)\cdot \bigl(p_1^X(c)+f\bigr(p_1^X(c)\bigr)\bigr).
\end{align*}
Hence $\combweightedfilling^i_X(v)\le f(v)^2+f(v)\cdot (v+f(v))$ is also polynomial. 
\end{proof}

\begin{defn} \label{def:polyrank}
Let $\Gamma$ be a finitely generated group. We define 
    the \emph{polynomial filling rank} as 
    \begin{multline*}\frank\Gamma=\sup\bigl\{d\in\{3,4,\dots\}\mid \text{$\Gamma$ is of type $F_{d–1}$ and}\\\text{$\filling_\Gamma^i$ is polynomial for all $2\le i<d$} \bigr\}\in\bbN\cup\{\infty\}
    \end{multline*}
 provided the set over which we take the supremum is non-empty. 
 Otherwise we set $\frank\Gamma=2$.  
\end{defn}

The following statement is well known. 

\begin{prop}\label{prop: polynomial filling rank for uniform lattices}
A uniform lattice in a semisimple group or connected semisimple Lie group with finite center has infinite polynomial filling rank. 
\end{prop}

\begin{proof}
By Remark~\ref{rem: QI remark on G} we may assume that the ambient group $G$ is semisimple in the sense of~\S\ref{sec: setup}.  
Let $\Lambda<G$ be a uniform lattice in~$G\cong\prod_{i=1}^l \bfG_i(F_i)$. Depending on whether $F_i$ is archimedean or not, let $X_i$ be the symmetric space or the Bruhat-Tits building associated with $\bfG_i(F_i)$. The product $X=X_1\times\dots \times X_l$ is a CAT(0)-space on which $\Lambda$ acts properly and cocompactly by isometries. Thus the filling functions satisfy $\filling^n_X(v)\lessapprox v^{\frac{n}{n-1}}$ for $n\ge 2$ like in the corresponding Euclidean situation~\cite{wenger}. In particular, they are polynomial.  
By Theorem~\ref{thm: federer-fleming thm} the polynomial filling rank of $\Lambda$ is infinite.
\end{proof}

The case of non-uniform lattices is the subject of the next two deep theorems. 

\begin{theorem}[Leuzinger-Young~\cite{leuzinger+young}*{Theorem~1.2}]\label{thm: leuzinger-young}
Let $\Gamma$ be an non-uniform irreducible lattice in a connected semisimple Lie group~$G$ with a finite center and without compact factors. 
Then 
\[ \frank\Gamma=\max\{2, \rank G\}.\]
If $\rank G\ge 3$, then 
\[ \filling^i_\Gamma(v)\approx v^\frac{i}{i-1} \]
for all $2\le i< \rank G$ and there is constant $c>0$ such that 
\[ \filling^{\rank G}_\Gamma(v)\gtrapprox \exp\bigl(c\cdot v^\frac{1}{\rank G-1}\bigr). \]
\end{theorem}

\begin{theorem}[Bestvina-Eskin-Wortman~\cite{bestvina+eskin+wortman}*{Corollary~5}]\label{thm: bestvina-eskin-wortman}
Let $k$ be a global field and let $\bfG$ be a connected, absolutely almost simple, isotropic $k$-algebraic group. 
Let $S$ be a finite set of places of $k$ that includes all archimedean ones
and let $\Gamma<\bfG(k)$ be an $S$-arithmetic subgroup.
Then $\frank \Gamma \ge |S|-2$. 
\end{theorem}

\begin{remark}
       In~\cite{leuzinger+young}*{Theorem~1.2} it is assumed that $G$ is centerfree and in \cite{bestvina+eskin+wortman}*{Corollary~5} it is assumed that $\Gamma=\bfG(\calO_S)$.
       Since the filling functions are quasi-isometry invariants these assumptions could be relaxed as above. 
       Note also that what Bestvina-Eskin-Wortman call an \emph{$m$-dimensional isoperimetric inequality} corresponds to an \emph{$(m-1)$-dimensional isoperimetric inequality} in the paper by Leuzinger-Young. 
\end{remark}
 
\begin{prop}\label{prop: continuous chain contraction from filling}
Let $d=\frank\Gamma-2$. 
  Let $X$ be a contractible free simplicial $\Gamma$-complex with cocompact $(d+1)$-skeleton. Then there is a chain contraction $h_\ast: C_\ast^\cell(X,\bbC)\to C_{\ast+1}^\cell(X,\bbC)$ such that $h_i$ is continuous for every $0\le i\le d$ when $C_\ast^\cell(X,\bbC)$ carries the subspace topology from $\calS C_\ast^\cell(X)$. 
\end{prop}

\begin{proof}
Recall that the topology on $\calS C_\ast^\cell(X)$ is given by the family of norms in~\eqref{eq: space norm family}. 
Note that a homomorphism $f\colon C_i^\cell(X,\bbC)\to C_{i+1}^\cell(X,\bbC)$ is continuous 
if for every $r\in\bbN$ there is $s\in\bbN$ and $c>0$ such that $p_r^X(f(x))\le c\cdot p_s^X(x)$ for every $i$-chain~$x$, or equivalently, for every $i$-cell $x$. Now suppose that $f$ is integral, that is, $f$ restricts to a homomorphism $C_i^\cell(X,\bbZ)\to C_{i+1}^\cell(X,\bbZ)$. Then continuity of $f$ already follows if there is $s\in\bbN$, $c>0$ such that $p_1^X(f(e))\le c\cdot p_s^X(e)$ for every $i$-cell~$e$ since for an integral chain $f(e)$ we have $p_r^X(f(e))\le p_1^X(f(e))^r$ and for every cell $e$ we have $p_s^X(e)^r=p_{r+s}^X(e)$. 

We construct $h_\ast$ inductively. We define $h_{-1}\colon\bbC\to C_0^\cell(X,\bbC)$ to be the map that sends $1$ to a base vertex $v_0$. For a vertex $v$,  let $h_0(v)\in C_1^\cell(X,\bbC)$ be a simplicial geodesic  from $v$ to $v_0$. Obviously, $h_0$ is continuous, integral and $\partial h_0+h_{-1}\partial=\id$.  

Suppose there are continuous and integral homomorphisms $h_i\colon C_i^\cell(X,\bbC)\to C_{i+1}^\cell(X,\bbC)$ for $0\le i\le k-1$ such that 
$\partial h_i+h_{i-1}\partial=\id$  
for every $0\le i\le k-1$. If $k-1<d$, for every $k$-cell $e\in C_k^\cell(X,\bbZ)\subset C_k^\cell(X,\bbC)$ we choose an \emph{optimal}  
boundary $b\in C^\cell_{k+1}(X,\bbZ)$, that is $\combweightedfilling_X^{k+1}(e')=p_1^X(b)$, of the $k$-cycle 
$e'=e-h_{k-1}(\partial e)$. 
By Theorem~\ref{thm: federer-fleming thm} and Lemma~\ref{lem: polynomial equivalence}, the function $\combweightedfilling_X^i$ is polynomial for $1\le i<\frank\Gamma$.
Since $k+1\le d+1<\frank\Gamma$ there is an exponent $s\in \bbN$ and $c>0$ such that 
\[ p_1^X(b)\le c\cdot p_1^X(e')^s\]
for every $k$-cell~$e$. If $k-1\ge d$ we pick any coboundary~$b$. 
By setting $h_k(e)=b$ for every $k$-cell~ $e$ we obtain a homomorphism $C_k^\cell(X,\bbC)\to C_{k+1}^\cell(X,\bbC)$. Let $k-1<d$. By continuity of $h_{k-1}$ there is an exponent $t\in\bbN$ and $\tilde c>0$ such that $p_1^X(e')\le \tilde c\cdot p_t^X(e)$ for every $k$-cell~$e$. 
Hence we obtain that 
\[ p_1^X\bigl(h_k(e)\bigr)=p_1^X(b)\le c\cdot p_1^X(e')^s\le c\cdot \tilde c^s\cdot p_t^X(e)^s=c\cdot \tilde c^s\cdot p_{t+s}^X(e)\]
for every $k$-cell~$e$. This means that $h_k$ is continuous. 
\end{proof}

\subsection{The comparison map and the polynomial filling rank}\label{subsec: comparison map}

We combine the results of Leuzinger-Young and Bestvina-Eskin-Wortman on filling functions of arithmetic lattices with the homological results in Subsection~\ref{subsec: homological algebra for polynomial} to prove the bijectivity of the comparison map in low degrees.

\begin{prop}\label{prop: comparison map and polynomial filling rank}
Let $\Gamma$ be a finitely generated group. Let $V$ be a unitary $\Gamma$-representation. Then the comparison map $H^i_\pol(\Gamma, V)\to H^i(\Gamma,V)$ is an isomorphism for every $0\le i<\frank\Gamma$. 
\end{prop}

\begin{proof}
 The statement follows from Propositions~\ref{prop: Schwartz chain contraction and comparison map} and~\ref{prop: comparison iso from geometric exactness} and~\ref{prop: continuous chain contraction from filling}. 
\end{proof}

As a consequence of Proposition~\ref{prop: polynomial filling rank for uniform lattices} we obtain: 

\begin{cor}\label{cor: comparison map uniform lattices}
Let $\Lambda<G$ be a uniform lattice in either a semisimple group or a connected semisimple Lie group with finite center. Let $V$ be a unitary $\Lambda$-representation. Then the comparison map $H^i_\pol(\Lambda, V)\to H^i(\Lambda,V)$ is an isomorphism for every $i\ge 0$. 
\end{cor}

As a consequence of Theorem~\ref{thm: leuzinger-young} and Theorem~\ref{thm: bestvina-eskin-wortman} we obtain the following two corollaries. 

\begin{cor}
Let $\Gamma$ be an irreducible lattice in a connected semisimple Lie group~$G$ with a finite center and without compact factors. 
Let $V$ be a unitary $\Gamma$-representation. 
Then the comparison map $H^i_\pol(\Gamma, V)\to H^i(\Gamma,V)$ is an isomorphism for every $0\le i<\max\{2,\rank G\}$. 
\end{cor}

\begin{cor}
Let $k$ be a global field and let $\bfG$ be a connected, absolutely almost simple, isotropic $k$-algebraic group. 
Let $S$ be a finite set of places of $k$ that includes all archimedean ones
and let $\Gamma<\bfG(k)$ be an $S$-arithmetic subgroup.
Let $V$ be a unitary $\Gamma$-representation. 
Then the comparison map $H^i_\pol(\Gamma, V)\to H^i(\Gamma,V)$ is an isomorphism for every $0\le i<\max\{2,|S|-2\}$. 
\end{cor}

\subsection{The passage from the lattice to the semisimple group}

We examine the transition between the polynomial cohomology of semisimple (Lie) groups and the polynomial cohomology of their lattices.

To this end, we need a good understanding of unitary induction in terms of fundamental domains and lattice cocycles. 
Let $\Gamma<G$ be a lattice in a locally compact second countable group~$G$. A choice of a measurable $\Gamma$-fundamental domain $X\subset G$ determines a measurable function~$\alpha\colon G\times X\to\Gamma$, called the \emph{lattice cocycle} for $X$, defined by 
\[ \alpha(g,x)=\gamma \Leftrightarrow g^{-1}x\gamma\in X.\]
We always assume that $1\in X$. 

We obtain a left $G$-action on $X$ via \[g\cdot x\defq gx\alpha(g^{-1},x).\] 
The lattice cocycle satisfies the \emph{cocycle equation}
\begin{equation}\label{eq: cocycle equation} \alpha(gh,x)=\alpha(g,x)\alpha(h,g^{-1}\cdot x). 
\end{equation}
The Borel isomorphism $X\to G/\Gamma$ is equivariant with respect to this action and the left translation action on the target. 

Let $V$ be a unitary $\Gamma$-representation. 
We turn $L^2(X,V)$ into a unitary $G$-representation via the left $G$-action: 
\begin{equation}\label{eq: unitary induction formula} (g\cdot f)(x)=\alpha(g,x)f(g^{-1}\cdot x).
\end{equation}
It is straightforward to verify that the unitary isomorphism
\begin{equation*}
    \omega\colon L^2(X,V)\xrightarrow{\cong} I^2 (V), ~ 
    \omega(f)(g)=\alpha(g^{-1},1)f\bigl(g\alpha(g^{-1},1)\bigr) 
\end{equation*}
is $G$-equivariant. A similar statement is true for $I^2_{\loc}(V)$ and $L^2_{\loc}(X,V)$. 

\begin{definition}\label{defi: universally integrable}
  Let $\Gamma$ be a finitely generated lattice in a compactly generated locally compact second countable group~$G$ with a Haar measure $\mu$. The lattice $\Gamma<G$ is \emph{universally integrable} if there is a measurable fundamental domain $X\subset G$ of $\Gamma$ such that for the associated lattice cocycle $\alpha\colon G\times X\to \Gamma$ the following holds. 
  \begin{enumerate}[(a)]
  \item The restriction of the word metric in $G$ to $\Gamma$ is a quasi-isometric to the word metric of $\Gamma$. We denote the word length in $G$ (and $\Gamma$) by $l$.  
  \item There are constants $C,D>0$ such that \[l(\alpha(g,x))\le C+D\cdot (l(g)+l(x))\] for every $x\in X$ and every $g\in G$. 
  \item For every $g\in G$ and every $p>0$ we have  
\[ \int_X l\bigl(\alpha(g,x)\bigr)^p d\mu(x)<\infty.\]
\end{enumerate}  
In this case, we say that $X$ and $\alpha$ are \emph{universally integrable}
fundamental domain or lattice cocycle, respectively. 
\end{definition}

\begin{prop}\label{prop: universally integrable}
The following four types of lattices are universally integrable: 
\begin{enumerate}
\item A uniform lattice in a locally compact second countable group. 
\item An irreducible lattice in semisimple group of higher rank. 
\item An irreducible lattice in a connected semisimple Lie group of higher rank and with finite center. 
\item If $\Gamma<G$ is universally integrable and $\Gamma\to K$ is a homomorphism to a compact group, then the diagonal embedding $\Gamma\hookrightarrow G\times K$ is a universally integrable lattice. 
\end{enumerate}
\end{prop}

\begin{proof}
The first statement is easy to verify since $\Gamma\hookrightarrow G$ is a quasi-isometry. 
The second statement is a consequence of Theorem~\ref{thm: lubotzky-mozes-raghunathan} for (a); for (b) and (c) we refer to Shalom's paper~\cite{Shalom}*{Section~2} for a proof  and the relevant references. 
 It is a consequence of the reduction theorem of Borel-Harish-Chandra-Behr-Harder that constructs a fundamental domain by generalized Siegel sets. 
The third statement is an easy consequence of the second and Remark~\ref{rem: QI remark on G}. 

For the fourth statement, note that $G$ and $G\times K$ are quasi-isometric. 
Let $\alpha\colon G\times X\to\Gamma$ be a universally integrable lattice cocycle. Then $X\times K$ is a $\Gamma$-fundamental domain of $G\times K$. The associated cocycle $\alpha'\colon (G\times K)\times (X\times K)\to\Gamma$ satisfies $\alpha'((g,k),(x,l))=\alpha(g,x)$. Therefore $\alpha'$ is universally integrable. 
\end{proof}

\begin{prop}\label{prop: G complex for polynomial lattice}
Let $\Gamma<G$ be a universally integrable lattice.  Then the restriction $C_\pol(G^{\ast+1}, V)\to C_\pol(\Gamma^{\ast+1}, V)$ is a $\Gamma$-equivariant chain homotopy equivalence. 
\end{prop}

\begin{proof}
Let $\alpha\colon G\times X\to \Gamma$ be a universally integrable lattice cocycle. Let $\pi\colon G\to \Gamma$ be the map 
$\pi(g)=\alpha(g,1)$ for $g\in G$, that means $g^{-1}\pi(g)\in X$ for $g\in G$. The map $\pi$ is $\Gamma$-equivariant. 

Next we define a chain homotopy inverse $J^\ast$ of the restriction map.  Let $\chi\colon G\to [0,\infty)$ be a continuous function supported in a compact subset $K$ around $1\in G$ such that 
\[ \int_G\chi d\mu=1.\]
For $f\in C_\pol (\Gamma^{\ast+1}, V)$ we set $J^n(f)(g_0,\dots,g_n)$ to be 
\begin{equation}\label{eq: htp inverse}
	\int_{G^{n+1}}f(\pi(h_0),\dots, \pi(h_n))\chi(g_0^{-1}h_0)\dots\chi(g_n^{-1}h_n)d\mu(h_0)\dots d\mu(h_n).
 \end{equation}
 The integral clearly exists since the integrand is zero outside the compact set $g_0K\times\dots\times g_nK$ and the image of a compact subset under $\pi$ is finite. 
The cochain $I^n(f)$ is polynomial: Since $\alpha$ is universally integrable, there is a linear upper bound of $l(\pi(h))$ by $l(h)$. Since $f$ is a polynomial and continuous cochain 
there is thus a polynomial upper bound of $\norm{f(\pi(h_0),\dots, \pi(h_n))}$ by the lengths of the $h_i$. Since the integrand is non-zero only if each $h_i$ lies in $g_iK$ and $K$ is compact we obtain a polynomial upper bound of the integrand of $J^n(f)(g_0,\dots,g_n)$ by the lengths of the $g_i$ and thus of $J^n(f)(g_0,\dots,g_n)$ itself. It also easily follows that $J^n(f)$ is continuous on $G^{n+1}$. 

Clearly, $J^*$ is a chain map. It is a $\Gamma$-equivariant chain map because of the invariance of the Haar measure~$\mu$ and the equivariance of~$\pi$. 

Finally, we show that $J^\ast$ is chain homotopy inverse of the restriction $\res^\ast$. By discreteness of $\Gamma<G$ we may assume that the compact support $K$ of the auxiliary function $\chi$ satisfies $K\cap\Gamma=\{1\}$. Then the integrand in the formula~\eqref{eq: htp inverse} for $J^n(f)(\gamma_0,\dots,\gamma_n)$ where $(\gamma_0,\dots,\gamma_n)\in\Gamma^{n+1}$ is constant and equal to $f(\gamma_0,\dots,\gamma_n)$. Therefore, $\res^\ast\circ J^\ast=\id$. 

To define the chain homotopy $J^\ast\circ\res^\ast\simeq \id$ 
we set $S^n_i(f)(g_0,\dots, g_{n-1})$ to be 
\[ \int_{G^i} f\bigl(\pi(h_0),\dots,\pi(h_i),g_i,\dots,g_{n-1}\bigr)\chi(g_0^{-1}h_0)\dots\chi(g_i^{-1}h_i)d\mu(h_0)\dots d\mu(h_i)\]
for every $n\ge 1$ and every $0\le i\le n-1$.  

This defines a homomorphism $S^n_i\colon C_\pol(G^{n+1}, V)^\Gamma\to C_\pol(G^n, V)^\Gamma$ which is proved the same way as for $J^n$ being well defined. We leave it to the reader to verify the following simplicial relations: 
\begin{align*}
S^{n+1}_i\partial_j^n&=\partial_{j-1}^{n-1}S^n_i~\text{for $2\le i+2\le j\le n+1$}\\
S^{n+1}_i\partial_j^n&=\partial^{n-1}_jS^n_{i-1}~\text{for $0\le j\le i-1\le n-1$}\\
S^{n+1}_i\partial_{i+1}^n&=S^{n+1}_{i+1}\partial^n_{i+1}~\text{ for $0\le i\le n-1$}\\
S^{n+1}_0\partial^n_0&=\id\\
S^{n+1}_n\partial^n_{n+1}&=J^n\circ\res^n
\end{align*}
The desired chain homotopy is given as $H^n=\sum_{i=0}^{n-1} (-1)^iS^n_i$. A straightforward calculation yields 
\[ 
\partial^{n-1} H^n+H^{n+1}\partial^n=\id-J^n\circ \res^n.\qedhere
\]
\end{proof}

\begin{theorem}\label{thm: restriction is injective in poly cohomology}
Let $\Gamma<G$ be a universally integrable lattice. Then the restriction $H_\pol^\ast(G,V)\to H_\pol^\ast(\Gamma, V)$ is split injective in all degrees. Further, if $\Gamma<G$ is uniform, then also the restriction $H^\ast_c(G,V)\to H^\ast(\Gamma, V)$ is split injective in all degrees, and the comparison map for $G$ and $V$ is a retract of the comparison map for $\Gamma$ and $V$. 
\end{theorem}

\begin{proof}
We factorize the restriction map on the chain level in an obvious way into two chain maps such that the first chain is split injective and second chain map is a chain homotopy equivalence: 

\begin{tikzcd}
  C_\pol(G^{\ast+1}, V)^G \arrow[r] & C_\pol(G^{\ast+1}, V)^\Gamma\arrow[r, "\simeq"]\arrow[l, dotted, bend right] & C_\pol(\Gamma^{\ast+1}, V)^\Gamma
\end{tikzcd}

The second map is a chain homotopy equivalence by Proposition~\ref{prop: G complex for polynomial lattice}. It remains to prove the injectivity of the first map in cohomology. The first chain map has indeed a left inverse -- indicated by the dotted arrow -- that sends $f\in C_\pol (G^{n+1},V)^\Gamma$ to 
\begin{equation}\label{eq: split}
(g_0,\dots,g_n)\mapsto\int_{G/\Gamma} gf(g^{-1}g_0,\dots,g^{-1}g_n)d\mu(g).
\end{equation}
The integrand does not depend on the choice of the representative $g\in G/\Gamma$. Once we proved that the integral exists and is bounded polynomially in the lengths of the $g_i$ it is obvious that the above map defines a left inverse chain map. Let $X\subset G$ be a universally integrable fundamental domain of~$\Gamma$. 
We rewrite the above integral as 
\[ 
\int_X xf(x^{-1}g_0, \dots, x^{-1}g_n)d\mu(x).
\] 
As $f$ is a polynomial cochain there is $p\in\bbN$ such that 
\[\norm{xf(x^{-1}g_0, \dots, x^{-1}g_n)}\lessapprox \bigl(l(x^{-1}g_0)+\dots+l(x^{-1}g_n)\bigr)^p\lessapprox l(x)^p+l(g_0)^p+\dots+l(g_n)^p\]
for $x\in X$ and $g_0,\dots, g_n\in G$.  
By universal integrability the integral exists and 
\[ 
\int_X xf(x^{-1}g_0, \dots, x^{-1}g_n)d\mu(x)\lessapprox l(g_0)^p+\dots+l(g_n)^p.
\]
For the statement about the ordinary cohomology, we assume now that $\Gamma<G$ is uniform. The formula~\eqref{eq: split} defines a left inverse of the restriction $C(G^{\ast+1}, V)^G\to C(G^{\ast+1},V)^\Gamma$: The integral over $G/\Gamma$ of a continuous cochain exists by compactness. Since the left inverses for the polynomial and ordinary chain complexes are given by the same expression~\eqref{eq: split}, the comparison map for $G$ is a retract of the one of~$\Gamma$. 
\end{proof}

\subsection{The comparison and restriction map for semisimple groups and their lattices} \label{subsec:Shapiro++}

In this section we state and prove Theorem~\ref{thm: Shapiro surjectivity} and some of its applications.
This theorem is an extension of the classical Shapiro Lemma, which was discussed in \S\ref{subsec:SL} in the setting of cocompact lattices and involves the induction modules $I^2(V)$ and $I^2_{\loc}(V)$ which were introduce there.
We will now review the general setting of the Shapiro Lemma for a locally compact second countable group~$G$ and a lattice $\Gamma<G$, which is not necessarily cocompact.

The chain complex $C(G^{\ast+1},V)$ is contained in the chain complex $L^2_{\loc}(G^{\ast+1}, V)$ of locally square-integrable functions. The inclusion is a chain homotopy equivalence~\cite{Blanc}*{3.5.~Corollaire}, so both complexes compute the 
continuous group cohomology $H^\ast_c(G,V)$. The chain map  
\begin{equation}\label{eq: shapiro chain iso} 
L^2_{\loc}(G^{\ast+1}, V)^\Gamma
\xrightarrow{\cong} 
L^2_{\loc}\bigl(G^{\ast+1}, I^2_{\loc}(V)\bigr)^G 
\end{equation}
that sends $f\colon G^{n+1}\to V$ to $(g_0,\dots,g_n;g)\mapsto f(g^{-1}g_0,\dots g^{-1}g_n)$ is an isomorphism~\cite{Blanc}*{8.6~Proposition}. 

Let $X\subset G$ be a fundamental domain of $\Gamma$, and let $\alpha$ be the associated lattice cocycle and $\pi(g)=\alpha(g,1)$. The map $\pi\colon G\to \Gamma$ is $\Gamma$-equivariant and measurable. 

\begin{lemma}\label{lem: chain map between two realizations}
The map $C(\Gamma^{\ast+1}, V)^\Gamma\to L^2_{\loc}(G^{\ast+1}, V)^\Gamma,~f\mapsto f(\pi(\_),\dots, \pi(\_)),$ is a chain homotopy equivalence. 
\end{lemma}

\begin{proof}
The map $\pi$ maps compact sets to finite sets. Therefore the image of $f$ is locally square-integrable. It is obvious that $f$ is chain map. By~\cite{Blanc}*{3.2.1~Proposition and 3.4~Th\'eor\`eme}
the chain complex $L^2_{\loc}(G^{\ast+1}, V)$ is a strong relatively injective $\Gamma$-resolution of~$V$. The same is true for $C(\Gamma^{\ast+1},V)$. By the fundamental theorem of relative homological algebra~\cite{Blanc}*{2.9~Proposition} the chain map is a chain homotopy equivalence. 
\end{proof}

Let us consider the homomorphism 
\begin{equation}\label{eq: Shapiro}
 H_c^i(G, I^2(V))\to H^i(\Gamma, V)
 \end{equation}
which is induced by the composition of chain maps 
\[ L^2_{\loc}\bigl(G^{\ast+1}, I^2(V)\bigr)^G\to L^2_{\loc}\bigl(G^{\ast+1}, I^2_{\loc}(V)\bigr)^G\xleftarrow{\cong} L^2_{\loc}(G^{\ast+1}, V)^\Gamma\xleftarrow{\simeq} C(\Gamma^{\ast+1},V)^\Gamma.\]
The second map is the chain isomorphism~\eqref{eq: shapiro chain iso}, the third map is any chain homotopy inverse of the map in Lemma~\ref{lem: chain map between two realizations}. The first map is induced by the continuous inclusion 
$ I^2(V)\hookrightarrow I^2_{\loc}(V)$
of Fr\'echet spaces. 

\begin{theorem}\label{thm: Shapiro surjectivity}
Let $\Gamma<G$ be a universally integrable lattice. Let $V$ be a unitary $\Gamma$-representation. The homomorphism~\eqref{eq: Shapiro} 
is surjective for every degree $0\le i<\frank\Gamma$. 
In case $V$ is a $G$-representation, the composition of \eqref{eq: Shapiro} with the map $H_c^i(G, V)\to H_c^i(G, I^2(V))$ induced by the inclusion $V\hookrightarrow I^2(V)$ is the restriction map $H_c^i(G, V)\to H^i(\Gamma, V)$.
\end{theorem}

\begin{proof}
Consider the following chain map which we refer to as the \emph{induction map},
\begin{gather*}
I^\ast\colon C_\pol(\Gamma^{\ast+1}, V\bigr)^\Gamma\to L^2_{\loc}\bigl(G^{\ast+1}, L^2(X,V)\bigr)^G\\ I(f)(g_0,\dots,g_n)(x)=f(\alpha(g_0,x),\dots, \alpha(g_n,x)).
\end{gather*}
We assume that $\alpha$ is universally integrable and 
verify that $I^\ast$ is well defined. Let $f$ be a polynomial cochain of the left hand side. The $G$-invariance of $I(f)$ is a direct consequence of~\eqref{eq: cocycle equation} and~\eqref{eq: unitary induction formula}. There is $k\in\bbN$ such that  
\[\norm{f(g_0,\dots,g_n)}\lessapprox \bigl(l(g_0)+\dots+l(g_n)\bigr)^k\lessapprox l(g_0)^k+\dots l(g_n)^k\]
for $(n+1)$-tuples in~$G$. 
By universal integrability it follows that $I(f)(g_0,\dots, g_n)$ is square integrable over~$X$ for every $(g_0, \dots, g_n)\in G^{n+1}$. Similarly, $I(f)$ is locally square integrable. 
Consider the following diagram 
\[
\begin{tikzcd}
C_\pol\bigl(\Gamma^{\ast+1},V\bigr)^\Gamma\arrow[dd]\arrow[r, "I^\ast"] & L^2_{\loc}\bigl(G^{\ast+1}, L^2(X,V)\bigr)^G\arrow[d]\\
& L^2_{\loc}\bigl(G^{\ast+1}, L^2_{\loc}(X,V)\bigr)^G\\
C\bigl(\Gamma^{\ast+1},V\bigr)^\Gamma\arrow[r, "\simeq"] &L^2_{\loc}\bigl(G^{\ast+1}, V\bigr)^\Gamma\arrow[u, "\cong"]
\end{tikzcd}
\]
Let $f$ be an element in the left top corner. If we send $f$ via the lower path to the middle module we obtain the cochain $(g_0,\dots, g_n;x)\mapsto f(\pi(x^{-1}g_0),\dots, \pi(x^{-1}g_n))$. Via the upper path we obtain the cochain $(g_0,\dots, g_n;x)\mapsto f(\alpha(g_0,x),\dots, \alpha(g_n,x))$. One easily sees that $\alpha(g,x)=\pi(x^{-1}g)$. 
So the diagram commutes. 

Taking cohomology, the diagram shows that the comparison homomorphism $H ^\ast_\pol(\Gamma,V)\to H^\ast(\Gamma, V)$ 
is the composition of $H^\ast(I^\ast)$ and the map~\eqref{eq: Shapiro}. Thus~\eqref{eq: Shapiro} is surjective whenever the comparison map is surjective. Proposition~\ref{prop: comparison map and polynomial filling rank} concludes the proof of the first statement.     
The second statement is straightforward.
\end{proof}

We prove that the comparison map for semisimple groups is an isomorphism. This will yield injectivity of the restriction map from a semisimple group to a lattice in low degrees.

\begin{theorem}\label{thm: comparison map for universally integrable groups}
Assume that $G$ is a compactly generated locally compact second countable group that has a uniform lattice whose polynomial filling rank is~$\infty$. 
Then the comparison map 
  $H_\pol^\ast(G; V)\to H^\ast_c(G; V)$ is an isomorphism in all degrees for every unitary $G$-representation $V$. 
  \end{theorem}

\begin{proof}
Let $V$ be a unitary $G$-representation. 
Choose a uniform lattice $\Lambda<G$. The comparison map 
$H_\pol^\ast(\Lambda, V)\to H^\ast(\Lambda, V)$ is an isomorphism by Corollary~\ref{cor: comparison map uniform lattices}. Since the comparison map $H_\pol^\ast(G, V)\to H^\ast_c(G, V)$ is a retract of the isomorphism $H_\pol^\ast(\Lambda, V)\to H^\ast(\Lambda, V)$ by Theorem~\ref{thm: restriction is injective in poly cohomology},  it is itself an isomorphism. 
\end{proof}

The previous theorem in combination with Proposition~\ref{prop: polynomial filling rank for uniform lattices} implies the following result.

\begin{theorem}\label{thm: comparison map for semisimple Lie}
Every semisimple group and every connected semisimple Lie group with finite center satisfies the assumption of Theorem~\ref{thm: comparison map for universally integrable groups}.
In particular, for such a group $G$, the comparison map 
  $H_\pol^\ast(G; V)\to H^\ast_c(G; V)$ 
  is an isomorphism in all degrees for every unitary $G$-representation $V$. 
\end{theorem}

\begin{theorem}\label{thm: injectivity G to lattice}
Let $G$ be a group satisfying the assumption of Theorem~\ref{thm: comparison map for universally integrable groups}. 
If $\Gamma<G$ is a universally integrable lattice, then the restriction $H_c^i(G,V)\to H^i(\Gamma, V)$ is injective for every $0\le i<\frank \Gamma$. 
\end{theorem}

\begin{proof}
The square with restriction maps as horizontal arrows and comparison maps as vertical arrows commutes. 
\[\begin{tikzcd}
H_\pol^\ast(G,V)\arrow[r, hook]\arrow[d, "\cong"]& H_\pol^\ast(\Gamma,V)\arrow[d, "\cong"]\\
H_c^\ast(G,V)\arrow[r]& H^\ast(\Gamma,V)    
\end{tikzcd}
\]
The left vertical map is an isomorphism in all degrees by assumption. The right vertical map is an isomorphism in degrees below $\frank\Gamma$ by Proposition~\ref{prop: comparison map and polynomial filling rank}. The upper horizontal map is injective by Theorem~\ref{thm: restriction is injective in poly cohomology}. This finishes the proof. 
\end{proof}

\section{Proofs of the main theorems and additional results} \label{sec:pfmain}

In the first two subsections we prove the theorems stated in the introduction and also present some additional results. In~\S\ref{subsec: simple connectivity} we discuss the role of $\bfG$ being simply connected (or not) in Theorem~\ref{thm:arith} and introduce Theorem~\ref{thm:adelicC}. Finally, in~\S\ref{subsec: bekka-cowling} we completely determine the \emph{cohomological} unitary dual of $\bfG(k)$ using Theorems~\ref{thm:adelic} and~\ref{thm:GknonGA} -- despite the fact that $\bfG(k)$ is not a type I group and there is no chance to determine the whole unitary dual. 

\subsection{Theorems about lattices in Lie  groups}

Below we prove Theorems~\ref{thm:semisimpleT},~\ref{thm: bijectivity Lie to lattice}, and~\ref{thm:nonuniformextension}. 

We describe the common starting point of Theorems~\ref{thm:semisimpleT},~\ref{thm: bijectivity Lie to lattice}, and~\ref{thm:nonuniformextension} before we delve into the specific cases. In Theorem~\ref{thm:nonuniformextension} we define $V_0$ to be the subspace of $G$-continuous vectors as in 
Definition~\ref{def: G-continuous vectors}. 

In all cases we may and will assume that $G$ is of higher rank since the degree~$0$ case is obvious (Theorem~\ref{thm:semisimpleT}), follows from Howe-Moore (Theorem~\ref{thm: bijectivity Lie to lattice}) or follows from $V_0^G=V^\Gamma$ (Theorem~\ref{thm:nonuniformextension}) according to the discussion after Definition~\ref{def: G-continuous vectors}. 
We will use the polynomial filling rank of $\Gamma$ given in Definition~\ref{def:polyrank}.
By Theorem~\ref{thm: leuzinger-young} and Proposition~\ref{prop: polynomial filling rank for uniform lattices} it satisfies $\frank\Gamma=\rank G$ if $\Gamma$ is non-uniform or $\frank\Gamma=\infty$ if $\Gamma$ is uniform. Moreover, by Proposition~\ref{prop: universally integrable}(3), in all three cases $\Gamma<G$ is a universally integrable lattice, as defined in Definition\ref{defi: universally integrable}.

\begin{proof}[Conclusion of proof of Theorem~\ref{thm:semisimpleT}]
Using Lemma~\ref{lem: product with compact group}, we assume that $G$ has no compact factors.
By Theorem~\ref{thm:Lier0}, $G$ has property $(T_{r_0(G)-1})$.
By Theorem~\ref{thm: Shapiro surjectivity} we have an epimorphism \[H_c^j(G,I^2(V))\to H^j(\Gamma, V).\] 
    for every $j<\rank G \leq \frank\Gamma$. 
    Thus, if $V^\Gamma=I^2(V)^G=0$, then $H_c^j(G, I^2(V))=0$ for $j<n+1=\min\{r_0(G),\rank(G)\}$
    and we conclude that $\Gamma$ has property $(T_n)$.
    If $m+1$ equals the minimal rank of a simple factor of $G$ then $m\leq n$,
    thus both $G$ and $\Gamma$ have property $(T_m)$.
    This finishes the proof.
\end{proof}

\begin{proof}[Conclusion of proof of Theorem~\ref{thm: bijectivity Lie to lattice}]
Let $0\le j<\rank G$. 
By Theorem~\ref{thm: comparison map for semisimple Lie}, $G$ satisfies the assumption of Theorem~\ref{thm: comparison map for universally integrable groups}. 
The injectivity of the restriction map $H_c^j(G,V)\to H^j(\Gamma, V)$ follows from Theorem~\ref{thm: injectivity G to lattice}.
By Corollary~\ref{cor:Liespecgap} and Theorem~\ref{thm:Lie-vanishss-general} and $r(G)\ge\rank G$, $H^j_c\bigl(G,L^2_0(G/\Gamma)\bar{\otimes} V\bigr )=0$. 
Thus, the map 
\[ H_c^j(G,V) \to H^j_c\bigl(G,L^2_0(G/\Gamma)\bar{\otimes} V \bigr) \cong H^j_c\bigl(G,I^2(V)\bigr) \]
induced by the inclusion $V \hookrightarrow I^2(V)$ is an isomorphism. 
We conclude by Theorem~\ref{thm: Shapiro surjectivity} that 
$H_c^j(G,V)\to H^j(\Gamma, V)$ is surjective.
\end{proof}

\begin{proof}[{Proof of Theorem~\ref{thm:nonuniformextension}}]
    The theorem clearly follows from Theorem~\ref{thm: bijectivity Lie to lattice}, once we show $H^j(\Gamma,V_0^\perp)=0$.
    We thus assume that $V=V_0^\perp$, that is, $V_0=0$, and will show that $H^j(\Gamma,V)=0$.
    Every almost simple factors $G_i$ of $G$ has T since $G$ does.
    By the discussion after Definition~\ref{def: G-continuous vectors}, $I^2(V)^{G_i}=0$. Hence $I^2(G)$ has no almost invariant vectors as a $G_i$-representation. 
    By Theorem~\ref{thm:Lie-vanishss-general} and $r(G)\ge\rank G$ we get that $H^j_c(G,I^2(V))=0$, whence $H^j(\Gamma,V)=0$ by Theorem~\ref{thm: Shapiro surjectivity}.  
\end{proof}

\subsection{Theorems about adelic groups}

The next result is an analogue of Theorem~\ref{thm:nonuniformextension}, where $G$ is more general, but the bound on the cohomology degrees is worse. 

\begin{theorem}\label{thm: bijectivity arithmetic}
Let $k$ be a number field and let $\calO$ be its ring of integers.
Let $S$ be a non-empty finite set of places of $k$ that includes all archimedean places. Let $\calO_S\subset k$ be the ring of $S$-integers of $k$. 
Let $\bfG$ be a connected, simply connected, isotropic and absolutely almost simple $k$-algebraic group that is defined over $\calO_S$. 

Set $\Gamma=\bfG(\calO_S)$,
$G=\prod_{s\in S} \bfG(k_s)$ and $K=\prod_{s\notin S} \bfG(\mathcal{O}_s)$
and consider $\Gamma$ as a lattice in $G\times K$ via the obvious embedding.

Then for every unitary representation $V$ of $G\times K$ and for every $0\le i<|S|-1$, the restriction $H_c^i(G\times K,V)\to H^i(\Gamma, V)$ is an isomorphism. 
\end{theorem}

\begin{proof}
By Theorem~\ref{thm: bestvina-eskin-wortman} we have that $\frank \Gamma\ge |S|-2$.
If $|S|=1$ then the statement is empty, thus we assume as we may that $\Gamma$ is of higher rank. 
We consider a unitary representation $V$ of $G\times K$ and $i<|S|-2$ and the comparison map $H_c^i(G,V)\to H^i(\Gamma, V)$.

Every simple factor of $G$ is a compactly generated group by \cite[Chapter I, Corollary (2.3.5)]{Margulis} which has a cocompact lattice by \cite{Borel-Harder}, hence so is the case also for $G$.
By Proposition~\ref{prop: polynomial filling rank for uniform lattices}
we get that $G$ is a compactly generated group which has a cocompact lattice $\Lambda$ whose polynomial filling rank is~$\infty$. 
Viewing $\Lambda$ as a lattice in $G\times K$, we get that also $G\times K$ is a compactly generated group which has a cocompact lattice whose polynomial filling rank is~$\infty$.
By Proposition~\ref{prop: universally integrable}(2) the projection of $\Gamma$ to $G$ is universally integrable, thus by Proposition~\ref{prop: universally integrable}(4), $\Gamma$ is universally integrable in $G\times K$. 
By Theorem~\ref{thm: injectivity G to lattice} the restriction $H_c^i(G\times K,V)\to H^i(\Gamma, V)$ is injective.

We note that $(G\times K)/\Gamma$ embeds $(G\times K)$-equivariantly as an open subset into $\bfG(\bbA(k))/\bfG(k)$, thus we get  embeddings of unitary $G\times K$-representations,
\[ \Bigl(L^2_0\bigl((G\times K)/\Gamma\bigr)\bar{\otimes} V\Bigr)^K \hookrightarrow L^2_0\bigl((G\times K)/\Gamma\bigr)\bar{\otimes} V \hookrightarrow L^2_0\bigl(\bfG(\bbA(k))/\bfG(k)\bigr)\bar{\otimes} V.\]
By Corollary~\ref{cor:Clozel++}, for every $s\in S$ the representation 
on the right does not have $\bfG(k_s)$-almost invariant vectors.
It follows that also the representation on the left does not have $\bfG(k_s)$-almost invariant vectors.
By Theorem~\ref{thm:vanishss-general} we get that $H^*_c(G,(L^2_0(G\times K/\Gamma)\bar{\otimes} V)^K)=0$ and we conclude by Lemma~\ref{lem: product with compact group}
that $H^*_c(G \times K,L^2_0(G\times K/\Gamma)\bar{\otimes} V)=0$.
Thus the map 
\[ H_c^i\bigl(G\times K,V\bigr) \to H^*_c\bigl(G\times K,L^2\bigl((G\times K)/\Gamma\bigr)\bar{\otimes} V\bigr) \cong H^*_c\bigl(G\times K,I^{2} V\bigr) \]
induced by the inclusion $V \hookrightarrow I^{2} V$ is an isomorphism. 
We conclude by Theorem~\ref{thm: Shapiro surjectivity} that 
$H_c^i(G \times K,V)\to H^i(\Gamma, V)$ is surjective.
\end{proof}

Conjecturally, the bound on the cohomology degrees in Theorem~\ref{thm: bijectivity arithmetic} could be improved. Nevertheless, this bound is good enough for the following proof.

\begin{proof}[Proof of Theorem~\ref{thm:adelic}]
By forming a suitable restriction of scalars, we may assume that $\bfG$ is absolutely almost simple~see \cite[3.1.2]{Tits}.
The restriction of scalars is compatible with the structure of the adelic group as a topological group by~\cite[Example 4.2]{Conrad}.
Denote $\Gamma=\bfG(k)$ and $G=\bfG(\bbA(k))$.
If $\bfG$ is $k$-anisotropic then $\Gamma<G$ is cocompact
and we are left, as in the proof of Theorem~\ref{thm:isoss}, to show that $H^*_c\bigl(G, L^2_0(G/\Gamma)\bar\otimes V\bigr)=0$.
This follows from Corollary~\ref{cor:Clozel++} and Theorem~\ref{thm:adelvanish}.
We thus assume from now on that $\bfG$ is $k$-isotropic.

By \cite[Theorem 3.4(1)]{Conrad}, $\bfG$ is defined over $\calO_{S_0}$, where $S_0$ is a finite set of places of $k$, containing all archimedean ones.
Let $S$ be a finite non-empty set of places of $k$ containing $S_0$. Let $\bbA_{S}(k)\subset\bbA(k)$ be the open subring of the adeles of $k$ that are integral at all places away from~$S$. 
The topological ring $\bbA(k)$ is a colimit of the rings $\bbA_{S}(k)$. Similarly, by \cite[Theorem 3.4(1)]{Conrad},
\[ \mathbf{G}(\bbA(k))\cong \colim_S \mathbf{G}\bigl(\bbA_{S}(k)\bigr).\]
Therefore the $n$-cochain group $C\bigl(\mathbf{G}(\bbA(k))^{n}, V\bigr)\cong C\bigl(\mathbf{G}(\bbA(k))^{n+1}, V\bigr)^{\mathbf{G}(\bbA(k))}$ for $\mathbf{G}(\bbA(k))$ is the inverse limit 
\[C\bigl(\mathbf{G}(\bbA(k))^n, V\bigr)\cong \lim_S C\bigl(\mathbf{G}(\bbA_{S}(k))^n, V\bigr). \]
The structure maps of the limit are surjective since we extend any continuous map on $\mathbf{G}(\bbA_{S}(k))$ to $\mathbf{G}(\bbA_{S'}(k))$, where $S\subset S'$, by Dugundji's generalization~\cite{dugundji}*{Theorem~4.1} of Tietze's extension theorem. So the inverse system on the right satisfies the Mittag-Leffler condition. 
The rows in the following commutative diagram are the Milnor exact sequences for inverse systems~\cite{weibel}*{Theorem~3.5.8 on p.~83}. The vertical arrow are the restriction maps. 
\[
\begin{tikzcd}
\lim^1_S H^{n-1}_c\bigl(\mathbf{G}(\bbA_{S}(k),V\bigr)\arrow[d]\arrow[r,hook]& H^n_c\bigl(\mathbf{G}(\bbA({k}),V\bigr)\arrow[d]\arrow[r,twoheadrightarrow]&\lim_S H^n_c\bigl(\mathbf{G}(\bbA_{S}(k),V\bigr)\arrow[d]\\
\lim^1_S H^{n-1}\bigl(\mathbf{G}(\calO_{S}),V\bigr)\arrow[r,hook]& H^n\bigl(\mathbf{G}(k),V\bigr)\arrow[r,twoheadrightarrow]&\lim_S H^n\bigl(\mathbf{G}(\calO_{S}),V\bigr)
\end{tikzcd}
\]
It suffices to show that for fixed~$n$ the restriction 
\begin{equation}\label{eq: finite adelic} H^n_c\bigl(\mathbf{G}(\bbA_{S}(k)),V\bigr)\to H^n\bigl(\mathbf{G}(\calO_{S}),V\bigr)
\end{equation}
is an isomorphism for $|S|\gg 1$. 
In view of \cite[Theorem 3.6]{Conrad}, this follows from Theorem~\ref{thm: bijectivity arithmetic}.
\end{proof}

\begin{remark} \label{rem:nonSC}
    Note that the only uses of the assumption that $\bfG$ is simply connected in the proofs of Theorem~\ref{thm: bijectivity arithmetic} and Theorem~\ref{thm:adelic} were via Corollary~\ref{cor:Clozel++}.
    The proofs in fact show that, even without this assumption,  the conclusion of Theorem~\ref{thm:adelic} holds true for every unitary representation $U$ of $\bfG(\bbA(k))$ such that for every place $s$ of $k$, the representation 
$L^2_0(\bfG(\bbA(k))/\bfG(k))\bar{\otimes} U$ does not have $\bfG(k_s)$-almost invariant vectors. 
\end{remark}

\begin{theorem} \label{thm:GknonGA}
Let $k$ be a number field and let $\bfG$ be a connected, simply connected almost simple group.
Assume that $\bfG(k_s)$ has property T for every place $s$ of $k$. 
    Let $V$ be a unitary $\bfG(k)$-representation.
    Then there exists a $\bfG(k)$-subrepresentation $V_0$ on which the representation extends to $\bfG(\bbA(k))$ such that $H^\ast(\bfG(k),V_0^\perp)=0$.
    In particular, the maps 
    \[ H^\ast_c(\bfG(\bbA(k)),V_0) \to H^\ast(\bfG(k),V_0) \to H^\ast(\bfG(k),V) \]
    are isomorphisms,
    where the first map comes from the restriction from $\bfG(\bbA(k))$ to $\bfG(k)$ and the second map comes from the inclusion of $V_0$ in $V$.
\end{theorem}

\begin{proof}
As in the proof of Theorem~\ref{thm:adelic}, we may and will assume that $\bfG$ is absolutely almost simple.
We set $\Gamma=\bfG(k)$, $G=\bfG(\bbA(k))$ and $G_s=\bfG(k_s)$ for every place $s$ of $k$.
Denote by $A$ the finite set of places $s$ such that $\bfG$ is $k_s$-anisotropic. By strong approximation~\cite{Margulis}*{II, Theorem 6.8}, for every place $s \notin A$, the image of $\Gamma$ is dense in $G/G_s$.
For $s\not\in A$, we let $V_s$ be the maximal subrepresentation of $V$ on which the $\Gamma$-action extends to $G/G_s$~(see the discussion after Definition~\ref{def: G-continuous vectors}). Let $V_0$ be the closure of the sum of the spaces $V_s$, $s\notin A$.
We note that the $\Gamma$-action on $V_0$ extends to a $G$-action in a unique way.
In view of Theorem~\ref{thm:adelic}, we are left to show that 
$H^\ast(\Gamma,V_0^\perp)=0$.
We may assume that $V=V_0^\perp$, that is,  $V_s=0$ for every $s\notin A$
and will show that $H^\ast(\Gamma,V)=0$.

If $\bfG$ is $k$-anisotropic this follows from Shapiro Lemma. 
Indeed, by \cite[Proposition 5.1]{BBHP} (cf.~the discussion after Definition~\ref{def: G-continuous vectors}) we obtain that for every place $s\notin A$,
$(I^2 V)^{G_s}=0$. Hence, by  Theorem~\ref{thm:adelvanish}, $H^*_c(G,I^2 V)=0$.

From now on, let $\bfG$ be $k$-isotropic. In particular, $A=\emptyset$.
Fixing $i$, we will show that $H^i(\Gamma,V)=0$. 
By \cite[Theorem 3.4(1)]{Conrad}, $\bfG$ is defined over $\calO_{S_0}$, where $S_0$ is a finite set of places of $k$, containing all archimedean ones.
Let $S$ be a finite set of places of $k$ containing $S_0$ and satisfying $|S|>i+2$. 
We use freely \cite[Theorems 3.4 and 3.6]{Conrad}.
Set $G_S=\prod_{s\in S} G_s$, $K_S= \prod_{s\notin S} \bfG(\mathcal{O}_s)$ and consider $G_S \times K_S$ as an open subgroup of $G$.
Setting $\Gamma_S=\Gamma \cap (G_S \times K_S)$, we have  
\[ G\cong \colim_S G_S \times K_S, \quad  \Gamma \cong \colim_S \Gamma_S.\]
Since $\Gamma \cdot (G_S \times K_S)$ is open and of cofinite volume in $G$, it is cofinite. It follows that for $S$ large enough, $\Gamma \cdot (G_S \times K_S)=G$.
We assume below that this is the case.
We get that the inclusion $(G_S \times K_S)/\Gamma_S \hookrightarrow G/\Gamma$ is a $(G_S \times K_S)$-equivariant bijection and we deduce the existence of an isomorphism of $(G_S \times K_S)$-representations, 
\[ I_{\Gamma_S}^{G_S \times K_S} \circ \res_{\Gamma_S}^\Gamma V \cong  \res_{G_S \times K_S}^G \circ I_{\Gamma}^G~ V. \]
Below we consider unitary induction, but we deviate from our usual notation $I^2(V)$ in order to indicate between which groups we induce. 
We obtain that for $s\in S$, 
\[ (I_{\Gamma_S}^{G_S \times K_S} \circ \res_{\Gamma_S}^\Gamma V)^{G_s} \cong
(\res_{G_S \times K_S}^G \circ I_{\Gamma}^G V)^{G_s}=(I_{\Gamma}^G V)^{G_s}=0,\]
thus 
\[ \bigl((I_{\Gamma_S}^{G_S \times K_S} \circ \res_{\Gamma_S}^\Gamma V)^{K_S}\bigr)^{G_s}=0.\]
We view $(I_{\Gamma_S}^{G_S \times K_S} \circ \res_{\Gamma_S}^\Gamma V)^{K_S}$ as a $G_S$-representation. Note that $r(G_S)\geq \rank(G_S) \geq |S|>i+2$.  
Theorem~\ref{thm:vanishss-general} implies that 
\[ H^i_c(G_S,(I_{\Gamma_S}^{G_S \times K_S} \circ \res_{\Gamma_S}^\Gamma V)^{K_S})=0.\]
We conclude from Lemma~\ref{lem: product with compact group} that 
\[ H^i_c(G_S \times K_S,I_{\Gamma_S}^{G_S \times K_S} \circ \res_{\Gamma_S}^\Gamma V)=0. \]

As in the proof of Theorem~\ref{thm: bijectivity arithmetic}, one concludes that  $\Gamma_S$ is universally integrable in $G_S \times K_S$. 
Viewing $\Gamma_S$ as a lattice in $G_S$, we get by Theorem~\ref{thm: bestvina-eskin-wortman} that $\frank \Gamma_S \ge |S|-2> i$.
We deduce from Theorem~\ref{thm: Shapiro surjectivity} that $H^i(\Gamma_S,V)=0$.
By Milnor's exact sequence for inverse systems~\cite{weibel}*{Theorem~3.5.8 on p.~83} we have 
\[ 0\to {\lim}^1_S H^{i-1}\bigl(\Gamma_S,V\bigr) \to H^i\bigl(\Gamma,V\bigr) \to \lim_S H^i\bigl(\Gamma_S,V\bigr) \to 0
\]
and we conclude that $H^i\bigl(\Gamma,V\bigr)=0$.
\end{proof}

\begin{proof}[Proof of Theorem~\ref{thm:Gkhigherrank}]
    This is a special case of Theorem~\ref{thm:GknonGA}.
\end{proof}

\begin{proof}[Proof of Theorem~\ref{thm:arith}]
For a non-archimedean local field $F$ we get either by Theorem~\ref{thm:Casselman} or by the main result in~\cite{dymara+janus} that 
\[ H^\ast(\bfG(F),V)=H^0(\bfG(F),V)=V. \] 
We conclude by an iterative application of the Hochschild-Serre spectral sequence \cite[Theorem 9.1]{Blanc} and a limiting argument as in the proof of Theorem~\ref{thm:adelvanish} that 
\[ H^\ast\bigl(\bfG(\mathbb{A}_f(k)),V\bigr)=H^0\bigl(\bfG(\mathbb{A}_f(k)),V\bigr)=V. \] 
Using the identification $\mathbb{A}(k)\cong (k\otimes \bbR) \times \mathbb{A}_f(k)$ and another application of the Hochschild-Serre spectral sequence \cite[Theorem 9.1]{Blanc}, we deduce
that the map $H_c^*(\mathbf{G}(k\otimes \bbR),V)\to H_c^*(\mathbf{G}(\mathbb{A}(k)),V)$ is an isomorphism in all degrees.
The map $H_c^*(\mathbf{G}(\mathbb{A}(k)),V)\to H^*(\mathbf{G}(k),V)$ is an isomorphism in all degrees by Theorem~\ref{thm:adelic}.
We are left to show that the map $H^*(\mathbf{G}(k),V) \to H^*(\Gamma,V)$ is an isomorphism in degrees lower than
$\rank \mathbf{G}(k\otimes \bbR)$.
Equivalently, it is enough to show that the map $H_c^*(\mathbf{G}(k\otimes \bbR),V)\to H^*(\Gamma,V)$
is an isomorphism in degrees lower than
$\rank \mathbf{G}(k\otimes \bbR)$.
For this, we use the proof of Theorem~\ref{thm: bijectivity Lie to lattice} -- except for using Theorem~\ref{thm:Clozel++arith} instead of Corollary~\ref{cor:Liespecgap}.
\end{proof}

\subsection{On the assumption of simple connectivity}\label{subsec: simple connectivity}

In the following version of Theorem~\ref{thm:arith}, which applies to the trivial representation, the group $\bfG$ is not assumed to be simply connected.

\begin{theorem} \label{thm:adelicC}
Let $k$ be a number field and $\mathbb{A}(k)$ the ring of adeles of~$k$.
Let $\mathbf{G}$ be a connected almost simple $k$-algebraic group and $\Gamma<\mathbf{G}(k)$ an associated arithmetic subgroup.
Consider the natural maps
\[ H_c^*(\mathbf{G}(k\otimes \bbR),\bbC)\to H_c^*(\mathbf{G}(\mathbb{A}(k)),\bbC)\to H^*(\mathbf{G}(k),\bbC) \to H^*(\Gamma,\bbC). \]
Then the first two maps are isomorphisms in all degrees and the last map, namely $H^*(\mathbf{G}(k),\bbC) \to H^*(\Gamma,\bbC)$, 
is an isomorphism in degrees lower than
$\rank \mathbf{G}(k\otimes \bbR)$.
\end{theorem}

The following is a version of Theorem~\ref{Clozel} in which $\bfG$ is not assumed to be simply connected.

\begin{lemma} \label{lem:SGnonSS}
    Let $k$ be a number field and $\mathbb{A}(k)$ its ring of adeles. Let $\mathbf{G}$ be a connected almost simple $k$-algebraic group.
    Then for every place $s$ of $k$ such that $\bfG$ is $k_s$-isotropic, the $\mathbf{G}(\mathbb{A}(k))$ unitary representation
    \[ L^2_0(\bfG(\bbA(k))/\bfG(k)) \]
    does not have almost invariant vectors with respect to $\mathbf{G}(k_s)$.
\end{lemma}

\begin{proof}
    We let $\phi:\tilde{\bfG} \to \bfG$ be the simply connected cover.
    As in the proof of Theorem~\ref{Burger-Sarnak}, we obtain that  
    $\bfG(k)\cdot \phi_{\bbA(k)}(\tilde{\bfG}(\bbA(k)))$ is a normal open subgroup of finite index in $\bfG(\bbA(k))$. We get a finite group     
    \[ D=\bfG\bigl(\bbA(k)\bigr)/\bigl(\bfG(k)\cdot \phi_{\bbA(k)}(\tilde{\bfG}(\bbA(k)))\bigr),\]
    a homomorphism $\bfG(\bbA(k)) \to D$ 
    and a $\bfG(\bbA(k))$-equivariant factor map
    \[ \bfG(\bbA(k))/\bfG(k) \to D. \]
    We thus get an inclusion of unitary $\bfG(\bbA(k))$-representations
    \[ L^2_0(D) \hookrightarrow L^2_0(\bfG(\bbA(k))/\bfG(k)). \]
    Since $L^2_0(D)$ is finite dimensional and it has no invariant vectors, it has no almost invariant vectors as well.
    We are left to show that its complement, $L^2_0(D)^\perp$, has no almost invariant vectors. For this, we consider $L^2_0(D)^\perp$ as $\tilde{\bfG}(\bbA(k))$-representation.
As such, we obtain from~\eqref{eq:adelicSC} that 
\[ L^2_0(D)^\perp \cong L^2_0(D) \otimes L^2_0(\tilde{\bfG}(\bbA(k))/\tilde{\bfG}(k))^{\Ker(\phi_{\bbA(k)})}. \]
Theorem~\ref{Clozel} implies that 
$L^2_0(D)^\perp$ has no $\tilde{\bfG}(\bbA(k))$-almost invariant vectors. In particular, $L^2_0(D)^\perp$ has no almost invariant vectors as a unitary $\bfG(\bbA(k))$-representation. This finishes the proof.
\end{proof}

The following corollary follows immediately, as in the first part of the proof of Theorem~\ref{thm:Clozel++arith}.

\begin{cor} \label{cor:SGnonSS}
    Let $k$ be a number field and let $\mathbf{G}$ be a connected almost simple $k$-algebraic group.
Let $S$ be a set of places of $k$ that includes all the archimedean ones, and let $\Gamma<\bfG(k)$ be an $S$-arithmetic subgroup.
Let $G$ be the restricted product of the groups $\bfG(k_s)$, $s\in S$,
and consider $\Gamma$ as a lattice in $G$.
Then for every $t\in S$ such that $G_t=\bfG(k_t)$ is non-compact, the unitary representation  
    $L^2_0(G/\Gamma)$ does not have $G_t$-almost invariant vectors
\end{cor} 

\begin{proof}[Proof of Theorem~\ref{thm:adelicC}]
The map $H_c^*(\mathbf{G}(k\otimes \bbR),\bbC)\to H_c^*(\mathbf{G}(\mathbb{A}(k)),\bbC)$ is an isomorphism in all degrees by Corollary~\ref{cor:adeletriv}.
The proof that the map $H_c^*(\mathbf{G}(\mathbb{A}(k),\bbC)\to H^*(\mathbf{G}(k),\bbC)$ is an isomorphism in all degrees
goes verbatim as the Proof of Theorem~\ref{thm:adelic}, relying on Theorem~\ref{thm: bijectivity arithmetic}, using Lemma~\ref{lem:SGnonSS} instead of Corollary~\ref{cor:Clozel++}; see Remark~\ref{rem:nonSC}.
We are left to show that the map $H^*(\mathbf{G}(k),\bbC) \to H^*(\Gamma,\bbC)$ is an isomorphism in degrees lower than
$\rank \mathbf{G}(k\otimes \bbR)$.
Equivalently, it is enough to show that the map $H_c^*(\mathbf{G}(k\otimes \bbR),\bbC)\to H^*(\Gamma,\bbC)$
is an isomorphism in degrees lower than
$\rank \mathbf{G}(k\otimes \bbR)$.
For this, we use the proof of  Theorem~\ref{thm: bijectivity Lie to lattice} -- except for using Corollary~\ref{cor:SGnonSS} instead of Corollary~\ref{cor:Liespecgap}.
\end{proof}

\begin{remark}
   Let $\hat{D}$ be the Pontryagin dual of the group $D$ appearing in the proof of Lemma~\ref{lem:SGnonSS} and regard its elements as one dimensional unitary representations of $\bfG(\bbA(k))$ via the map $\bfG(\bbA(k))\to D$.
   An inspection of the proof of Lemma~\ref{lem:SGnonSS} shows that 
for a unitary representations $U$ of $\mathbf{G}(\mathbb{A}(k))$,
if for every place $s$ for which $\bfG$ is $k_s$-isotropic and for every $1\neq \phi\in \hat{D}$, the representation $U \otimes \phi$ does not have $\bfG(k_s)$ almost invariant vectors, then 
 the $\mathbf{G}(\mathbb{A}(k))$ unitary representation
    \[ L^2_0(\bfG(\bbA(k))/\bfG(k)) \otimes L^2_0(D) \]
    does not have almost invariant vectors with respect to $\mathbf{G}(k_s)$. 
Thus, under this assumption on $U$, an analogue of Theorem~\ref{thm:adelicC} still holds when the trivial representation $\bbC$ is replaced by $U$, with essentially the same proof.
\end{remark}

\subsection{The Bekka-Cowling map}\label{subsec: bekka-cowling}

We recall the following theorem of Bekka and Cowling.

\begin{theorem}[{\cite[Theorem 1]{Bekka-Cowling}}] \label{thm:bekka-cowling}
    Let $k$ be a number field and $\bfG$ be a connected, simply connected almost simple $k$-algebraic group. Then the restriction of an irreduciceble representation of $\bfG(\bbA(k))$ to $\bfG(k)$ is irreducible.
    The induced map $\widehat{\bfG(\bbA(k))}\to \widehat{\bfG(k)}$ is injective.
\end{theorem}

We refer to the injection $\widehat{\bfG(\bbA(k))}\hookrightarrow \widehat{\bfG(k)}$ as the \emph{Bekka-Cowling map}.

\begin{theorem} \label{thm:bekka-cowling-cohom}
Let $k$ be a number field and $\bfG$ be a connected, simply connected almost simple $k$-algebraic group.
The restriction of the Bekka-Cowling map to the cohomological dual of $\bfG(\bbA(k))$ gives an injective 
map $\widehat{\bfG(\bbA(k))}_{\cohom}\hookrightarrow \widehat{\bfG(k)}_{\cohom}$ which is a homeomorphism onto its image. 
\end{theorem}

\begin{proof}
Let $G=\bfG(\bbA(k))$ and $\Gamma=\bfG(k)$.
   The image of the Bekka-Cowling map is in $\widehat{\Gamma}_{\cohom}$ by Theorem~\ref{thm:adelic}
   and it is injective by Theorem~\ref{thm:bekka-cowling}. 
   $\widehat{G}_{\cohom}$ is discrete by Theorem~\ref{thm:adeliccohom}, so we are left to show that the image in $\widehat{\Gamma}_{\cohom}$ is discrete.

   As in the proof of Theorem~\ref{thm:adeliccohom}, we fix a cohomological irreducible unitary $G$-representation $V$. 
   Let $U$ be the direct sum of all other cohomological irreducible unitary $G$-representations. We assume by contradiction that $V$ is weakly contained in $U$ as $\Gamma$-representations.
By the continuity of induction, \cite[Proposition F.3.4]{Tbook}, $I^2V\cong V\bar{\otimes} L^2(G/\Gamma)$ is weakly contained in $I^2U\cong U\bar{\otimes} L^2(G/\Gamma)$  as unitary \mbox{$G$-re}presentations.
Thus, $V$ is contained in $U\bar{\otimes} L^2(G/\Gamma)$  as unitary \mbox{$G$-re}presentations.
$V$ is not weakly contained in $U$, by the discreteness of $\widehat{G}_{\cohom}$,
so we conclude that $V$ is contained in $U\bar{\otimes} L^2_0(G/\Gamma)$  as $G$-representations.
By Theorem~\ref{thm:adeliccohom}, $V$ is trivial with respect to $G_s=\bfG(k_s)$ for almost every place $s$.
We thus get a contradiction to Corollary~\ref{cor:Clozel++}.
This finishes the proof.   
\end{proof}

Consider the unitary dual $\widehat{\bfG(k_s)}$ and its subspace $\widehat{\bfG(k_s)}_{\cohom}$ consisting of cohomological representations.

\begin{cor} \label{Gkcohomdual}
  Let $k$ be a number field and let $\bfG$ be a connected, simply connected almost simple group.
Assume that for every place $s$ of $k$, $\bfG(k_s)$ has property T.
The restriction of the Bekka-Cowling map to the cohomological dual of $\bfG(\bbA(k))$ gives a bijective homeomorphism $\widehat{\bfG(\bbA(k))}_{\cohom}\cong \widehat{\bfG(k)}_{\cohom}$.

For every $V\in \widehat{\bfG(k)}_{\cohom}$ there exist a finite set $S$ of places of $k$,
and for each $s\in S$ a cohomological irreducible unitary representations $V_s$ of $\bfG(k_s)$ such that $V\cong \bar{\otimes}_S V_s$ and 
\[ H^*(\bfG(k),V) \cong \otimes_S H^*_c(\bfG(k_s),V_s) \]
is finite dimensional. 
\end{cor}

\begin{proof}
This is immediate by Theorem~\ref{thm:bekka-cowling-cohom}, Theorem~\ref{thm:GknonGA} and Theorem~\ref{thm:adeliccohom}.
\end{proof}

\begin{remark}
    The assumption in Theorem~\ref{thm:GknonGA} and Corollary~\ref{Gkcohomdual} that the local groups $\bfG(k_s)$ have property T for every place $s$ of $k$ is clearly satisfied if $\rank_k(\bfG) \geq 2$, a notable example being $\SL_n(k)$ for $n\geq 3$,
    however there are other examples of such groups, as the following example shows.
    Fix a $k$-division algebra $D$ of degree $p$, an odd prime, and let $\bfG$ be $\SL_1(D)$, thus $\rank_k(\bfG) =0$. Then for every field extension $k'$ of $k$, either $D$ splits over $k'$, in which case $\rank_{k'}(\bfG)=p-1$, or $\rank_{k'}(\bfG)=0$.
\end{remark}

\begin{conjecture}
    The assumption in Theorem~\ref{thm:GknonGA} and Corollary~\ref{Gkcohomdual} that the local groups $\bfG(k_s)$ have property T could be removed.
\end{conjecture}

\appendix 

\section{The invariant $r(G)$} \label{sec:rG}

This appendix gives for semisimple groups the precise values of the invariants $r$ and $r_0$ defined in Definition~\ref{def:r}.
We recall from the product formula, Theorem~\ref{thm:productformula},
that for a semisimple group $G$ with factor groups $G_i$,
$r(G)=\sum_i r(G_i)$ and $r_0(G)=\min_i r(G_i)$.
For compact groups we have $r(G)=0$ and $r_0(G)=\infty$ by Lemma~\ref{lem:compact} and for almost simple groups over non-archimedean local fields, we have $r(G)=\rank(G)$ by Theorem~\ref{thm:Casselman}.
We are thus left to establish the values associated with non-compact almost simple Lie groups.
In this case we clearly have $r(G)=r_0(G)$ and, by the discussion proceeding Theorem~\ref{thm:Vogan-Zuckerman}, this value
depends on $G$ only up to a local isomorphism.

The following table lists all non-compact simple Lie groups, up to local isomorphism, together with their real rank and their invariant $r_0(G)=r(G)$.
Note that we always have $r_0(G)=r(G)\geq \rank(G)$.
The table is taken from \cite{Enright}*{Theorem 7.2} and \cite{Kumaresan}*{Theorem 2} (for complex Lie groups) and \cite{Vogan-Zuckerman}*{Theorem 8.1} (for real Lie groups).

\newpage

\begin{center}
\begin{tabular}{ |l|c|c| } 
 \hline
 Type of $G$ & $\rank_{\mathbb{R}}(G)$ & $r_0(G)=r(G)$ \\ 
 \hline \hline
 $\SL(n+1,\mathbb{C})$, $n\geq 1$ & $n$ & $n$ \\ 
 $\SO(2n+1,\mathbb{C})$, $n\geq 2$ & $n$ & $2n-1$ \\ 
 $\SO(2n,\mathbb{C})$, $n\geq 4$ & $n$ & $2n-2$ \\ 
 $\Sp(2n,\mathbb{C})$, $n\geq 3$ & $n$ & $2n-1$ \\ 
 \hline
 $E_6(\mathbb{C})$ & 6 & 16 \\
 $E_7(\mathbb{C})$ & 7 & 27 \\
 $E_8(\mathbb{C})$ & 8 & 57 \\
 $F_4(\mathbb{C})$ & 4 & 15 \\
 $G_2(\mathbb{C})$ & 2 & 5 \\
 \hline
 $\SL(n+1,\mathbb{R})$, $n\geq 1$ & $n$ & $n$ \\ 
 $\SL(n,\mathbb{H})$, $n\geq 4$ & $n-1$ & $2n-2$ \\ 
 $\SL(3,\mathbb{H})$ & $2$ & $3$ \\ 
 $\SU(p,q)$, $1\leq p\leq q$ & $p$ & $p$ \\ 
 $\SO(p,q)$, $1\leq p\leq q$ & $p$ & $p$ \\ 
 $\SO(n,\mathbb{H})$, $n\geq 4$ & $[n/2]$ & $n-1$ \\ 
 $\Sp(2n,\mathbb{R})$, $n\geq 3$ & $n$ & $n$ \\ 
$\Sp(p,q)$, $1\leq p\leq q$ & $p$ & $2p$ \\ 
 \hline
 $E_{6}^{6}$ & 6 & 13 \\
 $E_{6}^{2}$ & 4 & 8 \\
 $E_{6}^{-14}$ & 2 & 8 \\
 $E_{6}^{-26}$ & 2 & 6 \\
 $E_{7}^{7}$ & 7 & 15 \\
 $E_{7}^{-5}$ & 4 & 12 \\
 $E_{7}^{-25}$ & 3 & 11 \\
 $E_{8}^{8}$ & 8 & 29 \\
$E_{8}^{-24}$ & 4 & 24 \\
$F_{4}^{4}$ & 4 & 8 \\
$F_{4}^{-20}$ & 1 & 4 \\
$G_{2}^{2}$ & 2 & 3 \\
 \hline
\end{tabular}
\end{center}


\begin{bibdiv}
\begin{biblist}
\bib{abrams_etal}{article}{
   author={Abrams, Aaron},
   author={Brady, Noel},
   author={Dani, Pallavi},
   author={Duchin, Moon},
   author={Young, Robert},
   title={Pushing fillings in right-angled Artin groups},
   journal={J. Lond. Math. Soc. (2)},
   volume={87},
   date={2013},
   number={3},
   pages={663--688},
}

	

\bib{BBHP}{article}{
    AUTHOR = {Bader, Uri},
    AUTHOR = {Boutonnet, R\'{e}mi},
    AUTHOR = {Houdayer, Cyril},
    AUTHOR = {Peterson, Jesse},
     TITLE = {Charmenability of arithmetic groups of product type},
   JOURNAL = {Invent. Math.},
  FJOURNAL = {Inventiones Mathematicae},
    VOLUME = {229},
      YEAR = {2022},
    NUMBER = {3},
     PAGES = {929--985},
}
\bib{bader+furman+gelander+monod}{article}{
 author={Bader, Uri},
 author={Furman, Alex},
 author={Gelander, Tsachik},
 author={Monod, Nicolas},
 title={Property {{\((T)\)}} and rigidity for actions on Banach spaces},
 journal={Acta Mathematica},
 volume={198},
 number={1},
 pages={57--105},
 date={2007},
}
\bib{stability_four}{misc}{
   author={Bader, Uri},
   author={Lubotzky, Alexander},
   author={Sauer, Roman},
   author={Weinberger, Shmuel},
   title={Stability and instability of lattices in semisimple groups},
   year={2023},
   note={arXiv:2303.08943},
}
\bib{BN1}{article}{
    AUTHOR = {Bader, Uri},
    AUTHOR = {Nowak, Piotr W.},
     TITLE = {Cohomology of deformations},
   JOURNAL = {J. Topol. Anal.},
  FJOURNAL = {Journal of Topology and Analysis},
    VOLUME = {7},
      YEAR = {2015},
    NUMBER = {1},
     PAGES = {81--104},
}

\bib{BN2}{article}{
    AUTHOR = {Bader, Uri},
    AUTHOR = {Nowak, Piotr W.},
     TITLE = {Group algebra criteria for vanishing of cohomology},
   JOURNAL = {J. Funct. Anal.},
  FJOURNAL = {Journal of Functional Analysis},
    VOLUME = {279},
      YEAR = {2020},
    NUMBER = {11},
     PAGES = {108730, 18},
}
\bib{survey}{article}{
  author = {Bader, Uri},
  author = {Sauer, Roman},
  title  = {Higher property T and below-rank phenomena of lattices},
  year   = {2025},
  eprint = {arXiv:2511.20192},
}

\bib{waist}{article}{
  author = {Bader, Uri},
  author = {Sauer, Roman},
  title = {Uniform waist inequalities in codimension two for manifolds with Kazhdan fundamental group},
  year = {2024},
  eprint = {arXiv:2407.19783},
}

\bib{Bekka}{article}{
    AUTHOR = {Bekka, M. B.},
     TITLE = {On uniqueness of invariant means},
   JOURNAL = {Proc. Amer. Math. Soc.},
  FJOURNAL = {Proceedings of the American Mathematical Society},
    VOLUME = {126},
      YEAR = {1998},
    NUMBER = {2},
     PAGES = {507--514},
}

\bib{Bekka-Cowling}{article}{
    AUTHOR = {Bekka, M. E. B.},
    AUTHOR = {Cowling, M.},
     TITLE = {Some irreducible unitary representations of {$G(K)$} for a
              simple algebraic group {$G$} over an algebraic number field
              {$K$}},
   JOURNAL = {Math. Z.},
  FJOURNAL = {Mathematische Zeitschrift},
    VOLUME = {241},
      YEAR = {2002},
    NUMBER = {4},
     PAGES = {731--741},
}

\bib{Tbook}{book}{
    AUTHOR = {Bekka, Bachir},
    AUTHOR = {de la Harpe, Pierre},
    AUTHOR = {Valette, Alain},
     TITLE = {Kazhdan's property ({T})},
    SERIES = {New Mathematical Monographs},
    VOLUME = {11},
 PUBLISHER = {Cambridge University Press, Cambridge},
      YEAR = {2008},
     PAGES = {xiv+472},
}

 \bib{Bernstein}{article}{
    AUTHOR = {Bern\v{s}te\u{\i}n, I. N.},
     TITLE = {All reductive {${\germ p}$}-adic groups are of type {I}},
   JOURNAL = {Funkcional. Anal. i Prilo\v{z}en.},
  FJOURNAL = {Akademija Nauk SSSR. Funkcional\cprime nyi Analiz i ego Prilo\v{z}enija},
    VOLUME = {8},
      YEAR = {1974},
    NUMBER = {2},
     PAGES = {3--6},
}

\bib{bestvina+eskin+wortman}{article}{
   author={Bestvina, Mladen},
   author={Eskin, Alex},
   author={Wortman, Kevin},
   title={Filling boundaries of coarse manifolds in semisimple and solvable
   arithmetic groups},
   journal={J. Eur. Math. Soc. (JEMS)},
   volume={15},
   date={2013},
   number={6},
   pages={2165--2195},
}
		
\bib{Blanc}{article}{
    AUTHOR = {Blanc, Philippe},
     TITLE = {Sur la cohomologie continue des groupes localement compacts},
   JOURNAL = {Ann. Sci. \'{E}cole Norm. Sup. (4)},
  FJOURNAL = {Annales Scientifiques de l'\'{E}cole Normale Sup\'{e}rieure. Quatri\`eme
              S\'{e}rie},
    VOLUME = {12},
      YEAR = {1979},
    NUMBER = {2},
     PAGES = {137--168},
}

\bib{BFG}{article}{
    AUTHOR = {Blasius, Don},
    AUTHOR = {Franke, Jens},
    AUTHOR = {Grunewald, Fritz},
     TITLE = {Cohomology of {$S$}-arithmetic subgroups in the number field
              case},
   JOURNAL = {Invent. Math.},
  FJOURNAL = {Inventiones Mathematicae},
    VOLUME = {116},
      YEAR = {1994},
    NUMBER = {1-3},
     PAGES = {75--93},
}

\bib{Borel-stable}{article}{
    AUTHOR = {Borel, Armand},
     TITLE = {Stable real cohomology of arithmetic groups},
   JOURNAL = {Ann. Sci. \'{E}cole Norm. Sup. (4)},
  FJOURNAL = {Annales Scientifiques de l'\'{E}cole Normale Sup\'{e}rieure.
              Quatri\`eme S\'{e}rie},
    VOLUME = {7},
      YEAR = {1974},
     PAGES = {235--272},
}
\bib{Borel-Harder}{article}{
    AUTHOR = {Borel, A.},
    AUTHOR = {Harder, G.},
     TITLE = {Existence of discrete cocompact subgroups of reductive groups
              over local fields},
   JOURNAL = {J. Reine Angew. Math.},
  FJOURNAL = {Journal f\"{u}r die Reine und Angewandte Mathematik. [Crelle's
              Journal]},
    VOLUME = {298},
      YEAR = {1978},
     PAGES = {53--64},
}

\bib{Borel-Wallach}{book}{
    AUTHOR = {Borel, A.},
    AUTHOR = {Wallach, N.},
     TITLE = {Continuous cohomology, discrete subgroups, and representations
              of reductive groups},
    SERIES = {Mathematical Surveys and Monographs},
    VOLUME = {67},
   EDITION = {Second},
 PUBLISHER = {American Mathematical Society, Providence, RI},
      YEAR = {2000},
     PAGES = {xviii+260},
}

\bib{Borel-Yang}{article}{
    AUTHOR = {Borel, Armand},
    AUTHOR = {Yang, Jun},
     TITLE = {The rank conjecture for number fields},
   JOURNAL = {Math. Res. Lett.},
  FJOURNAL = {Mathematical Research Letters},
    VOLUME = {1},
      YEAR = {1994},
    NUMBER = {6},
     PAGES = {689--699},
}
\bib{Burger-Sarnak}{article}{
    AUTHOR = {Burger, M.},
    AUTHOR = {Sarnak, P.},
     TITLE = {Ramanujan duals. {II}},
   JOURNAL = {Invent. Math.},
  FJOURNAL = {Inventiones Mathematicae},
    VOLUME = {106},
      YEAR = {1991},
    NUMBER = {1},
     PAGES = {1--11},
}

\bib{Casselman}{article}{
    AUTHOR = {Casselman, W.},
     TITLE = {On a {$p$}-adic vanishing theorem of {G}arland},
   JOURNAL = {Bull. Amer. Math. Soc.},
  FJOURNAL = {Bulletin of the American Mathematical Society},
    VOLUME = {80},
      YEAR = {1974},
     PAGES = {1001--1004},
}

\bib{Cowling}{article}{
    AUTHOR = {Cowling, Michael},
     TITLE = {The {K}unze-{S}tein phenomenon},
   JOURNAL = {Ann. of Math. (2)},
  FJOURNAL = {Annals of Mathematics. Second Series},
    VOLUME = {107},
      YEAR = {1978},
    NUMBER = {2},
     PAGES = {209--234},
}
\bib{conrad-reductive}{article}{
   author={Conrad, Brian},
   title={Reductive group schemes},
   conference={
      title={Autour des sch\'{e}mas en groupes. Vol. I},
   },
   book={
      series={Panor. Synth\`eses},
      volume={42/43},
      publisher={Soc. Math. France, Paris},
   },
   date={2014},
   pages={93--444},
}
\bib{Cowling79}{article}{
    AUTHOR = {Cowling, Michael},
     TITLE = {Sur les coefficients des repr\'{e}sentations unitaires des groupes
              de {L}ie simples},
 BOOKTITLE = {Analyse harmonique sur les groupes de {L}ie ({S}\'{e}m.,
              {N}ancy-{S}trasbourg 1976--1978), {II}},
    SERIES = {Lecture Notes in Math.},
    VOLUME = {739},
     PAGES = {132--178},
 PUBLISHER = {Springer, Berlin},
}
\bib{stability}{article}{
    AUTHOR = {De Chiffre, Marcus},
    AUTHOR = {Glebsky, Lev},
    AUTHOR = {Lubotzky, Alexander},
    AUTHOR = {Thom, Andreas},
     TITLE = {Stability, cohomology vanishing, and nonapproximable groups},
   JOURNAL = {Forum Math. Sigma},
  FJOURNAL = {Forum of Mathematics. Sigma},
    VOLUME = {8},
      YEAR = {2020},
     PAGES = {Paper No. e18, 37},
}

\bib{Clozel}{article}{
    AUTHOR = {Clozel, Laurent},
     TITLE = {D\'{e}monstration de la conjecture {$\tau$}},
   JOURNAL = {Invent. Math.},
  FJOURNAL = {Inventiones Mathematicae},
    VOLUME = {151},
      YEAR = {2003},
    NUMBER = {2},
     PAGES = {297--328},
}

\bib{Clozel-Ullmo}{article}{
    AUTHOR = {Clozel, Laurent},
    AUTHOR = {Ullmo, Emmanuel},
     TITLE = {\'{E}quidistribution des points de {H}ecke},
 BOOKTITLE = {Contributions to automorphic forms, geometry, and number
              theory},
     PAGES = {193--254},
 PUBLISHER = {Johns Hopkins Univ. Press, Baltimore, MD},
      YEAR = {2004},
}

\bib{connes+moscovici}{article}{
   author={Connes, Alain},
   author={Moscovici, Henri},
   title={Cyclic cohomology, the Novikov conjecture and hyperbolic groups},
   journal={Topology},
   volume={29},
   date={1990},
   number={3},
   pages={345--388},
}

\bib{Conrad}{article}{
    AUTHOR = {Conrad, Brian},
     TITLE = {Weil and {G}rothendieck approaches to adelic points},
   JOURNAL = {Enseign. Math. (2)},
  FJOURNAL = {L'Enseignement Math\'{e}matique. Revue Internationale. 2e
              S\'{e}rie},
    VOLUME = {58},
      YEAR = {2012},
    NUMBER = {1-2},
     PAGES = {61--97},
}
\bib{dugundji}{article}{
   author={Dugundji, J.},
   title={An extension of Tietze's theorem},
   journal={Pacific J. Math.},
   volume={1},
   date={1951},
   pages={353--367},
}
	
\bib{dymara+janus}{article}{
    AUTHOR = {Dymara, Jan},
    AUTHOR = {Januszkiewicz, Tadeusz},
     TITLE = {Cohomology of buildings and their automorphism groups},
   JOURNAL = {Invent. Math.},
  FJOURNAL = {Inventiones Mathematicae},
    VOLUME = {150},
      YEAR = {2002},
    NUMBER = {3},
     PAGES = {579--627},
}

\bib{Enright}{article}{
    AUTHOR = {Enright, Thomas J.},
     TITLE = {Relative {L}ie algebra cohomology and unitary representations
              of complex {L}ie groups},
   JOURNAL = {Duke Math. J.},
    VOLUME = {46},
      YEAR = {1979},
    NUMBER = {3},
     PAGES = {513--525},
}
\bib{epstein}{book}{
   author={Epstein, David B. A.},
   author={Cannon, James W.},
   author={Holt, Derek F.},
   author={Levy, Silvio V. F.},
   author={Paterson, Michael S.},
   author={Thurston, William P.},
   title={Word processing in groups},
   publisher={Jones and Bartlett Publishers, Boston, MA},
   date={1992},
   pages={xii+330},
}

\bib{FTN}{book}{
    AUTHOR = {Fig\`a-Talamanca, Alessandro},
    AUTHOR = {Nebbia, Claudio},
     TITLE = {Harmonic analysis and representation theory for groups acting
              on homogeneous trees},
    SERIES = {London Mathematical Society Lecture Note Series},
    VOLUME = {162},
 PUBLISHER = {Cambridge University Press, Cambridge},
      YEAR = {1991},
     PAGES = {x+151},
}
\bib{fisher+margulis}{article}{
   author={Fisher, David},
   author={Margulis, Gregory},
   title={Almost isometric actions, property (T), and local rigidity},
   journal={Invent. Math.},
   volume={162},
   date={2005},
   number={1},
   pages={19--80},
}
\bib{Folland}{book}{
    AUTHOR = {Folland, Gerald B.},
     TITLE = {A course in abstract harmonic analysis},
    SERIES = {Textbooks in Mathematics},
   EDITION = {Second},
 PUBLISHER = {CRC Press, Boca Raton, FL},
      YEAR = {2016},
     PAGES = {xiii+305 pp.+loose errata},
}

\bib{FournierFacio}{article}{
  author = {Fournier-Facio, Francesco},
  title  = {Stability, approximable quotients, and higher property {(T)}},
  year   = {2025},
  eprint = {arXiv:2512.09180},
}

\bib{FournierFacio-Sauer}{article}{
  author = {Fournier-Facio, Francesco},
  author = {Sauer, Roman},
  title  = {Kazhdan groups of dimension $16$ with an uncountable range of second $\ell^2$-Betti numbers},
  year   = {2026},
  eprint = {arXiv:2601.00074},
}


\bib{Franke}{article}{
    AUTHOR = {Franke, Jens},
     TITLE = {Harmonic analysis in weighted {$L_2$}-spaces},
   JOURNAL = {Ann. Sci. \'{E}cole Norm. Sup. (4)},
  FJOURNAL = {Annales Scientifiques de l'\'{E}cole Normale Sup\'{e}rieure. Quatri\`eme
              S\'{e}rie},
    VOLUME = {31},
      YEAR = {1998},
    NUMBER = {2},
     PAGES = {181--279},
}

\bib{Garland}{article}{
    AUTHOR = {Garland, Howard},
     TITLE = {{$p$}-adic curvature and the cohomology of discrete subgroups
              of {$p$}-adic groups},
   JOURNAL = {Ann. of Math. (2)},
    VOLUME = {97},
      YEAR = {1973},
     PAGES = {375--423},
}

\bib{Gelbart-Jaquet}{article}{
    AUTHOR = {Gelbart, Stephen},
    AUTHOR = {Jacquet, Herv\'{e}},
     TITLE = {A relation between automorphic representations of {${\rm
              GL}(2)$} and {${\rm GL}(3)$}},
   JOURNAL = {Ann. Sci. \'{E}cole Norm. Sup. (4)},
    VOLUME = {11},
      YEAR = {1978},
    NUMBER = {4},
     PAGES = {471--542},
}

\bib{grobner}{article}{
   author={Grobner, Harald},
   title={Residues of Eisenstein series and the automorphic cohomology of
   reductive groups},
   journal={Compos. Math.},
   volume={149},
   date={2013},
   number={7},
   pages={1061--1090},
}
\bib{Guichardet}{book}{
    AUTHOR = {Guichardet, A.},
     TITLE = {Cohomologie des groupes topologiques et des alg\`ebres de {L}ie},
    SERIES = {Textes Math\'{e}matiques [Mathematical Texts]},
    VOLUME = {2},
 PUBLISHER = {CEDIC, Paris},
      YEAR = {1980},
     PAGES = {xvi+394},
}

\bib{harishchandra}{article}{
    AUTHOR = {Harish-Chandra},
     TITLE = {Representations of a semisimple {L}ie group on a {B}anach space. {I}},
   JOURNAL = {Trans. Amer. Math. Soc.},
    VOLUME = {75},
      YEAR = {1953},
     PAGES = {185--243},
}

\bib{ji}{article}{
   author={Ji, Ronghui},
   title={Smooth dense subalgebras of reduced group $C^*$-algebras, Schwartz
   cohomology of groups, and cyclic cohomology},
   journal={J. Funct. Anal.},
   volume={107},
   date={1992},
   number={1},
   pages={1--33},
}
\bib{ji+ramsey}{article}{
   author={Ji, Ronghui},
   author={Ramsey, Bobby},
   title={The isocohomological property, higher Dehn functions, and
   relatively hyperbolic groups},
   journal={Adv. Math.},
   volume={222},
   date={2009},
   number={1},
   pages={255--280},
}

\bib{Kaluba-Mizerka-Nowak}{article}{
  author = {Kaluba, Marek},
  author = {Mizerka, Piotr},
  author = {Nowak, Piotr W.},
  title = {Spectral gap for the cohomological Laplacian of $\operatorname{SL}_3(\mathbb{Z})$},
  year = {2022},
  eprint={arXiv:2207.02783},
}
	

\bib{Karpushev-Vershik}{article}{
    AUTHOR = {Karpushev, S. I.},
    AUTHOR = {Vershik, A. M.},
     TITLE = {Cohomology of groups in unitary representations, neighborhood
              of the identity and conditionally positive definite functions},
   JOURNAL = {Mat. Sb. (N.S.)},
  FJOURNAL = {Matematicheski\u{\i} Sbornik. Novaya Seriya},
    VOLUME = {119(161)},
      YEAR = {1982},
    NUMBER = {4},
     PAGES = {521--533, 590},
}

\bib{Keys}{article}{
    AUTHOR = {Keys, David},
     TITLE = {Principal series representations of special unitary groups
              over local fields},
   JOURNAL = {Compositio Math.},
  FJOURNAL = {Compositio Mathematica},
    VOLUME = {51},
      YEAR = {1984},
    NUMBER = {1},
     PAGES = {115--130},
}

\bib{KM99}{article}{
    AUTHOR = {Kleinbock, D. Y.},
    AUTHOR = {Margulis, G. A.},
     TITLE = {Logarithm laws for flows on homogeneous spaces},
   JOURNAL = {Invent. Math.},
  FJOURNAL = {Inventiones Mathematicae},
    VOLUME = {138},
      YEAR = {1999},
    NUMBER = {3},
     PAGES = {451--494},
}

\bib{Kumaresan}{article}{
    AUTHOR = {Kumaresan, S.},
     TITLE = {On the canonical {$k$}-types in the irreducible unitary
              {$g$}-modules with nonzero relative cohomology},
   JOURNAL = {Invent. Math.},
  FJOURNAL = {Inventiones Mathematicae},
    VOLUME = {59},
      YEAR = {1980},
    NUMBER = {1},
     PAGES = {1--11},
}

\bib{Tim+Timo}{article}{
    AUTHOR = {de Laat, Tim},
    AUTHOR = {Siebenand, Timo},
     TITLE = {Exotic group {$C^*$}-algebras of simple {L}ie groups with real
              rank one},
   JOURNAL = {Ann. Inst. Fourier (Grenoble)},
  FJOURNAL = {Universit\'{e} de Grenoble. Annales de l'Institut Fourier},
    VOLUME = {71},
      YEAR = {2021},
    NUMBER = {5},
     PAGES = {2117--2136},
}

\bib{leuzinger+young}{article}{
    AUTHOR = {Leuzinger, Enrico},
    AUTHOR = {Young, Robert},
     TITLE = {Filling functions of arithmetic groups},
   JOURNAL = {Ann. of Math. (2)},
  FJOURNAL = {Annals of Mathematics. Second Series},
    VOLUME = {193},
      YEAR = {2021},
    NUMBER = {3},
     PAGES = {733--792},
}

\bib{Li-Sun}{article}{
    AUTHOR = {Li, Jian-Shu},
    AUTHOR = {Sun, Binyong},
     TITLE = {Low degree cohomologies of congruence groups},
   JOURNAL = {Sci. China Math.},
  FJOURNAL = {Science China. Mathematics},
    VOLUME = {62},
      YEAR = {2019},
    NUMBER = {11},
     PAGES = {2287--2308},
}
\bib{neumann+paucar}{article}{
  author = {L\'opez Neumann, Antonio},
  author = {Paucar, Juan},
  title  = {Quantitative polynomial cohomology and applications to {$\mathrm{L}^p$}-measure equivalence},
  year   = {2025},
  eprint = {arXiv:2512.18463},
}


\bib{lubotzky+mozes+raghunathan}{article}{
   author={Lubotzky, Alexander},
   author={Mozes, Shahar},
   author={Raghunathan, M. S.},
   title={The word and Riemannian metrics on lattices of semisimple groups},
   journal={Inst. Hautes \'{E}tudes Sci. Publ. Math.},
   number={91},
   date={2000},
   pages={5--53 (2001)},
}

\bib{Margulis}{book}{
    AUTHOR = {Margulis, G. A.},
     TITLE = {Discrete subgroups of semisimple {L}ie groups},
    SERIES = {Ergebnisse der Mathematik und ihrer Grenzgebiete (3) [Results
              in Mathematics and Related Areas (3)]},
    VOLUME = {17},
 PUBLISHER = {Springer-Verlag, Berlin},
      YEAR = {1991},
     PAGES = {x+388},
}
\bib{mimura+toda}{book}{
   author={Mimura, Mamoru},
   author={Toda, Hirosi},
   title={Topology of Lie groups. I, II},
   series={Translations of Mathematical Monographs},
   volume={91},
   publisher={American Mathematical Society, Providence, RI},
   date={1991},
   pages={iv+451},
}
\bib{monod}{article}{
   author={Monod, Nicolas},
   title={On the bounded cohomology of semi-simple groups, $S$-arithmetic
   groups and products},
   journal={J. Reine Angew. Math.},
   volume={640},
   date={2010},
   pages={167--202},
}

\bib{Moore}{article}{
    AUTHOR = {Moore, Calvin C.},
     TITLE = {Decomposition of unitary representations defined by discrete
              subgroups of nilpotent groups},
   JOURNAL = {Ann. of Math. (2)},
  FJOURNAL = {Annals of Mathematics. Second Series},
    VOLUME = {82},
      YEAR = {1965},
     PAGES = {146--182},
}

\bib{meyer}{article}{
   author={Meyer, Ralf},
   title={Combable groups have group cohomology of polynomial growth},
   journal={Q. J. Math.},
   volume={57},
   date={2006},
   number={2},
   pages={241--261},
}

\bib{meyer-derived}{article}{ 
author={Meyer, Ralf}, title={Embeddings of derived categories of bornological modules}, date={2004}, eprint={arXiv:math/0410596}, }

\bib{mok}{article}{
   author={Mok, Ngaiming},
   title={Harmonic forms with values in locally constant Hilbert bundles},
   booktitle={Proceedings of the Conference in Honor of Jean-Pierre Kahane
   (Orsay, 1993)},
   journal={J. Fourier Anal. Appl.},
   date={1995},
   pages={433--453},
}
\bib{ogle}{article}{
   author={Ogle, C.},
   title={Polynomially bounded cohomology and discrete groups},
   journal={J. Pure Appl. Algebra},
   volume={195},
   date={2005},
   number={2},
   pages={173--209},
}
\bib{Oh}{article}{
    AUTHOR = {Oh, Hee},
     TITLE = {Uniform pointwise bounds for matrix coefficients of unitary
              representations and applications to {K}azhdan constants},
   JOURNAL = {Duke Math. J.},
  FJOURNAL = {Duke Mathematical Journal},
    VOLUME = {113},
      YEAR = {2002},
    NUMBER = {1},
     PAGES = {133--192},
}
\bib{ozawa}{article}{
   author={Ozawa, Narutaka},
   title={A functional analysis proof of Gromov's polynomial growth theorem},
   language={English, with English and French summaries},
   journal={Ann. Sci. \'{E}c. Norm. Sup\'{e}r. (4)},
   volume={51},
   date={2018},
   number={3},
   pages={549--556},
}

		
\bib{Samei-Wiersma}{article}{ 
author={Samei, Ebrahim}, 
author={Wiersma, Matthew}, 
title={Exotic C*-algebras of geometric groups}, 
year={2018}, 
eprint={arXiv:1809.07007}}

\bib{Shahidi}{article}{
    AUTHOR = {Shahidi, Freydoon},
     TITLE = {A proof of {L}anglands' conjecture on {P}lancherel measures;
              complementary series for {$p$}-adic groups},
   JOURNAL = {Ann. of Math. (2)},
  FJOURNAL = {Annals of Mathematics. Second Series},
    VOLUME = {132},
      YEAR = {1990},
    NUMBER = {2},
     PAGES = {273--330},
}	

\bib{Shalom}{article}{
    AUTHOR = {Shalom, Yehuda},
     TITLE = {Rigidity of commensurators and irreducible lattices},
   JOURNAL = {Invent. Math.},
  FJOURNAL = {Inventiones Mathematicae},
    VOLUME = {141},
      YEAR = {2000},
    NUMBER = {1},
     PAGES = {1--54},
}

\bib{Tadic}{article}{
    AUTHOR = {Tadi\'{c}, M.},
     TITLE = {Dual spaces of adelic groups},
   JOURNAL = {Glas. Mat. Ser. III},
  FJOURNAL = {Dru\v{s}stvo Matemati\v{c}ara i Fizi\v{c}ara S. R. Hrvatske. Glasnik
              Matemati\v{c}ki. Serija III},
    VOLUME = {19(39)},
      YEAR = {1984},
    NUMBER = {1},
     PAGES = {39--48},
}
\bib{Tits}{article}{
    AUTHOR = {Tits, J.},
     TITLE = {Classification of algebraic semisimple groups},
 BOOKTITLE = {Algebraic {G}roups and {D}iscontinuous {S}ubgroups ({P}roc.
              {S}ympos. {P}ure {M}ath., {B}oulder, {C}olo., 1965)},
     PAGES = {33--62},
 PUBLISHER = {Amer. Math. Soc., Providence, R.I.},
      YEAR = {1966},
}


\bib{Thom-ICM}{article}{
 author={Thom, Andreas},
 book={
 title={Proceedings of the international congress of mathematicians 2018, ICM 2018, Rio de Janeiro, Brazil, August 1--9, 2018. Volume III. Invited lectures},
 publisher={Hackensack, NJ: World Scientific; Rio de Janeiro: Sociedade Brasileira de Matem\'atica (SBM)},
 },
 title={Finitary approximations of groups and their applications},
 pages={1779--1799},
 date={2018},
}

\bib{Veca}{thesis}{
    AUTHOR = {Alessandro Veca},
 TITLE = {The Kunze-Stein Phenomenon},
 school  ={University of New South
Wales, Sydney, Australia},
 YEAR = {2002},
 type    = {PhD Thesis},
}

\bib{Vogan}{article}{
    AUTHOR = {Vogan, Jr., David A.},
     TITLE = {Unitarizability of certain series of representations},
   JOURNAL = {Ann. of Math. (2)},
  FJOURNAL = {Annals of Mathematics. Second Series},
    VOLUME = {120},
      YEAR = {1984},
    NUMBER = {1},
     PAGES = {141--187},
}

\bib{Vogan-survey}{article}{
    AUTHOR = {Vogan, Jr., David A.},
     TITLE = {Cohomology and group representations},
 BOOKTITLE = {Representation theory and automorphic forms ({E}dinburgh,
              1996)},
    SERIES = {Proc. Sympos. Pure Math.},
    VOLUME = {61},
     PAGES = {219--243},
 PUBLISHER = {Amer. Math. Soc., Providence, RI},
      YEAR = {1997},
}

\bib{Vogan-Zuckerman}{article}{
    AUTHOR = {Vogan, Jr., David A.},
    AUTHOR = {Zuckerman, Gregg J.},
     TITLE = {Unitary representations with nonzero cohomology},
   JOURNAL = {Compositio Math.},
  FJOURNAL = {Compositio Mathematica},
    VOLUME = {53},
      YEAR = {1984},
    NUMBER = {1},
     PAGES = {51–90},
}
\bib{weibel}{book}{
   author={Weibel, Charles A.},
   title={An introduction to homological algebra},
   series={Cambridge Studies in Advanced Mathematics},
   volume={38},
   publisher={Cambridge University Press, Cambridge},
   date={1994},
   pages={xiv+450},
}
\bib{wenger}{article}{
   author={Wenger, Stefan},
   title={A short proof of Gromov’s filling inequality},
   journal={Proc. Amer. Math. Soc.},
   volume={136},
   date={2008},
   number={8},
   pages={2937–2941},
}
\bib{zimmer}{book}{
   author={Zimmer, Robert J.},
   title={Ergodic theory and semisimple groups},
   series={Monographs in Mathematics},
   volume={81},
   publisher={Birkh\"{a}user Verlag, Basel},
   date={1984},
   pages={x+209},
}
\bib{Zuckerman}{article}{
    AUTHOR = {Zuckerman, Gregg J.},
     TITLE = {Continuous cohomology and unitary representations of real reductive groups},
   JOURNAL = {Ann. of Math. (2)},
  FJOURNAL = {Annals of Mathematics. Second Series},
    VOLUME = {107},
      YEAR = {1978},
    NUMBER = {3},
     PAGES = {495–516},
}
	
\end{biblist}
\end{bibdiv}
\end{document}